\documentclass[11pt]{amsart}

\usepackage{epigamath}

\usepackage[english]{babel}

\numberwithin{equation}{section}

\usepackage[shortlabels]{enumitem}

\usepackage{amsthm} 
\usepackage{mathtools}

\usepackage{tikz}
\usepackage{tikz-cd}
\usepackage[all]{xy}
\usepackage{tikz}
\usetikzlibrary{arrows}
\usetikzlibrary{cd}
\usepackage{makecell}
\usepackage{leftindex}
\usepackage{extarrows}
\usepackage{etoolbox}

\newcommand{\Z}{\mathbb Z}   
\newcommand{\Q}{\mathbb Q}   
\newcommand{\F}{\mathbb F}   
\newcommand{\A}{\mathbb A}   
\newcommand{\G}{\mathbb G}   
\newcommand{\eey}{\mathscr{Y}}
\newcommand{\eem}{\mathscr{M}}
\newcommand{\eex}{\mathscr{X}}

\newcommand{\eee}{\mathscr{E}}
\renewcommand\l{\ell}

\newcommand{\yzt}{\mathscr{Y}_0(2)}

\newcommand{\mc}{\mathcal}

\newcommand{\id}{\operatorname{id}} 
\newcommand{\im}{\operatorname{im}} 
\newcommand{\PGL}{\operatorname{PGL}}
\newcommand{\GL}{\operatorname{GL}}

\newcommand{\spec}{\operatorname{Spec}} 
\newcommand{\rSpec}{\mathbf{Spec}} 

\newcommand{\sm}{\setminus}
\newcommand{\wt}{\widetilde}
\newcommand{\zmod}[1]{\mathbb Z/#1\mathbb Z}

\newcommand{\dirlim}{\varinjlim\limits}
\newcommand{\etpi}{\pi^{\mathrm{\acute{e}t}}}
\newcommand{\ctsHom}{\hom_{\mathrm{cts}}}
\newcommand{\units}[1]{#1^{\times}}
\newcommand{\punits}[1]{\units{\parens{#1}}}
\newcommand{\inv}[2][1]{#2^{-#1}}
\newcommand{\pinv}[2][1]{\inv[#1]{\parens{#2}}}

\newcommand{\pprimary}[1]{\leftindex_p{#1}}
\newcommand{\twoprimary}[1]{\leftindex_2{#1}}
\newcommand{\push}[1]{#1_*}

\newcommand{\pull}[1]{#1^*}

\newcommand{\define}[1]{\textit{#1}} 
\renewcommand{\bar}{\overline} 
\newcommand{\by}{\times}
\newcommand{\msB}{\mathscr B}
\newcommand{\msA}{\mathscr A}
\newcommand{\msO}{\mathscr O}
\newcommand{\msL}{\mathscr L}
\newcommand{\msM}{\mathscr M}

\newcommand{\mfp}{\mathfrak p}
\newcommand{\ul}{\underline}
\newcommand{\del}{\partial}
\newcommand{\actson}{\curvearrowright}

\newcommand{\opname}{\operatorname}
\DeclareMathOperator{\Ind}{Ind}
\DeclareMathOperator{\Br}{Br}
\DeclareMathOperator{\Tr}{Tr}
\DeclareMathOperator{\Aut}{Aut}
\DeclareMathOperator{\ram}{ram}

\DeclareMathOperator{\Nm}{Nm}

\newcommand{\mapdesc}[5]{
	\begin{matrix}
		\ifblank{#1}{}{\displaystyle#1:}&\displaystyle#2&\longrightarrow&\displaystyle#3\\
		&\displaystyle#4&\longmapsto&\displaystyle #5
	\end{matrix}
}

\newcommand{\set}[1]{\left\{#1\right\}} 

\newcommand{\tand}{\quad\text{and}\quad} 
\newcommand{\twhere}{\text{where }}
\newcommand{\twith}{\text{with }}

\renewcommand\t\text 
\newcommand\ttt\texttt

\newcommand{\parens}[1]{\!\left(#1\right)}
\newcommand{\p}[1]{\!\left(#1\right)} 
\newcommand{\pfrac}[2]{\parens{\frac{#1}{#2}}}
\newcommand{\brackets}[1]{\left\{#1\right\}}
\newcommand{\bracks}[1]{\brackets{#1}}
\renewcommand{\b}{\bracks} 
\newcommand{\sqbracks}[1]{\!\left[#1\right]} 
\newcommand{\sq}{\sqbracks}
\newcommand{\angles}[1]{\left\langle#1\right\rangle}

\newcommand{\dparens}[1]{\!\left(\!\left(#1\right)\!\right)}
\newcommand\dps\dparens 

\newcommand{\twocasesit}[3]{
	\begin{cases}
		\hfill\displaystyle #1\hfill&\text{if\, }\displaystyle #2\\
		\hfill\displaystyle #3\hfill&\text{otherwise.}
	\end{cases}
}

\newcommand{\Twocases}[4]{
	\begin{cases}
		\hfill\displaystyle #1\hfill&\text{if }\displaystyle #2\\
		\hfill\displaystyle #3\hfill&\text{if }\displaystyle #4
	\end{cases}
}
\newcommand{\Twocasesit}[4]{
	\begin{cases}
		\hfill\displaystyle #1\hfill&\text{if\, }\displaystyle #2\\
		\hfill\displaystyle #3\hfill&\text{if\, }\displaystyle #4
	\end{cases}
}
\newcommand{\Threecases}[6]{
	\begin{cases}
		\hfill\displaystyle #1\hfill&\text{if }\displaystyle #2\\
		\hfill\displaystyle #3\hfill&\text{if }\displaystyle #4\\
		\hfill\displaystyle #5\hfill&\text{if }\displaystyle #6
	\end{cases}
}
\newcommand{\Threecasesit}[6]{
	\begin{cases}
		\hfill\displaystyle #1\hfill&\text{if\, }\displaystyle #2\\
		\hfill\displaystyle #3\hfill&\text{if\, }\displaystyle #4\\
		\hfill\displaystyle #5\hfill&\text{if\, }\displaystyle #6
	\end{cases}
}

\newcommand{\commsquare}[8]{
	\begin{tikzcd}[ampersand replacement=\&]
	{\displaystyle #1}\ar[r, "{#2}"]\ar[d, "{#4}" left]\&{\displaystyle #3}\ar[d, "{#5}" right]\\
	{\displaystyle #6}\ar[r, "{#7}" above]\&{\displaystyle #8}
	\end{tikzcd}
}

\newcommand{\compdiag}[6]{ 
    \begin{tikzcd}[ampersand replacement=\&]
        {\displaystyle #1}\ar[r, "{#2}" below]\ar[rr, bend left, "{#6}" above]\&{\displaystyle #3}\ar[r, "{#4}" below]\&{\displaystyle #5}
    \end{tikzcd}
}

\def\lowsimiso{\vbox to 0pt{\vss\hbox{$\scriptstyle\sim$}\vskip-2.4pt}}
\def\lowsimisoo{\vbox to 0pt{\vss\hbox{$\sim$}\vskip-1.6pt}}

\newcommand{\mapstoo}{\longmapsto}

\newcommand{\xto}{\xrightarrow}

\newcommand{\too}{\longrightarrow}
\newcommand{\xtoo}{\xlongrightarrow}
\newcommand{\iso}{\xto\lowsimiso}
\newcommand{\isoo}{\xtoo\lowsimisoo}

\newcommand{\into}{\hookrightarrow}

\newcommand{\onto}{\twoheadrightarrow}

\newcommand{\squigto}{\rightsquigarrow}

\makeatletter
\newcommand*{\da@rightarrow}{\mathchar"0\hexnumber@\symAMSa 4B }
\newcommand*{\da@leftarrow}{\mathchar"0\hexnumber@\symAMSa 4C }
\newcommand*{\xdashrightarrow}[2][]{%
  \mathrel{%
    \mathpalette{\da@xarrow{#1}{#2}{}\da@rightarrow{\,}{}}{}%
  }%
}
\newcommand{\xdashleftarrow}[2][]{%
  \mathrel{%
    \mathpalette{\da@xarrow{#1}{#2}\da@leftarrow{}{}{\,}}{}%
  }%
}
\newcommand*{\da@xarrow}[7]{%
  \sbox0{$\ifx#7\scriptstyle\scriptscriptstyle\else\scriptstyle\fi#5#1#6\m@th$}%
  \sbox2{$\ifx#7\scriptstyle\scriptscriptstyle\else\scriptstyle\fi#5#2#6\m@th$}%
  \sbox4{$#7\dabar@\m@th$}%
  \dimen@=\wd0 %
  \ifdim\wd2 >\dimen@
    \dimen@=\wd2 %
  \fi
  \count@=2 %
  \def\da@bars{\dabar@\dabar@}%
  \@whiledim\count@\wd4<\dimen@\do{%
    \advance\count@\@ne
    \expandafter\def\expandafter\da@bars\expandafter{%
      \da@bars
      \dabar@ 
    }%
  }%
  \mathrel{#3}%
  \mathrel{%
    \mathop{\da@bars}\limits
    \ifx\\#1\\%
    \else
      _{\copy0}%
    \fi
    \ifx\\#2\\%
    \else
      ^{\copy2}%
    \fi
  }%
  \mathrel{#4}%
}
\makeatother
\newcommand\xdashto\xdashrightarrow
\newcommand\xdashfrom\xdashleftarrow

\makeatletter
\providecommand{\leftsquigarrow}{%
  \mathrel{\mathpalette\reflect@squig\relax}%
}
\newcommand{\reflect@squig}[2]{%
  \reflectbox{$\m@th#1\rightsquigarrow$}%
}
\makeatother

\renewcommand{\epsilon}{\varepsilon}
\renewcommand{\phi}{\varphi}

\renewcommand{\geq}{\geqslant}

\renewcommand{\id}{\operatorname{id}}
\renewcommand{\hom}{\operatorname{Hom}}

\usetikzlibrary{decorations.markings,patterns,hobby,arrows.meta}
\tikzset{>=stealth}

\allowdisplaybreaks

\DeclareMathOperator{\Pic}{Pic} 
\DeclareMathOperator{\coker}{coker} 
\DeclareMathOperator{\Spec}{Spec} 
\DeclareMathOperator{\Gal}{Gal}
 
\DeclareMathOperator{\ch}{char} 
 
\DeclareMathOperator{\Char}{char}
\DeclareMathOperator{\pl}{pl}
\DeclareMathOperator{\End}{End}

\newcommand{\az}{\A^1_k\backslash\{0\}}

\newcommand{\tors}{\mathrm{tors}}
\newcommand{\op}{\mathrm{op}}
\newcommand{\fppf}{\mathrm{fppf}}

\newtheorem{theorem}{Theorem}[section]
\newtheorem{lemma}[theorem]{Lemma}
\newtheorem{corollary}[theorem]{Corollary}
\newtheorem{proposition}[theorem]{Proposition}

\theoremstyle{definition}
\newtheorem{definition}[theorem]{Definition}
\newtheorem{notation}[theorem]{Notation}
\newtheorem{assumption}[theorem]{Assumption}
\newtheorem{recall}[theorem]{Recall}
\newtheorem{setup}[theorem]{Setup}

\theoremstyle{remark}
\newtheorem{remark}[theorem]{Remark}
\newtheorem{fact}[theorem]{Fact}

\AtBeginDocument{
  \DeclareSymbolFont{AMSb}{U}{msb}{m}{n}
  \DeclareSymbolFontAlphabet{\mathbb}{AMSb}}

\DeclareMathAlphabet{\mathbx}{U}{BOONDOX-ds}{m}{n}
\SetMathAlphabet{\mathbx}{bold}{U}{BOONDOX-ds}{b}{n}
\DeclareMathAlphabet{\mathbbx} {U}{BOONDOX-ds}{b}{n}

\SetMathAlphabet{\mathcal}{bold}{U}{dutchcal}{b}{n}
\DeclareMathAlphabet{\mathbcal}{U}{dutchcal}{b}{n} 

\setlist[enumerate,1]{label={\rm(\arabic*)}, ref={\rm\arabic*}}

\newcommand{\addappendix}{%
  \phantomsection 
  \setcounter{section}{0}
  \section*{\appendixname}
  \addcontentsline{toc}{section}{\appendixname}
  \setcounter{section}{1}
  \setcounter{theorem}{0}
}

\newcommand{\supth}[1]{\ensuremath{#1^{\mathrm{th}}}}
\newcommand{\supst}[1]{\ensuremath{#1^{\mathrm{st}}}}

\newcommand{\intoo}{\lhook\joinrel\longrightarrow}

\EpigaVolumeYear{10}{2026} \EpigaArticleNr{12} \ReceivedOn{July 11, 2024}
\InFinalFormOn{December 11, 2025}
\AcceptedOn{February 7, 2026}

\title{The Brauer group of $\boldsymbol{\eey_0(2)}$}
\titlemark{The Brauer group of $\boldsymbol{\eey_0(2)}$}

\author{Niven Achenjang}
\address{Department of Mathematics, Harvard University, Cambridge, MA 02138, USA}
\email{achenjang@math.harvard.edu}

\author{Deewang Bhamidipati}
\address{Department of Mathematics and Statistics, Carleton College, Northfield, MN 55057, USA}
\email{bdeewang@carleton.edu}

\author{Aashraya Jha}
\address{Department of Mathematics and Statistics, Boston University, Boston, MA 02215, USA}
\email{aashjha@bu.edu}

\author{Caleb Ji}
\address{Department of Mathematics, University of New South Wales, Sydney, NSW 2032, Australia}
\email{caleb.ji@unsw.edu.au}

\author{Rose Lopez}
\address{Department of Mathematics, University of California Berkeley, Berkeley, CA 94720, USA}
\email{roselopez@berkeley.edu}

\authormark{N.~Achenjang, D.~Bhamidipati, A.~Jha, C.~Ji, and R.~Lopez}

\AbstractInEnglish{We determine the Brauer group of the Deligne--Mumford stack $\eey_0(2)$, the moduli space of elliptic curves with a marked $2$-torsion subgroup, over bases of arithmetic interest. Antieau and Meier determined the Brauer group for $\eem_{1,1}$, the moduli stack of elliptic curves, by exploiting the fact that it is covered by the Legendre family and using the Hochschild--Serre spectral sequence. Over an algebraically closed field, Shin used the coarse space map to determine the Brauer group of $\eem_{1,1}$. We combine techniques from both to determine the Brauer group of $\eey_0(2)$.}

\MSCclass{14F22}

\KeyWords{Brauer groups, Deligne--Mumford stacks, arithmetic geometry, moduli stack of elliptic curves, \'etale cohomology, torsion subgroups}

\acknowledgement{This material is based upon work supported by the National Science Foundation (NSF) under Grant Number 1916439 for the 2023 summer AMS MRC. In addition, N.\,A.\ was partially supported by NSF Grant Number DMS-2101040. A.\,J.\ was supported by the Simons Foundation as part of the Simons Collaboration on Arithmetic Geometry, Number Theory, and Computation \#550023
and NSF grant DMS-1945452. C.\,J.\ was supported by NSF grant DGE-2036197. R.\,L.\ was partially supported by NSF grant DMS-1646385.}

\begin{document}



\maketitle

\begin{prelims}

\DisplayAbstractInEnglish

\bigskip

\DisplayKeyWords

\medskip

\DisplayMSCclass

\end{prelims}


\newpage

\setcounter{tocdepth}{1}

\tableofcontents


\section{Introduction}
The Brauer group of a field $k$ classifies the central simple algebras over $k$.  In \cite{Gr1,Gr2,Gr3}, Grothendieck generalized the definition of Brauer groups to schemes and established them as important cohomological invariants in algebraic geometry. The study of cohomological invariants of schemes has since been extended to those of stacks, beginning with Mumford's calculation $\Pic((\eem_{1,1})_k)\cong \Z/12$, where $k$ is a field with characteristic not 2 or 3; see \cite{mumford-picard}.  This calculation has been extended by Fulton and Olsson \cite{fulton-olsson} to the case of an arbitrary reduced base, while Niles \cite{Niles} performed a similar calculation for the covers $\eey_0(2)$ and $\eey_0(3)$.  More recently, there has been some interest in computing Brauer groups of algebraic stacks.  For example, Lieblich \cite{lieblich2011period} computed the Brauer group of $\Br(B\mu_n)$ over a field and applied it to the period-index problem, and Santens \cite{santens2023brauermanin} computed the Brauer group of certain stacky curves in his work on the Brauer--Manin obstruction for them.

A fundamental Deligne--Mumford stack is the classifying space of elliptic curves. In their paper \textit{The Brauer group of the moduli stack of elliptic curves}, \cite{AM}, Antieau and Meier compute the Brauer group of $\eem_{1,1}$, the moduli stack of elliptic curves, over various bases of arithmetic interest. Their main result is that over $\Spec(\Z)$, we have $\Br(\eem_{1,1})=0$. In addition to this, they were also able to compute the Brauer group of $\eem_{1,1}$ over $\Q$, over finite fields of characteristic not $2$, and over algebraically closed fields of characteristic not $2$. These latter computations were supplemented by Shin \cite{Shin}, who computed $\Br(\eem_{1,1,k})$ whenever $k$ is a finite or algebraically closed field of characteristic $2$, using the coarse space map. The notes  \cite{meier-cs} further explore using the coarse space map for computing Brauer groups. The approach initiated in Meier's notes is developed further in a follow-up to this paper \cite{achenjang-brauer} by the first author. Additionally, in \cite{dilorenzo2022cohomological}, the authors compute $\Br(\eem_{1,1,k})$ over any field $k$. We are not aware of any other papers computing Brauer groups of stacky modular curves.

In our paper, we study the Brauer group of $\eey_0(2)$, the moduli stack (over $\Z[\frac{1}{2}]$) of elliptic curves equipped with a cyclic subgroup of order $2$. 
As we define and recall in Section~\ref{sec: the brauer group}, this is equivalent to computing a closely related group known as the cohomological Brauer group.  
By combining ideas of both \cite{AM} and \cite{Shin}, we are able to compute this group as $\eey_0(2)$ is taken over various arithmetically interesting bases.

\begin{theorem}\label{main}\leavevmode
  \begin{enumerate}
        \item\label{main-1} $\Br(\eey_0(2))= \leftindex_2{\Br(\eey_0(2))}\cong\Q_2/\Z_2\oplus(\zmod2)^{\oplus4}\oplus\zmod4 .$

        This follows from Proposition~\ref{prop:p-tors-over-Z[1/2]-i} and Corollary~\ref{cor:2-tors-over-Z[1/2]}.
        \item\label{main-2} Let $\mathbf{P}$ denote the set of all primes.  Then \[\Br(\eey_0(2)_\Q)=\Br(\Q)\oplus H^1(\Q,\Q/\Z)\oplus \Z/2\Z ~~\oplus \bigoplus_{\substack{p\in \mathbf{P}\cup \{-1\}, \\ p\equiv 3\pmod 4}}\Z/2 ~~\oplus \bigoplus_{\substack{p\in \mathbf{P},\\ p\not\equiv 3\pmod 4}}\Z/4.\]

        This follows from Corollary~\ref{cor: p-tors-over-Q} and Proposition~\ref{prop: 2-torsion of Y02}.
        \item\label{main-3} $\Br(\eey_0(2)_{\bar k})=\zmod2$ if\, $\bar k$ is algebraically closed of characteristic not $2$.
          
        This is Theorem~\ref{thm:alg-closed-comp}.
        \item\label{main-4} $\Br(\eey_0(2)_k)=\Br(k)\oplus H^1(k,\Q/\Z)\oplus H^1(k,\zmod4)\oplus\zmod2$ if $k$ is any perfect field of characteristic not $2$.

        This is a special case of Theorem~\ref{thm:Br-over-field}.
    \end{enumerate}
\end{theorem}
\begin{remark}
    In \cite[Above Theorem 1.1]{AM}, the authors remark that their descriptions of $\Br(\eey(1)_S)$ (for $S=\Q,\Z[1/2],\Z$, or a finite or algebraically closed field of characteristic not $2$) vary greatly depending on the base, and so express skepticism at the possibility of writing a uniform description for $\Br(\eey(1)_S)$. However, in light of the computation of \cite{dilorenzo2022cohomological} that $\Br(\eey(1)_k)\cong\Br(\A^1_k)\oplus H^1(k,\zmod{12})$ for any field $k$ of characteristic not $2$, one can retroactively check that \cite{AM} similarly proved that $\Br(\eey(1)_S)\cong\Br(S)\oplus H^1(S,\zmod{12})$ for all bases $S$ appearing in their Theorem 1.1. 
    
    Since our approach in this paper is largely based off of \cite{AM}, our main result, Theorem~\ref{main}, also initially appears to show that $\Br(\eey_0(2)_S)$ depends subtly on the base $S$. However, given the statement of this theorem, one can check that it says
    \[\Br(\eey_0(2)_S)\cong\Br(S)\oplus H^1(S,\Q/\Z)\oplus H^1(S,\zmod4)\oplus\zmod2\cong\Br(Y_0(2)_S)\oplus H^1(S,\zmod4)\oplus\zmod2,\]
    where $S$ is any of the bases appearing in Theorem~\ref{main} and $Y_0(2)_S=\A^1_S\sm\{0\}$ is the coarse space of $\eey_0(2)$ (see Corollary~\ref{cor:Y0(2)-cms-general-base}). Note that $\Br(S)\oplus H^1(S,\Q/\Z)\cong\Br(Y_0(2)_S)$ when $S$ is a perfect field (see Remark~\ref{rem:Br-coarse-field}) or when $S=\Z[1/2]$ (see \cite[Corollary 2.6]{AMS}).
\end{remark}
\begin{remark}
    We remark that, as a consequence of Tsen's theorem, the Brauer group of a curve (that is an ordinary scheme) over an algebraically closed field is always trivial. 
    Theorem~\ref{main}\eqref{main-3} shows that this fails for some \textit{tame} stacky curves. The main results of \cite{AM,Shin} show that this can fail for wild stacky curves ($\Br(\eey(1)_{\bar\F_2})\neq0$, for example), but that $\Br(\eey(1)_{\bar k})=0$ in tame characteristics, even though $\eey(1)$ is stacky everywhere. It would be interesting to better understand what aspects of the geometry of a (tame) stacky curve lead to it failing to satisfying Tsen's theorem as well as to understand where Brauer classes of stacky curves over algebraically closed fields, when they exist, come from. In the case of $\Br(\eey_0(2)_{\bar k})$, we have Proposition~\ref{prop:Z/2Z-not-quat} showing that its nontrivial element is not representable by a global quaternion algebra over $\eey_0(2)_{\bar k}$.
\end{remark}

In order to prove Theorem~\ref{main}, the basic idea, following Antieau and Meier, is as follows. We take a series of Galois covers of $\eey_0(2)$ which reduce the computation of its Brauer group to computing the Brauer group of some scheme $X$ and then computing various Hochschild--Serre spectral sequences associated to the covers. Specifically, for computing the Brauer group of $\eem_{1,1}$, Antieau and Meier considered the $S_3$-cover $\eey(2)\to\eem_{1,1}$ and the $C_2$-cover $X:=\A^1\sm\{0,1\}\to\eey(2)$, the latter given by the Legendre family of elliptic curves, which will be described in Section~\ref{section: setup_of_spaces}. By understanding the $S_3$-module (resp.\ $C_2$-module) structure on the $\G_m$-cohomology of $\eey(2)$ (resp.\ $X$), they were able to compute the relevant spectral sequences and so determine $\Br(\eem_{1,1})$. In our work, we replace the $S_3$-cover $\eey(2)\to\eem_{1,1}$ with the $C_2$-cover $\eey(2)\to\eey_0(2)$. The relevant sequence of covers is pictured below:
\[\begin{tikzcd}
    X\ar[d, equals]\ar[r, "C_2"]&\eey(2)\ar[d, equals]\ar[r, "C_2"]\ar[rr, "S_3", bend left]&\eey_0(2)\ar[r]&\eem_{1,1}\\
    \A^1_{\Z[\frac{1}{2}]}\sm\{0,1\}&BC_{2,X}
    .
\end{tikzcd}\]

Consider a base scheme $S$ (\textit{e.g.}~$S=\Z[\frac{1}{2}]$, $\Q$, \textit{etc.}). We compute $\Br(\eey_0(2)_S)=\bigoplus_p\pprimary{\Br(\eey_0(2)_S)}$ one prime at a time. For each prime $p$, the $C_2$-cover $\eey(2)\to\eey_0(2)$ induces a $p$-local Hochschild--Serre spectral sequence
\begin{equation}\label{ss:p-local HS}
    E_2^{i,j}=H^i\p{C_2,H^j\p{\eey(2)_S,\G_m}}_{(p)}\implies H^{i+j}\p{(\eey_0(2)_S,\G_m}_{(p)},
\end{equation}
where $M_{(p)}:=M\otimes_\Z\Z_{(p)}$ for an abelian group $M$.

For a prime $p$, let $\pprimary{M}$ denote the $p$-primary torsion inside an abelian group $M$. Whenever $p\nmid2$, one has $E_2^{i,j}=0$, unless $i=0$. Thus, the spectral sequence immediately degenerates, and computing the $p$-primary torsion in $\Br(\eey_0(2)_S)$ reduces to computing the $p$-primary torsion in $\Br(\eey(2)_S)^{C_2}$. When $p$ is invertible on~$S$, we compute this in Theorem~\ref{thm:p-tors-away-from-p}, obtaining the following. 

\begin{proposition}[First half of Theorem~\ref{thm:p-tors-away-from-p}]
    Let $p\geq 3$, and let $S$ be a regular noetherian scheme on which $2p$ is invertible. Then we have a non-canonically split exact sequence
    $$0\too \pprimary{\Br'(S)}\too \pprimary{\Br'(\eey_0(2))}\too H^1\p{S,\Q_p/\Z_p}\too 0.$$
\end{proposition}

Whenever $p$ is not invertible on $S$, we set $S'=S\by_\Z\Z[1/p]$ and study the $p$-primary torsion $\Br(\eey_0(2)_S)$ by viewing it as a subgroup of $\Br(\eey_0(2)_{S'})$. Following a strategy used by \cite{AM} in the case of $3$-primary torsion, over $S=\Z[\frac{1}{2}]$, we are able to show that for all primes $p\geq 3$, no nonzero class in $\pprimary{\Br(\eey_0(2)}_{\Z[\frac{1}{2p}]})$ extends over all of $\eey_0(2)$; thus, we compute that $\pprimary{\Br(\eey_0(2))}=0$ over $\Z[\frac{1}{2}]$ for all primes $p\geq3$ (see Corollary~\ref{cor:p-tors-over-Z[1/2]-ii}).

For 2-primary torsion, we still use the spectral sequence (\ref{ss:p-local HS}). However, it is no longer given by a single column on the $E_2$-page, so more work is required. In order to compute the differentials in this spectral sequence, as well as to compute the resulting description of $\leftindex_2{\Br}(\eey_0(2)_S)$, we use the following two tricks:
\begin{itemize}
\item The spectral sequence (\ref{ss:p-local HS}) has a natural morphism to the analogous spectral sequence for the $S_3$-cover $\eey(2)\to\eem_{1,1,\Z[\frac{1}{2}]}$ computed in \cite[Section 9]{AM}. By exploiting this morphism, we are able to deduce descriptions of most differentials in (\ref{ss:p-local HS}) from descriptions of the analogous differentials in \cite[Section 9]{AM}.
    \item For what remains, we choose an arbitrary geometric point $\bar s\into S$ and compare the spectral sequence (\ref{ss:p-local HS}) over $S$ to the same spectral sequence over $\bar s$. This allows us to reduce the computation of the remaining differentials over a general base $S$ to the computation over an algebraically closed field. In this case, we use the fact that we can independently compute $\Br(\eey_0(2)_{\bar k})$ via the method of Shin \cite{Shin} (see Section~\ref{sect:coarse-space}) in order to retroactively understand these differentials, first over $\bar s$ and then over $S$.
\end{itemize}

Finally, when computing $\Br(\eey_0(2)_{\bar k})$, we will need to know that $\Pic(\eey_0(2)_{\bar k})\simeq\zmod4$. Whenever $\Char k\neq3$, this was shown by Niles \cite{Niles}. In order to handle the case $\Char k=3$ as well, we prove the following general result in the appendix.

\begin{theorem}
    Let $S$ be a $\Z[\frac{1}{2}]$-scheme, and let $Y_0(2)$ be the coarse space of\, $\eey_0(2)$. Then, $\Pic(\eey_0(2)_S)\cong\zmod4\oplus\Pic(Y_0(2)_S)$. 
\end{theorem}
\begin{remark}
    Recent work \cite{lopez2023picard} of one of the authors gives another approach to computing $\Pic(\eey_0(2)_k)$ for any field $k$ of characteristic not $2$.
\end{remark}

In summary, in logical order, the steps taken to prove Theorem~\ref{main} are: 
\begin{itemize}
    \item Compute $\Pic(\eey_0(2)_k)$ for algebraically closed fields $k$. See the appendix.
    \item Use this to compute $\Br(\eey_0(2)_k)$. See Section~\ref{sect:coarse-space}.
    \item Use that computation to compute the differentials in the spectral sequence (\ref{ss:p-local HS}) over $k$. See Lemma~\ref{lem:d202-diff-comp}.
    \item For more general bases $S$, compute the differentials in the spectral sequence (\ref{ss:p-local HS}) for $p=2$ by comparison with the case of an algebraically closed field and with the work in \cite{AM}. Use this to compute $\twoprimary{\Br}(\eey_0(2)_S)$, at least for $S=\spec(\Q)$ and $S=\spec(\Z[1/2])$. See Section~\ref{sect:2-prim-torsion}, especially Proposition~\ref{prop: 2-torsion of Y02}.
    \item Compute the spectral sequence (\ref{ss:p-local HS}), and so also $\pprimary{\Br}(\eey_0(2)_S)$, for $p\neq2$ assuming $p$ is not invertible in $S$. See Section~\ref{sect:p>=3 torsion, p invertible}.
    \item Compute the image of $\pprimary{\Br}(\eey_0(2)_{\Z[1/2]})\into\pprimary{\Br}(\eey_0(2)_{\Z[1/2p]})$ for each $p\neq2$. See Section~\ref{sect:p-prim-torsion-over-Z1/2}.
    \item Finally, to compute $\Br(\eey_0(2)_k)$ for perfect fields $k$, we use the Hochschild--Serre spectral sequence relative to the Galois cover $k^s/k$. See Section~\ref{sect:Br-field}.
\end{itemize}

\subsection*{Acknowledgements} The authors thank the organizers and leadership in the 2023 AMS Mathematics Research Communities, Explicit Computations with Stacks, for providing an environment encouraging research in which the authors met and formed a collaboration. The authors thank this MRC and its leadership and participants for providing interesting research problems, meaningful guidance, and helpful discussions. We thank Bjorn Poonen, Martin Olsson, John Voight, and Will Sawin in particular for helpful discussions. 
We also thank the anonymous referee for their helpful comments, and in particular for suggesting one of our proofs of Proposition~\ref{prop:section-no-extend}.

\medskip

\section{The Brauer group}
\label{sec: the brauer group}

We first define what we mean by Brauer groups in this paper. 

\begin{definition}\label{defn: C Brauer group}
  If $X$ is a quasi-compact and quasi-separated Deligne--Mumford stack, the \define{cohomological Brauer group} of $X$ is defined to be $\Br'(X)\coloneqq H^2(X,\G_m)_{\tors}$, the torsion subgroup of the \'etale cohomology group $H^2(X,\G_m)$.    
\end{definition}

\begin{definition}\label{defn: Azuyama algebras}
    An \define{Azumaya algebra} over a Deligne--Mumford stack $X$ is a sheaf of quasi-coherent
$\msO_X$\nobreakdash-algebras $\msA$ such that $\msA$ is \'etale-locally on $X$ isomorphic to $M_n(\msO_X)$, the sheaf of $n\times n$ matrices over $\msO_X$, for some  $n\geq 1$.
A pair of Azumaya algebras $\msA$ and $\msB$ are \define{Brauer equivalent} if there are vector bundles $E$ and $F$ on $X$ such that $\msA\otimes_{\msO_X} \End(E)\iso  \msB \otimes_{\msO_X} \End(F).$ 
\end{definition}

\begin{definition}\label{defn: Br}
The \define{Brauer group} $\Br(X)$ of a Deligne--Mumford
stack $X$ is the group of isomorphism classes of Azumaya algebras under tensor product modulo Brauer equivalence with inverses given by $-[\msA]=[\msA^{\op}]$.
\end{definition}

\begin{remark}
  There is a natural inclusion $\alpha_X\colon \Br(X)\to\Br'(X)$ given as follows. To ease the description, assume that $X$ is connected. An Azumaya algebra corresponds to a $\PGL_n$-torsor, and its image under $\alpha_X$ is given by the image of the torsor under the differential $H^1(X,\PGL_n)\to H^2(X, \G_m)$ coming from the sequence $0\to \G_m\to \GL_n \to \PGL_n\to 0$. 
\end{remark}

\begin{lemma}\label{lem:coh-Br=Az-Br}
    Let $S$ be a quasi-compact $\Z[\frac{1}{2}]$-scheme with an ample line bundle. Then, $\Br(\eey_0(2)_S)=\Br'(\eey_0(2)_S)$.
\end{lemma}
\begin{proof}
    For any quasi-compact algebraic stack $X$, we let $\alpha_X\colon\Br(X)\to\Br'(X)$ denote the natural inclusion. Corollary 2.6 of \cite{Shin} states that if $f\colon X\to Y$ is a finitely presented, finite, flat, surjective morphism of algebraic stacks and $\alpha_X$ is an isomorphism, then $\alpha_Y$ is an isomorphism as well. We wish to apply this to the finite \'etale cover $\eey(4)_S\to\eey_0(2)_S$. Note that $\eey(4)_S$ is quasi-projective, and so admits a relatively ample line bundle over $S$. Hence, if $S$ admits an ample line bundle, \cite[\href{https://stacks.math.columbia.edu/tag/0892}{Tag 0892}(1)]{stacks-project} shows that $\eey_0(2)_S$ does as well. Now, \cite[Proposition 2.5(v)]{AM} shows that $\alpha_{\eey(4)_S}$ is an isomorphism.
\end{proof}

In particular, in Theorem~\ref{main}, we compute $\Br(X)$ by computing $\Br'(X)$. For more details about the Brauer group in general, see \cite{Gr1,Gr2,Gr3}.

\medskip

\section{Setup}
\subsection{The spaces $\boldsymbol{\eem_{1,1},\eey_0(2)}$, and $\boldsymbol{\eey(2)}$} \label{section: setup_of_spaces}

In this section, we describe the moduli spaces of elliptic curves $\eem_{1,1}[\frac{1}{2}]$, $\eey(2)$, and $\eey_0(2)$, where the space $\eem_{1,1}[\frac{1}{2}]$ is the moduli space of elliptic curves over $\Z[\frac{1}{2}]$, the space $\eey(2)$ is the moduli space of elliptic curves with full level 2 structure, and the space $\eey_0(2)$ is the moduli space of elliptic curves equipped with a cyclic subgroup of order $2$. We will use the parameter space of Legendre curves to describe $\eey(2)$. 

\begin{definition}
    Let $E$ be an elliptic curve over a base scheme $S$. A \define{full level $2$ structure} on $E/S$ is an isomorphism $\ul{(\zmod2)^2}_S\iso E[2]$ from the constant group scheme over $S$ to $2$-torsion in $E$.
\end{definition}

\begin{definition}\label{def: 2 torsion}
    Let $E$ be an elliptic curve over a base scheme $S$. The elliptic curve $E/S$ is said to be \define{equipped with a cyclic subgroup of order $2$} if there is a map $\ul{(\zmod2)}_S\into E[2]$ from the constant group scheme over $S$ to $2$-torsion in $E$.
\end{definition} 

\begin{remark}
    In the literature, $\eey_0(N)$ is the moduli of elliptic curves with a fixed order $N$ subgroup, and $\eey_1(N)$ is the moduli of elliptic curves with a fixed point of order $N$. Because a subgroup of order 2 has a unique generator by an element of order 2, the spaces $\eey_0(2)$ and $\eey_1(2)$ are the same. To simplify notation, we will denote a fixed subgroup of order 2 by its unique generator. 
\end{remark}

\begin{definition}
    A \define{Legendre curve} with parameter $t$ over a scheme $S$ is an elliptic curve $E_t$ over $S$ with Weierstrass equation $y^2=x(x-1)(x-t)$. This equation has discriminant $\Delta=16t^2(t-1)^2$, which is a unit if and only if $2$, $t$, $t-1$ are all invertible on $S$.
\end{definition}

Let $$X=\A^1_{\Z[\frac12]}\setminus\{0,1\}=\spec\p{\Z\sq{\frac12,t^{\pm1},\pinv{t-1}}}.$$ 
The universal Legendre curve
$$E:y^2=x(x-1)(x-t)$$
over $X$ has three points of order $2$, which are $(0,0)$, $(1,0)$, and $(t,0)$. Fixing $\{(0,0),(1,0)\}$ as an ordered basis for its 2-torsion defines a full level $2$ structure on $E$, and so induces a map $\pi:X\to\eey(2)$.

\begin{proposition}\label{prop:Y(2)-cms}
    The above constructed $\pi\colon X\to\eey(2)$ is a $C_2$-torsor with trivial action on $X$, so $\eey(2)\simeq BC_{2,X}$. Furthermore, the coarse space map $c\colon \eey(2)\to X$ is given by sending an elliptic curve $y^2=(x-e_1)(x-e_2)(x-e_3)$ to $t=\frac{e_3-e_1}{e_2-e_1}\in X$.
\end{proposition}
\begin{proof}
    This is \cite[Proposition 4.5 and Corollary 4.6]{AM}.
\end{proof}

\begin{remark}
    The existence of a full level 2 structure on an elliptic curve $E$ will always allow us to write $E$ in the form $y^2=(x-e_1)(x-e_2)(x-e_3)$, with 2-torsion points $(e_1,0)$, $(e_2,0)$, $(e_3,0)$ and full level 2 structure on~$E$ given by the isomorphism 
    $$\varphi\colon \ul{(\Z/2\Z)^2_S} \too E[2], \quad (1,0)\longmapsto (e_1,0), \ (0,1)\longmapsto (e_2, 0).$$
Equivalently, we specify the full level 2 structure by the ordered basis $\{(e_1,0), (e_2,0)\}$. 

The isomorphism from $E: y^2=(x-e_1)(x-e_2)(x-e_3)$, defined over a ring $R$ in which $e_2 - e_1$ is a square, to $E_t: y'^2= x'(x'-1)(x'-t)$, where $t=\frac{e_3-e_1}{e_2-e_1}$, is given by the change of coordinates $(x,y)\mapsto \p{\frac{x-e_1}{e_2-e_1}, \frac{y}{(e_2-e_1)^{3/2}}}$,
which maps the 2-torsion points $(e_1,0)\mapsto (0,0)$, $(e_2,0)\mapsto (1,0)$, $(e_3,0)\mapsto (t,0)$. 
\end{remark}

We have an action of $S_3$ on $\eey(2)$ given as follows. Consider the elliptic curve 
$$E: y^2=(x-e_1)(x-e_2)(x-e_3)$$ 
with full level 2 structure on $E$ given by the ordered basis $\{(e_1,0), (e_2,0)\}$. Then we have an action of $S_3$ on the set $\{e_1, e_2, e_3\}$, where $\sigma=(123)$ sends $\{e_1, e_2, e_3\}$ to $\{e_2, e_3, e_1\}$ and $\tau=(23)$ sends $\{e_1, e_2, e_3\}$ to $\{e_1, e_3, e_2\}$. In particular, $\tau$ acts on $\eey(2)$ by fixing $E$ and sending the basis $\{(e_1,0), (e_2,0)\}$ for $E[2]$ to the basis $\{(e_1, 0), (e_3,0)\}$.  Then this action of $S_3$ on $\eey(2)$ induces an action of $S_3$ on $X$ where $\eta\in S_3$ acts on $t\in X$ by sending $t=t(e_1, e_2, e_3)=\frac{e_3-e_1}{e_2-e_1}$ to $\eta(t)=t(\eta(e_1, e_2, e_3))$. 

\begin{remark}\label{rem:cms-C2-act}
    In particular, $\tau$ sends a curve $y^2=(x-e_1)(x-e_2)(x-e_3)$ to $y^2=(x-e_1)(x-e_3)(x-e_2)$ and so sends
    \[t=\frac{e_3-e_1}{e_2-e_1}\mapstoo\frac{e_2-e_1}{e_3-e_1}=\inv t\]
    on the coarse space $X$.
\end{remark}

Consider the following maps $\eey(2)\to\eey_0(2)$ and $\eey(2)\to \eem_{1,1}$. Let $\eey(2)\to\eey_0(2)$ be the map sending
$$\p{E,\phi\colon\ul{(\zmod2)^2}_S\to E[2]}\longmapsto\p{E,\phi(1,0)},$$
and let $\eey(2)\to \eem_{1,1}$ be the map sending
$$\p{E,\phi\colon\ul{(\zmod2)^2}_S\to E[2]}\longmapsto E.$$ 

Clearly, the $S_3$-action defined above acts trivially over $\eem_{1,1}$, and the restriction to the $C_2$-action acts trivially over $\eey_0(2)$.

\begin{proposition}\label{prop:C2-Gal}
    The map $\eey(2)\to\eey_0(2)$ is a $C_2$-Galois cover, i.e., is a $C_2$-torsor.
\end{proposition}
\begin{proof}
    Note that $C_2$ acts on $\eey(2)$ via $(E,\phi)\mapsto(E,\bar\phi)$, where
    \[\bar\phi(1,0)=\phi(1,0)\tand\bar\phi(0,1)=\phi(1,0)+\phi(0,1).\]
    In order to show that $\eey(2)\to\eey_0(2)$ is a $C_2$-Galois cover, it suffices to show that for any $\spec(R)\to\eey_0(2)$, there exists an \'etale cover $\spec(R')\to\spec(R)$ such that the base change $\eey(2)\by_{\eey_0(2)}\spec(R')$ is a trivial $C_2$-torsor over $\spec(R')$.
    
    Fix any $\Z[\frac{1}{2}]$-algebra $R$ as well as an $R$-point $(E/R,P)\in\eey_0(2)(R)$. Then, \cite[Proof of Lemma~4.3]{AM} shows that there exists some \'etale extension $R\into R'$ above which $E$ gains a full level 2 structure $\phi\colon\ul{(\zmod2)^2}_{R'}\to E[2]_{R'}$. We may modify $\phi$ by an automorphism in order to assume that we have $\phi(1,0)=P_{R'}\in E[2]_{R'}$. Then $\phi$ and $\bar\phi$ are the only two level structures with this property since the automorphism group of $(\Z/2\Z)^2$ fixing $(1,0)$ is $C_2$. Thus, we get an isomorphism
    \[(C_2)_{R'}=C_2\by\spec(R')\too\eey(2)\by_{\eey_0(2)}\spec(R')\]
    sending $0\mapsto(E,\phi)$ and $1\mapsto(E,\bar\phi)$, proving the claim.
\end{proof}

We will exploit Proposition~\ref{prop:C2-Gal} in order to compute the coarse space of $\eey_0(2)$. In brief, Proposition~\ref{prop:Y(2)-cms} shows that $\A^1\sm\{0,1\}$ is the coarse space of $\eey(2)$, so it inherits an action of $C_2$. The coarse space of $\eey_0(2)$ will be the geometric quotient of $\A^1\sm\{0,1\}$ by this $C_2$-action.

\begin{lemma}\label{lem:C2-inv-comp}
    Let $C_2$ act on $\A^1\sm\{0,1\}=\spec(\Z[t][\inv t,\pinv{t-1}])$ via $t\mapsto1/t$. Then, $(\A^1\sm\{0,1\})/C_2\simeq\A^1\sm\{0\}=\spec(\Z[s][1/s])$ via the map
    \[\frac t{(t-1)^2}\colon\A^1\sm\{0,1\}\too\A^1\sm\{0\}.\]
\end{lemma}
\begin{proof}
    This map is $C_2$-invariant, so factors, on the level of coordinate rings, as
    \[A:=\Z[s][1/s]\intoo\p{\Z[t][\inv t,\pinv{t-1}]}^{C_2}\intoo\Z[t][\inv t,\pinv{t-1}]=:B,\]
    where the first map is given by $s\mapsto t/(t-1)^2$. 
    We would like to show that the first map is an isomorphism. We note that $A\to B$ is integral (since $t\in B$ is a root of the monic $(x-1)^2-x/s\in A[x]$) and hence finite of degree $2$. The map $B^{C_2}\into B$ is also finite of degree $2$, so $A\to B^{C_2}$ is finite of degree $1$. As $A$ is integrally closed, this forces $A\to B^{C_2}$ to be an isomorphism, as desired.
\end{proof}

\begin{lemma}\label{lem:cms-construct}
    There is a map $\eey_0(2)\to\A^1_{\Z[\frac{1}{2}]}\sm\{0\}$ fitting into a commutative diagram
    \[\commsquare{\eey(2)}c{\A^1_{\Z[\frac{1}{2}]}\sm\{0,1\}}{}{t/(t-1)^2}{\eey_0(2)}{}{\A^1_{\Z[\frac{1}{2}]}\sm\{0\}\rlap{,}}\]
    where the top morphism sends $y^2=(x-e_1)(x-e_2)(x-e_3)$ to $\frac{e_3-e_1}{e_2-e_1}$.
\end{lemma}
\begin{proof}
    As a consequence of Lemma~\ref{prop:C2-Gal}, we have $\eey_0(2)\simeq[\eey(2)/C_2]$, so a morphism $\eey_0(2)\to\A^1\sm\{0\}$ is equivalently a $C_2$-invariant map $\eey(2)\to\A^1\sm\{0\}$. Because $\A^1\sm\{0\}$ is a scheme, such a map factors uniquely through the coarse space morphism $c$, and this is equivalently a $C_2$-invariant map $\A^1\sm\{0,1\}\to\A^1\sm\{0\}$, as $c\colon\eey(2)\to\A^1\sm\{0,1\}$ is $C_2$-equivariant. Lemma~\ref{lem:C2-inv-comp} shows $t/(t-1)^2$ is $C_2$-invariant, and the claim follows. 
\end{proof}

\begin{proposition}\label{prop:y0(2)-cms}
    The map $\eey_0(2)\to\A^1_{\Z[\frac{1}{2}]}\sm\{0\}$ of Lemma~\ref{lem:cms-construct} is a coarse space map.
\end{proposition}
\begin{proof}
    The Keel--Mori theorem, see \cite[Theorem 11.1.2]{olsson}, asserts the existence of a coarse moduli space $\eey_0(2)\to Y_0(2)$ characterized by the fact that it is initial among maps from $\eey_0(2)$ to algebraic spaces. We will show that $\eey_0(2)\to\A^1_{\Z[\frac{1}{2}]}\sm\{0\}$ satisfies this universal property.

    Let $f\colon\eey_0(2)\to Y$ be any map to an algebraic space $Y$; we will show it factors through $\eey_0(2)\to\A^1_{\Z[\frac{1}{2}]}\sm\{0\}$. By Proposition~\ref{prop:Y(2)-cms}, the map $\eey(2)\to\A^1_{\Z[\frac{1}{2}]}\sm\{0,1\}$ is a coarse space map; thus, the composition \smash{$\eey(2)\to\eey_0(2)\xto fY$} factors through $\A^1\sm\{0,1\}$. The induced map $\A^1\sm\{0,1\}\to Y$ is necessarily $C_2$-invariant, so further factors through $[(\A^1\sm\{0,1\})/C_2]=:\mathscr X$. All in all, we have a commutative diagram
    \[\begin{tikzcd}
        \eey(2)\ar[r]\ar[d]&\A^1\sm\{0,1\}\ar[d]\\
        \eey_0(2)\ar[r]\ar[rr, bend right, "f"']&\mathscr X\ar[r]&Y.
    \end{tikzcd}\]
    Finally, the coarse space of $\mathscr X$ is $(\A^1\sm\{0,1\})/{C_2}=\A^1\sm\{0\}$ by Lemma~\ref{lem:C2-inv-comp}, so $f$ must factor as
    \[\eey_0(2)\too\mathscr X\too\A^1\sm\{0\}\too Y,\]
    which shows that $\eey_0(2)\to\A^1\sm\{0\}$ is initial among maps to algebraic spaces, proving the claim.
\end{proof}

\begin{notation}
    Inspired by the above, we set $Y_0(2):=\A^1_{\Z[\frac{1}{2}]}\sm\{0\}=\spec(\Z[\frac{1}{2}][s,\inv s])$.
\end{notation}

\begin{remark}
    An alternate proof of Proposition~\ref{prop:y0(2)-cms} would be to combine the computation  of $(\A^1\sm\{0,1\})/C_2$ in Lemma~\ref{lem:C2-inv-comp} with \cite[Lemma 8.1.5]{katz-mazur}.
\end{remark}

Proposition~\ref{prop:y0(2)-cms} computed the coarse space over $S=\spec(\Z[\frac{1}{2}])$. Below, we extend this computation to arbitrary bases. We will do this by appealing to \cite[Corollary 3.3]{AOV}, which first requires us to show that $\eey_0(2)$ is a tame stack.

\begin{lemma}\label{lem:y0(2)-tame}
  Every geometric point of\, $\eey_0(2)$ has an automorphism group of order dividing $4$.
  In other words, for any algebraically closed field $k$ of characteristic not $2$, any elliptic curve $E/k$, and any point $P\in E[2](k)$ of exact order $2$, the group $\Aut(E,P)$ of automorphisms of\, $E$ which fix $P$ has order dividing $4$. As a consequence, $\eey_0(2)$ is a tame $\Z[\frac{1}{2}]$-stack in the sense of\, \cite[Definition 3.1]{AOV}.
\end{lemma}
\begin{proof}
  Let $k$, $E$, $P$ be as in the statement. Consider any $\sigma\in\Aut(E,P)$. Because $E[2](k)\cong\zmod2\oplus\zmod2$ and~$\sigma$ fixes $P$, one sees that $\sigma^2$ must fix $E[2](k)$ pointwise. Thus, \cite[Corollary 2.7.2]{katz-mazur} implies that $\sigma^2=\pm1$. Thus, $\sigma$ has order dividing $4$. We conclude that $\Aut(E,P)\subset\Aut(E)[4]$. Finally, \cite[Proposition A.1.2(c)]{silverman} shows that, when $\Char k\neq2$, the number
  $\#\Aut(E)[4]$ divides $4$ always.
\end{proof}

\begin{corollary}\label{cor:Y0(2)-cms-general-base}
    Let $\pi\colon \eey_0(2)\to Y_0(2)=\A^1_{\Z[\frac{1}{2}]}\sm\{0\}$ be the coarse space map of Proposition~\ref{prop:y0(2)-cms}. Then, the base change $\pi_S\colon \eey_0(2)_S\to Y_0(2)_S$ is also a coarse space map.
\end{corollary}
\begin{proof}
    This follows from Lemma~\ref{lem:y0(2)-tame} and \cite[Corollary 3.3]{AOV}.
\end{proof}

\begin{lemma}\label{lem:j-inv from s}
    The $j$-invariant map $j\colon Y_0(2)\to\A^1_{\Z[\frac{1}{2}]}$ is $j(s)=256(s+1)^3/s^2$.
\end{lemma}
\begin{proof}
    This follows from the commutative diagram
    \[\compdiag{\A^1\sm\{0,1\}}{t/(t-1)^2}{Y_0(2)}j{\A^1}{}\]
    and the well-known fact, see \cite[Proposition III.1.7(b)]{silverman}, that the $j$-invariant map $\A^1\sm\{0,1\}\to\A^1$ for the Legendre family is $t\mapsto256(t^2-t+1)^3/(t^2(t-1)^2)$.
\end{proof}

We summarize the results of this section in the diagrams below: 

\[\begin{tikzcd}
	{\eey(2)} & {\eey_0(2)} \\
	& {\eem_{1,1}[1/2]\rlap{.}}
	\arrow["{C_2}", from=1-1, to=1-2]
	\arrow["{\deg 3}", from=1-2, to=2-2]
	\arrow["{S_3}"', from=1-1, to=2-2]
\end{tikzcd}\]

The maps on the coarse spaces are given as follows:
\[\begin{tikzcd}
	{Y(2)\cong X=\A_t^1[1/2]\setminus \{0,1\}} & {Y_0(2)\cong \A_s^1[1/2]\setminus\{0\}} \\
	& {M_{1,1}[1/2]=\A_j^1[1/2]\rlap{,}}
	\arrow["{C_2}", from=1-1, to=1-2]
	\arrow["{\deg 3}", from=1-2, to=2-2]
	\arrow["{S_3}"', from=1-1, to=2-2]
\end{tikzcd}\]
where
\begin{align*}
    t &\longmapsto s=\dfrac{t}{(1-t)^2},\\
    t &\longmapsto j=\dfrac{256(t^2-t+1)^3}{t^2(t-1)^2},\\
    s &\longmapsto j=\dfrac{256(s+1)^3}{s^2}.
\end{align*}

\subsection{$\boldsymbol{BC_n}$ cohomology}\label{sec: BC_n coh}

We recall some results from \cite{AM} to compute the cohomology of $BC_{n,S}$ for a scheme $S$. To compute the cohomology of $\eey(2)$, they use the fact that $\eey(2)\cong BC_{2,X}$ (see Proposition~\ref{prop:Y(2)-cms} above), where $X$ denotes the universal Legendre curve.

More generally, let $S$ be a scheme, and let $C_{n,S}$ be the constant group scheme over $S$. We want to compute the cohomology of $BC_{n,S}$. We have the coarse space and covering maps
\begin{align*}
    c\colon BC_{n,S}\too S,\\
    \pi\colon S\too BC_{n,S}.
\end{align*}
Using the Galois cover $\pi$, one obtains the Hochschild--Serre spectral sequence:  
\[E^{p,q}_2=H^p\p{C_n,H^q(S,\G_m)} \Longrightarrow H^{p+q}(BC_{n,S}, \G_m).\]

Following \cite[Section 4.1]{lieblich2011period} and using the coarse space map $c$, we get another spectral sequence
\[E^{p,q}_2=H_{\fppf}^p(S,R^q_{\fppf}c_{*}\G_m)\Longrightarrow H_{\fppf}^{p+q}(BC_{n,S}, \G_m).\]

Combining the data obtained from these spectral sequences, we get the following proposition.

\begin{proposition}\label{Propn: BC_n Brauer}
Let $S$ be a scheme and $c\colon BC_{n,S}\to S$ be the coarse space map.

There exist an isomorphism
\[\G_m(BC_{n,S})\cong G_m(S)\]
as well as short exact sequences
\[0 \too \Pic(S)\overset{c^*}\too \Pic(BC_{n,S})\too \mu_n(S)\too 0\]
and
\[0\too H^2(S,\G_m)\overset{c^*}\too H^2(BC_{n,S},\G_m)\too H^1_{\pl}(S,\mu_n)\too 0,\]
\[0\too \Br'(S)\overset{c^*}\too \Br'(BC_{n,S})\too H^1_{\pl}(S,\mu_n)\too 0,\]
\[0\too \Br(S)\overset{c^*}\too \Br(BC_{n,S})\too H^1_{\pl}(S,\mu_n)\too 0,\]
all of which are split $($via $\pull\pi)$. The isomorphism and the short exact sequences are functorial in $BC_{n,S}$ $($i.e., endomorphisms of\, $BC_{n,S}$ induce endomorphisms of exact sequences in a functorial manner$)$, but the splittings
are only functorial in $S$.
\end{proposition}

\begin{proof}
    See \cite[Proposition 3.2]{AM}.
\end{proof}

\subsection{Group cohomology of $\boldsymbol{C_2}$}\label{sect:C2-coh}

Let $\rho=\Ind_{C_2}^{S_3}\Z$ be the $S_3$-representation induced from $\texttt{triv}$, $\Z$ with trivial $C_2$-action. Then $\rho$ is the standard permutation representation of $S_3$ on $\Z^{\oplus 3}$. We have a map of $S_3$-representations $\rho\to \texttt{triv}$ given by the $S_3$-equivariant map $\Z^{\oplus 3}\to \Z$, $(a,b,c)\mapsto a+b+c$. Let $\wt\rho$ be the kernel of the map $\rho\to\texttt{triv}$. 

\begin{lemma}
    For any connected normal noetherian scheme $S$ over $\Z[\frac{1}{2}]$, there is an $S_3$-equivariant exact sequence
    \begin{equation}\label{4.8}
    0\too\G_m(S)\too\G_m(X_S)\too\wt\rho\too0, 
\end{equation}
    where $S_3$ acts trivially on $\G_m(S)$ and $\wt\rho$ is generated by the images of $t$ and $t-1$. The exact sequence is non-equivariantly split.     
\end{lemma}
\begin{proof}
 This is  \cite[Lemma 4.8]{AM}.
\end{proof}

As in the proof given by \cite{AM}, we choose the isomorphism $\G_m(X_S)/\G_m(S)\iso\wt\rho$ to send $t\mapsto(0,1,-1)$ and $t-1\mapsto(1,0,-1)$. With these choices, we can easily compute how $\tau=(23)$ and $\sigma=(123)$ act on $t$ and $t-1$ in $\G_m(X_S)$. 

We have $\tau(t)=\frac{1}{t}$ and $\tau(t-1)=\frac{t-1}{t}$, because $\tau(0,1,-1)= (0,-1,1)$ corresponds to $\frac{1}{t}$ and $\tau(1,0,-1)=(1,-1,0)=(1,0,-1)-(0,1,-1)$ corresponds to $\frac{t-1}{t}$. A similar computation shows that $\sigma(t)=\frac{t-1}{t}$ and $\sigma(t-1)=t$. One can check that this $S_3$-action is equivalent to the one defined on the Legendre curves.

Since we are interested in the $C_2$-cover $\eey(2)\to \eey_0(2)$, we will use that the sequence (\ref{4.8}), when restricted to the action of $C_2=\langle(23)\rangle\subseteq S_3$, is a $C_2$-equivariant exact sequence, 
 \begin{equation}\label{4.8'}
    0\too\G_m(S)\too\G_m(X_S)\too\wt\rho|_{C_2}\too0, 
\end{equation}
which is now equivariantly split.

\begin{remark} \label{rem:C2-inv} 
  Concretely, $\wt\rho$ is the $S_3$-rep $\{(a,b,c)\in\Z^3:a+b+c=0\}$, and $C_2\into S_3$ acts by fixing $a$ and swapping $b,c$. An element $(a,b,c)$ of $\wt\rho$ is $C_2$-invariant if and only if $b=c$, so $C_2$-invariants are generated by $(-2,1,1)=(0,1,-1)-2(1,0,-1)$, which is the image of $\frac{t}{(t-1)^2}\in \G_m(X_S)$. In particular, the map $\G_m(X_S)^{C_2}\to H^0(C_2,\wt\rho)$ on invariants is surjective.
\end{remark}

\begin{fact}[{See \cite[Section 8]{coh-of-groups-CF}}]\label{cyclic coh}
    Given a cyclic group $C_n=\angles s$ acting on some abelian group $M$, one has
    \[H^i(C_n,M)\simeq\Threecases{M^{C_n}}{i=0}{\b{m\in M:(1+s+\dots+s^{n-1})\cdot m=0}\big/\b{s\cdot m-m:m\in M}}{i>0\t{ odd}}{\b{m\in M:s\cdot m-m=0}\big/\b{(1+s+\dots+s^{n-1})\cdot m:m\in M}}{i>0\t{ even}.}\]
\end{fact}

\begin{lemma}[\textit{cf.} {\cite[Lemmas 5.2 and~5.3]{AM}}]\label{lem:C2-coh}
    Let $M$ be a trivial $C_2$-module. Let \texttt{sgn} be the $1$-dimensional $C_2$-rep where the generator acts by $-1$. Then,
    \[H^i(C_2,M)=\Threecasesit{M}{i=0,}{M[2]}{i>0\text{ odd,}}{M/2M}{i>0\text{ even}}\tand  H^i(C_2,\texttt{sgn}\otimes M)=\Twocasesit{M[2]}{i\text{ even,}}{M/2M}{i\text{ odd}.}\]
    Even more useful, 
    \[H^i(C_2,\wt\rho\vert_{C_2}\otimes M)=\twocasesit M{i=0,}0\]
\end{lemma}
\begin{proof}
    Write $C_2=\angles{s}$. In all cases, we compute these cohomology groups using Fact~\ref{cyclic coh}. In the first case,~$M$ has trivial $C_2$-action, so $(1+s)\cdot m=2m$ and $s\cdot m-m=0$ for all $m\in M$. In the second case, $\texttt{sgn}\otimes M$ is isomorphic to $M$, where $s$ acts by negation, so $(1+s)\cdot m=0$ and $s\cdot m-m=-2m$. In the last case, $\wt\rho\vert_{C_2}\otimes M$ is isomorphic to $M^{\oplus2}$, where $s$ acts by swapping the two factors, so $(1+s)\cdot(m,n)=(m+n,m+n)$ and $s\cdot(m,n)-(m,n)=(n-m,m-n)$. The result follows from these descriptions and Fact~\ref{cyclic coh}.
\end{proof}

In addition to Lemma~\ref{lem:C2-coh}, we will also use the following three $S_3$-equivariant, and thus $C_2$-equivariant, short exact sequences from \cite{AM} fairly often: 
\begin{equation}\label{AM 5.4}
    0\too \G_m(S)\too \G_m(\eey(2)_S)\cong \G_m(X_S)\too \wt\rho|_{C_2}\too 0,
\end{equation}
\begin{equation}\label{AM 5.5}
    0\too \Pic(X_S)\cong \Pic(S) \too \Pic(\eey(2)_S)\too \mu_2(S) \too 0,
\end{equation}
\begin{equation}\label{AM 5.6}
    0\too \Br'(X_S)\too \Br'(\eey(2)_S) \too H^1(X_S, \mu_2)\too 0.
\end{equation}
These are sequences 5.4, 5.5, and 5.6 from \cite{AM}, respectively. In sequence (\ref{AM 5.4}), we have replaced $\wt\rho$ with $\wt\rho|_{C_2}$. Sequence (\ref{AM 5.5}) is also $S_3$-equivariantly split.

From the whole above discussion, we can deduce the following.

\begin{lemma}[\textit{cf.} {\cite[Lemma~5.7]{AM}}] 
\label{lem:01-lines-of-ss}
    If $S$ is a connected regular noetherian scheme over $\Z[\frac{1}{2}]$, then there exist a natural extension 
    \[0 \too \G_m(S) \too H^0(C_2, \G_m(\eey(2)_S)) \too H^0(C_2, \wt\rho|_{C_2}) \cong \Z \too 0\] 
    and isomorphisms 
\[
H^p(C_2, \G_m(S))\cong H^p(C_2, \G_m(\eey(2)_S))
\]
for all $p\ge 1$.  We also have natural isomorphisms 
\[
H^p(C_2, H^1(\eey(2)_S, \G_m))\cong H^p(C_2, \Pic(S))\oplus H^p(C_2, \mu_2(S))
\]
for all $p\ge 0$. 
\end{lemma}
\begin{proof}
    For the first two statements, one takes $C_2$-cohomology of the $C_2$-equivariant short exact sequence  (\ref{4.8'}), 
    \[0\too\G_m(S)\too\G_m(X_S)\too\wt\rho|_{C_2}\too0,\]
    and applies Lemma~\ref{lem:C2-coh} to get a long exact sequence 
\begin{align*}
0 &\too H^0(C_2, \G_m(S)) \too H^0(C_2, \G_m(X_S) \too H^0(C_2, \wt\rho|_{C_2}) \\ 
& \too H^1(C_2, \G_m(S)) \too H^1(C_2, \G_m(X_S) \too 0 \\ 
&\too H^2(C_2, \G_m(S)) \too H^2(C_2, \G_m(X_S) \too 0 \\ 
&\too \cdots. 
\end{align*} 
That the map $\G_m(\eey(2)_S)^{C_2}=\G_m(X_S)^{C_2}\to H^0(C_2,\wt\rho)$ is surjective follows from Remark~\ref{rem:C2-inv}. For the last statement, one appeals to the split $S_3$-equivariant, and thus split $C_2$-equivariant,
split exact sequence (\ref{AM 5.5})
\[
0\too \Pic(X_S)\cong \Pic(S)\too \Pic(\eey(2)_S)\too \mu_2(S)\too 0.
\]
Due to the splitting of the above sequence, after taking the long exact sequence in cohomology, we get a section $H^p(C_2, \mu_2(S))\to H^p(C_2,\Pic(\eey(2)_S))$ for all $p\geq 0$, which gives the result.
\end{proof}

Recall we have the short exact sequence (\ref{AM 5.6}), $0\to\Br'(X_S)\to\Br'(\eey(2)_S)\to H^1(X_S,\mu_2)\to0$. In the next lemma, we study $H^1(X_S,\mu_2)^{C_2}$.

\begin{lemma}[\textit{cf.} {\cite[Lemma 5.8]{AM}}\label{lem:C2-inv}]
    Let $S$ be a connected, normal noetherian $\Z[\frac{1}{2}]$-scheme. Then, there is a short exact sequence
    \[0\too H^1(S,\mu_2)\too H^1(X_S,\mu_2)^{C_2}\too\zmod2\too0.\]
    Furthermore, this sequence is non-canonically split, because everything above is an $\F_2$-vector space.
\end{lemma}
\begin{proof}
As in the proof of \cite[Lemma 5.8]{AM}, we use the $S_3$-equivariant, and thus $C_2$-equivariant, short exact sequence
$$0\too \G_m(X_S)/2\too H^1(X_S,\mu_2)\too \Pic(X_S)[2]\too 0$$
to obtain a long exact sequence
$$0\too (\G_m(X_S)/2)^{C_2}\too H^1(X_S,\mu_2)^{C_2}\too \Pic(X_S)[2]^{C_2}\too H^1(C_2, \G_m(X_S)/2)\too \cdots. $$
Since $X_S\to S$ is $C_2$-equivariant, we obtain the following map of complexes: 
\[\begin{tikzcd}
	0 & {\G_m(S)/2} & {H^1(S, \mu_2)} & {\Pic(S)[2]} & 0 \\
	0 & {(\G_m(X_S)/2)^{C_2}} & {H^1(X_S, \mu_2)^{C_2}} & {(\Pic(X_S)[2])^{C_2}} & {H^1(S_3, \G_m(X_S)/2)}
        \arrow[from=1-1, to=2-1]
	\arrow[from=1-1, to=1-2]
	\arrow[from=1-2, to=1-3]
	\arrow[from=1-3, to=1-4]
	\arrow[from=1-4, to=1-5]
	\arrow[hook, from=1-2, to=2-2]
	\arrow[from=2-1, to=2-2]
	\arrow[from=2-2, to=2-3]
	\arrow[from=2-3, to=2-4]
	\arrow[from=2-4, to=2-5]
	\arrow[hook, from=1-5, to=2-5]
	\arrow["\cong", from=1-4, to=2-4]
	\arrow[hook, from=1-3, to=2-3].
\end{tikzcd}\]

From the proof of \cite[Lemma 5.8]{AM}, the leftmost nonzero map factors through  the isomorphism $\G_m(S)/2\xrightarrow{\cong} (\G_m(X_S)/2)^{S_3}$ and is thus injective; furthermore, $\Pic(S)[2]\cong (\Pic(X_S)[2])^{C_2}$ (in fact, the Picard group  $\Pic(S)\cong\Pic(X_S)$ has trivial $C_2$-action).  By the four lemma, the second leftmost nonzero map is an injection. By the snake lemma, we have an isomorphism 
\begin{equation*}
    \coker \left[\G_m(S)/2\rightarrow \G_m(X_S)/2^{C_2}\right]\cong \coker\left[H^1(S,\mu_2)\to H^1(X_S,\mu_2)^{C_2}\right],  
\end{equation*}
and thus a short exact sequence 
\begin{equation}\label{ses:H1(mu2) stuff}
    0\too H^1(S, \mu_2)\too H^1(X_S, \mu_2)^{C_2}\too \coker \left[\G_m(S)/2\rightarrow \G_m(X_S)/2)^{C_2}\right]\too 0. 
\end{equation}

To compute the cokernel appearing above, we can apply the snake lemma to the following homomorphism of short exact sequences:
\[\begin{tikzcd}
    0\ar[r]&\G_m(S)\ar[r]\ar[d, "2"]&\G_m(X_S)\ar[r]\ar[d, "2"]&\wt\rho|_{C_2}\ar[r]\ar[d, "2"]&0\\
    0\ar[r]&\G_m(S)\ar[r]&\G_m(X_S)\ar[r]&\wt\rho|_{C_2}\ar[r]&0.
\end{tikzcd}\]
Using that $\wt\rho\cong\Z^2$ as $\Z$-modules, we know that the multiplication by $2$ map is injective, so the snake lemma gives
\begin{equation}\label{ses:units-mod-2}
    0\too\G_m(S)/2\too\G_m(X_S)/2\too\wt\rho|_{C_2}/2\too0.
\end{equation}
Applying Lemma~\ref{lem:C2-coh} to $\wt\rho\vert_{C_2}/2\simeq\wt\rho\vert_{C_2}\otimes\zmod2$ and taking $C_2$-invariants in (\ref{ses:units-mod-2}) yields
\[
    0\too\G_m(S)/2\too(\G_m(X_S)/2)^{C_2}\too\zmod2\overset{\delta}\too\G_m(S)/2\too H^1(C_2,\G_m(X_S)/2)\too 0.
\]
Note that $t/(t-1)^2$ is an element of $(\G_m(X_S)/2)^{C_2}$ and maps to a generator of $\zmod2$, since it is non-constant (for example because $t/(t-1)^2=t\in\G_m(X_S)/2$) so does not come from $\G_m(S)$. Thus $(\G_m(X_S)/2)^{C_2}\to \Z/2\Z$ is surjective, and we have a short exact sequence
\[
    0\too\G_m(S)/2\too(\G_m(X_S)/2)^{C_2}\too \zmod2\too0,
\]
which shows that $\coker [\G_m(S)/2\rightarrow \p{\G_m(X_S)/2}^{C_2}]=\Z/2\Z.$ Combined with (\ref{ses:H1(mu2) stuff}), this proves the lemma.
\end{proof}

\section{The \texorpdfstring{$\boldsymbol{p}$}{p}-primary torsion in \texorpdfstring{$\boldsymbol{\Br(\eey_0(2)_{\Z[\frac{1}{2p}]})}$}{Br(Y\_0(2)\_Z[1/2p])} for primes \texorpdfstring{$\boldsymbol{p\ge 3}$}{p at least 3}}\label{sect:p>=3 torsion, p invertible}

In this section, we describe $p$-primary torsion in the Brauer group of $\eey_0(2)$ over $\Z[\frac{1}{2p}]$, for $p\geq 3$. Following the strategy of \cite[Section 5 and Theorem 7.1]{AM}, we show that, over a base $S$ on which $2p$ is invertible, $\pprimary{\Br'}(\eey_0(2)_S)$ is an extension of some $H^1$ cohomology group by $\pprimary{\Br'}(S)$. Later, in Section~\ref{sect:p-prim-torsion-over-Z1/2}, when $S=\Z[\frac{1}{2p}]$, we will show that none of the $p$-primary Brauer classes on $\eey_0(2)_{\Z[\frac{1}{2p}]}$ extend to all of $\eey_0(2)$; \textit{i.e.}, we will show that $\pprimary{\Br'}(\eey_0(2))=0$ (see Proposition~\ref{prop:p-tors-over-Z[1/2]-i}).

\begin{lemma}\label{lem:p-tors-first-step}
    Let $S$ be a regular noetherian scheme. Then, $\pprimary{\Br}'(\eey_0(2)_S)\cong\pprimary{\Br}'(X_S)^{C_2}$.
\end{lemma}
\begin{proof}
    The only contribution to $\pprimary{\Br'}(\eey_0(2)_{S})$ in the $p$-local Hochschild--Serre spectral sequence 
    \[E_2^{i,j}=H^i\p{C_2,H^j\p{\eey(2)_S,\G_m}}_{(p)}\implies H^{i+j}\p{\eey_0(2)_S,\G_m}_{(p)}\]
    is $E_2^{0,2}=\pprimary{\Br}'(\eey(2)_S)^{C_2}$.
    Indeed, for $i\ge1$ and $j\ge0$, the group $H^i(C_2,H^j(\eey(2)_S,\G_m))$ is $2$-torsion (because $C_2$ is). Consequently, $E_2^{i,j}=H^i(C_2,H^j(\eey(2)_S,\G_m))_{(p)}=0$ when $i\ge1$. Thus, the spectral sequence collapses on the $E_2$-page and $\pprimary{\Br'}(\eey_0(2)_{S})=\pprimary{\Br'}(\eey(2)_S)^{C_2}=\pprimary{\Br'}(BC_{2,X_S})^{C_2}$.
    Now, consider the split short exact sequence 
    \[0\too \Br'\p{X_S} \overset{c^*}\too \Br'\p{BC_{2,X_S}}\overset{r}\too H^1\p{X_S,\mu_2}\too 0\]
    from Proposition~\ref{Propn: BC_n Brauer}. 
    Taking $p$-primary torsion of this sequence, we obtain
    \[0 \too \pprimary{\Br'}(X_S) \too \pprimary{\Br'}\p{BC_{2,X_S}} \too \pprimary{H^1}(X_S,\mu_2)\too 0.\]
    Then  we see that $\pprimary{\Br'}(X_S)\cong \pprimary{\Br'}(BC_{2,X_S})$ because $\pprimary{H^1}(X_S,\mu_2)=0$ for $p\geq 3$ and thus $\pprimary{\Br'}(\eey_0(2)_{S}) \cong \hphantom{\mkern-2mu}\pprimary{\Br'}(X_S)^{C_2}$.  
\end{proof}

\begin{proposition}[\textit{cf.} {\cite[Theorem 7.1]{AM}}]\label{thm:p-tors-away-from-p}
    Let $p\geq 3$ be prime, and let $S$ be a regular noetherian scheme on which $2p$ is invertible. Then we have a non-canonically split exact sequence 
    \[0\too\pprimary{\Br'}(S)\too \pprimary{\Br'}(\eey_0(2)_S)\too H^1\p{S,\Q_p/\Z_p}\too 0.\]
The map $\pprimary{\Br'}(\eey_0(2)_S)\to H^1(S,\Q_p/\Z_p)$ can be described as the composition of pullback to $X_S$ and taking the ramification at the divisor $0$. 
\end{proposition}

\begin{proof} 
By Lemma~\ref{lem:p-tors-first-step}, we know that $\pprimary{\Br}'(\eey_0(2)_S)\cong\pprimary{\Br}'(X_S)^{C_2}$.
Since $p$ is invertible on $S$, Lemmas 5.10 and~5.11 of \cite{AM} give an exact sequence 
\[
0\too \pprimary{\Br'}(S)\too \pprimary{\Br'}(X_S)\too \pprimary{H^3}_{\set{0, 1}}\p{\A^1_S, \G_m}\too 0
\]
such that $\pprimary{H^3}_{\{0,1\}}(\A^1_S,\G_m)$ is $C_2$-equivariantly isomorphic to $\wt\rho\vert_{C_2}\otimes H^1\p{S,\Q_p/\Z_p}$. 
Taking $C_2$-cohomology and using Lemma~\ref{lem:C2-coh}, we obtain a long exact sequence
\[
0\too \pprimary{\Br'}(S) \too \p{\pprimary{\Br'}(X_S)}^{C_2}\too \p{\wt\rho\vert_{C_2}\otimes H^1\p{S, \Q_p/\Z_p}}^{C_2}\cong H^1\p{S, \Q_p/\Z_p}\too  \pprimary{\Br'}(S)[2] = 0.
\]
Noting that $(\pprimary{\Br'}(X_S))^{C_2} \cong \Br'(\eey_0(2)_S)$, we obtain the following short exact sequence: 
\[
0\too \pprimary{\Br'}(S) \too \Br'(\eey_0(2)_S)\too H^1\p{S, \Q_p/\Z_p}\too 0.
\]

Since $\eey_0(2)$ has a $\Z[\frac{1}{2}]$-point (see Remark~\ref{rem:Z[1/2]-point}), it also has an $S$-point, and so the sequence is split.

The isomorphisms $\pprimary{Br'}(\eey_0(2))\cong \pprimary{\Br'}(BC_{2,X_S})^{C_2}\cong \pprimary{\Br'}(X_S)^{C_2}$ are induced by pullback to $X_S$.

The map $\pprimary{\Br'}(X_S)\to \pprimary{H^3}_{0,1}(\A^1_S, \G_m)$ is given by taking the ramification divisors at 0 and 1. The $C_2$-action on $\pprimary{H^3}_{\{0,1\}}(\A^1_S, \G_m)$ is isomorphic to $\wt\rho|_{C_2}\otimes H^1(S,\Q_p/\Z_p)$, which, as shown in the proof of Lemma~\ref{lem:C2-coh}, is isomorphic to $H^1(S,\Q_p/\Z_p)\oplus H^1(S,\Q_p/\Z_p)$ with $C_2$-action the swapping action. Thus, the isomorphism $(\wt\rho|_{C_2}\otimes H^1(S,\Q_p/\Z_p))^{C_2}\cong H^1(S,\Q_p/\Z_p)$ is given by projection onto the first coordinate, so $\hphantom{\mkern-2mu}\pprimary{\Br'}(X_S)\to H^1(S,\Q_p/\Z_p)$ is just given by the ramification at 0. 
\end{proof}

\begin{remark}\label{rem:Z[1/2]-point}
    We remark here that $\eey_0(2)$ has a $\Z[\frac{1}{2}]$-point. For example, there is the curve $E:y^2=x^3-11x-14$ (which has discriminant $2^5$) equipped with the $2$-torsion point $(-2,0)$, defined over $\Z[\frac{1}{2}]$. This is the curve \cite[\href{https://www.lmfdb.org/EllipticCurve/Q/32/a/1}{Elliptic Curve 32.a1}]{lmfdb}.
\end{remark}

\begin{corollary} \label{cor: p-tors-over-Q}
Let $p\ge 3$ be a prime.  Then
    \[\pprimary{\Br'}\p{\eey_0(2)_\Q} = \pprimary{\Br'}(\Q) \oplus H^1\p{\Spec(\Q), \Q_p/\Z_p}.\] 
\end{corollary}

When $p\ge 3$, we can also compute the $p$-primary torsion of 
$\Br'(\eey_0(2)_{\Z[\frac{1}{2p}]})$ more explicitly.

\begin{lemma}\label{lem:H1Qp/Zp}
    Fix a prime $p\ge 3$. Then,
    \[H^1\p{\Z\sq{\frac1{2p}},\frac{\Q_p}{\Z_p}}\simeq\frac{\Q_p}{\Z_p}.\]
\end{lemma}
\begin{proof}
    Because cohomology commutes with colimits, we have
    \[H^1\p{\Z\left[\frac{1}{2p}\right],\Q_p/\Z_p}=\dirlim_nH^1\p{\Z\left[\frac{1}{2p}\right],\zmod{p^n}},\]
    so it suffices to show that $H^1(\Z[\frac{1}{2p}],\zmod{p^n})\simeq\zmod{p^n}$. For this, \cite[Example 11.3]{milneLEC} gives the first isomorphism below:
    \[H^1\p{\Z\left[\frac{1}{2p}\right],\zmod{p^n}}\simeq H^1_{\t{Group}}\p{\etpi_1\p{\Z\left[\frac{1}{2p}\right]},\zmod{p^n}}=\ctsHom\p{\etpi_1\p{\Z\left[\frac{1}{2p}\right]},\zmod{p^n}}.\]
    At the same time, $\etpi_1(\Z[\frac{1}{2p}])=:G$ is the Galois group of the maximal extension $K/\Q$ unramified away from~$2$ and $p$. By class field theory, $G^{\t{ab}}\simeq\units\Z_2\oplus\units\Z_p$. Since $\units\Z_2\cong\zmod2\oplus\Z_2$ and $\units\Z_p\cong\units\F_p\oplus\Z_p$, we conclude that
    \[\ctsHom\p{\etpi_1\p{\Z\left[\frac{1}{2p}\right]},\zmod{p^n}}=\ctsHom(\Z_p,\zmod{p^n})\simeq\zmod{p^n},\]
    proving the claim.
\end{proof}

\begin{corollary} 
Let $p\ge 3$ be a prime.  Then
    \[\pprimary{\Br'}\p{\eey_0(2)_{\Z[\frac{1}{2p}]}} \simeq \Q_p/\Z_p \oplus\Q_p/\Z_p.\] 
\end{corollary}
\begin{proof}
    By \cite[Example 2.19(3)]{AM}, we have $\Br(\Z[\frac{1}{2p}])=\zmod2\oplus\Q/\Z$. The claim follows from this, Lemma~\ref{lem:H1Qp/Zp}, and Theorem~\ref{thm:p-tors-away-from-p}.
\end{proof}

\section{The \texorpdfstring{$\boldsymbol{p}$}{p}-primary torsion in \texorpdfstring{$\boldsymbol{\Br(\eey_0(2))}$}{Br(Y\_0(2))} for primes \texorpdfstring{$\boldsymbol{p\ge3}$}{p at least 3}}\label{sect:p-prim-torsion-over-Z1/2}

In Section~\ref{sect:p>=3 torsion, p invertible}, we were able to compute $\pprimary{\Br'}(\eey_0(2)_S)$ for regular, noetherian schemes $S/\Z[\frac{1}{2p}]$. Since we are ultimately interested in computing the Brauer group of $\eey_0(2)$, which is defined over $\Z[\frac{1}{2}]$, we would now like to indicate how to extend this computation to regular, noetherian schemes $S/\Z[\frac{1}{2}]$. After some preliminaries, we will specialize to the case $S=\spec\Z[\frac{1}{2}]$.

\begin{notation}
    For $S/\Z[\frac{1}{2}]$ and $p\ge3$ prime, we set $S[1/p]:=S_{\Z[\frac{1}{2p}]}=S\by_{\Z[\frac{1}{2}]}\Z[\frac{1}{2p}]$. Note that $\eey_0(2)_S[1/p]\simeq\eey_0(2)_{S[1/p]}$. We will sometimes write this instead as $\eey_0(2)_S[1/2p]=\eey_0(2)_S[1/p]$ to emphasize that both $2$ and $p$ are inverted on the both.
\end{notation}

The key to computing $\Br'(\eey_0(2)_S)$ will be to use the fact \cite[Proposition 2.5(iv)]{AM} that, for $S$ a regular, noetherian $\Z[\frac{1}{2}]$-scheme and $p\ge3$ a prime, $\Br'(\eey_0(2)_S)\into\Br'(\eey_0(2)_S[1/p])$ is injective. Thus, the problem of computing $\Br'(\eey_0(2)_S)$ becomes one of checking which Brauer classes over $\eey_0(2)_S[1/p]$ extend to $\eey_0(2)_S$. In order to show that certain classes of $\eey_0(2)_S[1/p]$ do \textit{not} extend over all of $\eey_0(2)_S$, we use a strategy due to Antieau and Meier---described below as well as in \cite[Discussion under Theorem~1.3]{AM}---which ultimately rests on the observation that $\Br'(\msO_K)=0$ for any nonarchimedean local field $K$; see \cite[Corollary 6.9.3]{poonen-rat-pts}.

\begin{remark}[Strategy for showing Brauer classes do not extend]\label{rem:class-extend-strat}
    Fix a $\Z[\frac{1}{2}]$-scheme $S$, a prime $p\ge3$, and a Brauer class $\alpha\in\Br'(\eey_0(2)_S[1/p])$. In this remark, we describe a criterion for showing $\alpha$ does not extend over all of $\eey_0(2)_S$. 

    Let $K$ be a nonarchimedean local field of residue characteristic p. Suppose that we are given some $f\colon \spec(\msO_K)\to\eey_0(2)_S$, \textit{i.e.}, a tuple $(E,P,g\in S(\msO_K))$, where $E/\msO_K$ is an elliptic scheme, $P\in E[2](\msO_K)$ is a point of exact order $2$ in both fibers, and $g$ is an $\msO_K$-point of $S$ (note that $g$ is superfluous if $S=\Z[\frac{1}{2}]$).
    Let $f_K$ be the restriction of $f$ to the generic fiber over $\msO_K$ (so $f_K$ lands in $\eey_0(2)_S[1/p]$), and consider the commutative square
    \[\commsquare{\spec(K)}{f_K}{\eey_0(2)_S[1/p]}{}{}{\spec(\msO_K)}f{\eey_0(2)_S\rlap{.}}\]
    By \cite[Corollary 6.9.3]{poonen-rat-pts}, we have $\Br'(\msO_K)=0$; thus if $\alpha\in\Br'(\eey_0(2)_S[1/p])$ came from $\Br'(\eey_0(2)_S)$, then we would have $\pull f_K(\alpha)=0\in\Br'(K)$, as can be checked by moving the other way around the square. Thus, taking the contrapositive, if one can construct an $f$ such that $\pull f_K(\alpha)\neq0$, then one can conclude that $\alpha\not\in\im\p{\Br'(\eey_0(2)_S)\into\Br'(\eey_0(2)_S[1/p])}$.
\end{remark}

Remark~\ref{rem:class-extend-strat} reduces the problem of showing that certain Brauer classes do not extend to the problem of constructing suitable families of elliptic curves. Following the strategy of \cite[Sections 7 and~8]{AM}, we will in turn reduce this to computing certain Hilbert symbols. In the end, we will be able to prove the following.

\begin{proposition}\label{prop:p-tors-over-Z[1/2]-i}
    For all primes $p\geq 3$, we have $\pprimary{\Br'}(\eey_0(2))=0$.
\end{proposition}

\begin{setup}
    Throughout the remainder of this section, fix a prime $p\ge3$, and for now, fix a regular, connected noetherian scheme $S/\Z[\frac{1}{2p}]$.
\end{setup}

Recall (Theorem~\ref{thm:p-tors-away-from-p})
    there is a split exact sequence
    \begin{equation}\label{ses:p-tors-away-from-p}
        0\too\pprimary{\Br'}(S)\too\pprimary{\Br'}(\eey_0(2)_S)\too H^1\p{S,\Q_p/\Z_p}\too0,
    \end{equation}
    where the map $\pprimary{\Br'}(\eey_0(2)_S)\to H^1(S,\Q_p/\Z_p)$ is computed by pulling a Brauer class $\alpha\in\pprimary{\Br'}(\eey_0(2)_S)$ back to a class $\alpha_X\in\Br'(X_S)$, and then taking its ramification $\opname{ram}_0(\alpha_X)\in H^1(S,\Q_p/\Z_p)$ at the divisor $t=0$.
    
Our first goal is to determine an explicit splitting of the exact sequence (\ref{ses:p-tors-away-from-p}) and so have an explicit description of all classes in $\pprimary{\Br'}(\eey_0(2)_S)$. We will do this by following the strategy of \cite[Lemma 7.2]{AM}. 
\begin{notation}
    Let $\pi_S\colon\eey_0(2)_S\to Y_0(2)_S=\A^1_S\sm\{0\}$ be the coarse space map of Corollary~\ref{cor:Y0(2)-cms-general-base}, and let~$s$ denote the coordinate on $Y_0(2)_S=\rSpec_S(\msO_S[s,\inv s])$. Let $T_n\to\eey_0(2)_S$ be a $\mu_{p^n}$-torsor representing the image of $s\in\G_m(Y_0(2)_S)$ under the Kummer map
    \[\kappa_{p^n}\colon\G_m(Y_0(2)_S)=\G_m(\eey_0(2)_S)\too H^1(\eey_0(2),\mu_{p^n});\]
    \textit{e.g.}, $T_n$ could be the torsor $\eey_0(2)(\sqrt[p^n]s)$.
\end{notation}
\begin{remark}
    The above definition of $T_n$ makes sense even when $p=2$.
\end{remark}

\begin{theorem}\label{thm:sect-construct}
    The colimit $\sigma$ of the compositions
    \[\sigma_n\colon H^1(S,C_{p^n})\too H^1(\eey_0(2)_S,C_{p^n})\xto{\smile[T_n]}H^2(\eey_0(2)_S,\mu_{p^n})\too\pprimary{\Br'}(\eey_0(2)_S),\]
    for all $n\ge1$, gives a section $\sigma\colon H^1(S,\Q_p/\Z_p)\to\pprimary{\Br'}(\eey_0(2)_S)$ of the exact sequence~\eqref{ses:p-tors-away-from-p}.
\end{theorem}
\begin{proof}
    These maps are compatible as $n$ varies, so it suffices to separately check that each composition
    \[H^1(S,C_{p^n})\overset{\sigma_n}\too\Br'(\eey_0(2)_S)[p^n]\too H^1(S,\Q_p/\Z_p)[p^n]=H^1(S,C_{p^n})\]
    is the identity. By Theorem~\ref{thm:p-tors-away-from-p}, the map $\Br'(\eey_0(2)_S)[p^n]\to H^1(S,C_{p^n})$ is computed by pulling back to $X_S$ and then taking the ramification along $t=0$.
     By construction (see Lemma~\ref{lem:C2-inv-comp}), the invertible function $s\in\G_m(Y_0(2)_S)$ pulls back to $t/(t-1)^2\in\G_m(X_S)$, so the torsor $T_n\to\eey_0(2)_S$ pulls back to a torsor $T_n'\to X_S$ representing the image of $t/(t-1)^2$ in the Kummer map on $X_S$. Thus, the composition $H^1(S, C_{p^n})\to H^1(S, C_{p^n})$ is equivalently given by first pulling back $\chi\in H^1(S,C_{p^n})$ to $X_S$, then taking the cup product with $[T_n']$, and finally taking the ramification along $t=0$. 

    Let $k$ be the function field of $S$. We remark that the restriction map $H^1(S,C_{p^n})\to H^1(k,C_{p^n})$ is injective. Indeed, $H^1(S,C_{p^n})\simeq\ctsHom(\etpi_1(S),C_{p^n})$ (and similarly for $H^1(k,C_{p^n})$), and the natural map $\etpi_1(\spec(k))\to\etpi_1(S)$ is a surjection by \cite[Proposition 3.3.6]{fu-etale}; these together show that $H^1(S,C_{p^n})\to H^1(k,C_{p^n})$ is injective. Thus, to show that the composition $H^1(S,C_{p^n})\to H^1(S,C_{p^n})$ is the identity map, it suffices to show that it is the identity after base change to $k$. 

    Let $K:=k(t)$ be the function field of $X_S$, and let $R:=k[t]_{(t)}$ be the local ring at $t=0$. Then $t/(t-1)^2$ is a uniformizer at $t=0$, and the base change of the composition map is given by sending a class $\chi\in H^1(k, C_{p^n})$ to $\ram_0(\chi', \frac{t}{(t-1)^2})$, where $\chi'$ is the pullback of $\chi$ to $K$. 

    Finally, consider the Galois extension of fields $L/K$ defined by $\chi'$. Let $S$ be the integral closure of $R$ in $L$. Since the extension $L/K$ came from $\chi$ over $k$, the class of the extension $S/R$ also comes from $\chi$, and thus by \cite[Proposition 2.16]{AM}, we conclude that $\ram_0(\chi', \frac{t}{(t-1)^2})=\chi$, and thus that the composition is the identity map. 
\end{proof}

\subsection{Connection to Hilbert symbols}
For the remainder of this section, we consider $\eey_0(2)$ over the base $S=\spec(\Z[\frac{1}{2}])$.
By Theorem~\ref{thm:sect-construct}, every $\alpha\in\Br'(\eey_0(2)[1/p])$ is a sum of a constant class (one coming from $\Br'(\Z[\frac{1}{2p}])$) and the class of a cyclic algebra, \textit{i.e.}, the cup product of a $C_{p^n}$-torsor and the $\mu_{p^n}$-torsor $T_n$. In Remark~\ref{rem:class-extend-strat}, we saw that determining whether a class $\alpha$ in $\Br'(\eey_0(2)[1/p])$ extends over all of $\eey_0(2)$ is related to the problem of evaluating $\alpha\in\Br'(\eey_0(2)[1/p])$ at points of $\eey_0(2)$ valued in local fields. By the below result, when pulled back to a suitable local field, the triviality of cyclic algebras can be understood in terms of classical Hilbert symbols.

 \begin{notation}
    Let $K$ be a local field containing a primitive $\supth{n}$ root of unity $\zeta$. The choice of $\zeta$ defines an isomorphism $\ul{C_n}_K\iso\mu_{n,K}$. Combined with the usual Kummer isomorphism, we have $H^1(K,C_n)\iso H^1(K,\mu_n)\simeq\units K/(\units K)^n$. Given $a,b\in\units K$, let $(a,b)_\zeta\in\Br(K)$ denote the image of $a\otimes b$ under the composition
    \[\frac{\units K}{(\units K)^n}\otimes\frac{\units K}{(\units K)^n}\isoo H^1(K,C_n)\otimes H^1(K,\mu_n)\xtoo\smile H^2(K,\mu_n)\too H^2(K,\G_m)=\Br(K).\]
\end{notation}

\begin{proposition}[\textit{cf.} {\cite[Proposition 2.17]{AM}}]\label{prop:cyclic-alg-hilbert}
     Let $K$ be a local field containing a primitive $\supth{n}$ root of unity $\zeta$. Let $\mfp$ denote the prime ideal in the valuation ring of $K$, and let
     \[\pfrac{-,-}\mfp\colon\frac{\units K}{(\units K)^n}\otimes\frac{\units K}{(\units K)^n}\too\mu_n(K)\]
     denote the usual Hilbert symbol; see \cite[Section V.3]{neukirch}. Then, for any $a,b\in\units K$, 
     \[(a,b)_\zeta=1\in\Br(K)\iff\pfrac{a,b}\mfp=1\in\mu_n(K).\]
\end{proposition}
We have a commutative diagram
\[\begin{tikzcd}
     &\pprimary{\Br'}(\Z[\frac{1}{2}])\ar[r, hook]\ar[d]&\pprimary{\Br'}(\eey_0(2)[1/2])\ar[d]\\
        0\ar[r]&\pprimary{\Br'}(\Z[\frac{1}{2p}])\ar[r]&\pprimary{\Br'}(\eey_0(2)[\frac{1}{2p}])\ar[r]&H^1(\Z[\frac{1}{2p}],\Q_p/\Z_p)\ar[r]\ar[l, "\sigma"', bend right]&0,
    \end{tikzcd}\]
with $\sigma$ the section computed in Theorem~\ref{thm:sect-construct}. We are interested in figuring out which Brauer classes in the bottom row come from ones in the top row. We start by determining when the elements of $\pprimary{\Br'}(\eey_0(2)[1/p])$ coming from $H^1(\Z[\frac{1}{2p}],\Q_p/\Z_p)$ do not extend to $\pprimary{\Br'}(\eey_0(2))$.

\begin{remark} \label{rem:alphan}
    Our earlier computation (Lemma~\ref{lem:H1Qp/Zp}) that $H^1(\Z[\frac{1}{2p}],\Q_p/\Z_p)\simeq\Q_p/\Z_p$ shows that, for every~$n$, there is a unique Galois extension $F_n/\Q$ which is unramified outside $2p$ and whose Galois group is isomorphic to $\zmod{p^n}$. In fact, because $\Q(\zeta_{p^{n+1}})/\Q$ has (cyclic) Galois group $\units{(\zmod{p^{n+1}})}$ and is unramified outside $p$, we see that $F_n$ is the unique subfield of $\Q(\zeta_{p^{n+1}})$ with Galois group $\zmod{p^n}$ over $\Q$. Thus, $H^1(\Z[\frac{1}{2p}],C_{p^n})=H^1(\Z[\frac{1}{2p}],\Q_p/\Z_p)[p^n]$ is generated by the class of the ring of integers of $F_n/\Q$.
\end{remark}

\begin{definition}\label{defn:alpha n}
    We define $\alpha_n:=[\msO_{F_n}]$ to be the class of the ring of integers of $F_n/\Q$ in $H^1(\Z[\frac{1}{2p}],C_{p^n})$.
\end{definition}

\begin{proposition} \label{prop:not extend}
    Let $K=K_n=\Q_p(\zeta_{p^n})$, $\mfp$ be the prime ideal in the valuation ring of $K$, $\alpha_n\in H^1(\Z[\frac{1}{2p}],C_{p^n})$ be as in Definition~\ref{defn:alpha n}, and $t_0\in \msO_{K}$ be such that $E:y^2=x(x-1)(x-t_0)$ has good reduction at $\mfp$ $($i.e., is a Legendre curve over $\msO_K)$ and \[\pfrac{\zeta_p,t_0/(t_0-1)^2}\mfp\neq1.\]
    Let $f\colon \Spec (\msO_K) \to \eey_0(2)$ be the map given by $E$, and $f_K\colon\Spec (K)\to \eey_0(2)[1/p]$ the pullback of the map to the generic fiber. Then $f_K^*(\sigma(\alpha_n))\neq 0$, where $\sigma$ is the section computed in Theorem~\ref{thm:sect-construct}, and thus $\sigma(\alpha_n)$ does not extend to $\pprimary{\Br'}(\eey_0(2))$.
\end{proposition}

\begin{proof}
    We first explicitly compute the class of $\alpha_n$ restricted to $K$. Since $K$ contains a primitive $\supth{(p^n)}$ root of unity,
    \[H^1(K,C_{p^n})=H^1(K,\mu_{p^n})=\G_m(K)/p^n=\frac{\units K}{(\units K)^{p^n}},\]
    where $x\in\units K$ corresponds to the class of the extension $K\p{\sqrt[p^n]x}/K$ in $H^1(K,C_{p^n})$.

    The class of $\alpha_n$, the extension $F_n$ over $\Q$, restricted to $K$ is given by the class of the extension $KF_n$ over~$K$. The compositum $KF_n$ is equal to $\Q_p(\zeta_{p^{n+1}})$, because we have $KF_n\subset\Q_p(\zeta_{p^{n+1}})$ by construction and $[KF_n:\Q_p]$ is divisible by both $p^n$ and $(p-1)$, again by construction. Since $\zeta_{p^{n+1}}=\sqrt[p^n]{\zeta_p}$, we conclude that $\alpha_n=[F_n]\in H^1(\Z[\frac{1}{2p}],C_{p^n})$ pulls back to $\zeta_p\in\units K/(\units K)^{p^n}\cong H^1(K, C_{p^n})$.

     We have the following commutative diagram: 
     \[\begin{tikzcd}
        {}\pprimary{\Br'}(K)&\pprimary{\Br'}(\eey_0(2)_{K})\ar[l,"t=t_0"] \ar[r] & H^1\p{K,C_{p^n}} \ar[l, "\sigma"', bend right]\\ &
        \pprimary{\Br'}\p{\eey_0(2)[\frac{1}{2p}]}\ar[r]\ar[u]& H^1\p{\Z[\frac{1}{2p}],C_{p^n}}\ar[u]\ar[l, "\sigma"', bend right],
    \end{tikzcd}\]
    in which $\sigma$ is the section in Theorem~\ref{thm:sect-construct} and the vertical arrows up are restrictions to $K$. The vertical map on the right is given by $\alpha_n\mapsto \zeta_p$, and $f_{K_n}^*$ is given by the composition $\pprimary{\Br'}(\eey_0(2)[\frac{1}{2p}])\to{}\pprimary{\Br'}(K_n)$. Thus $f_{K_n}^*(\sigma(\alpha_n))$ is given by the cup product $(\zeta_p, t_0/(t_0-1)^2)$, which is nontrivial since   \[\pfrac{\zeta_p,t_0/(t_0-1)^2}\mfp \neq 1.\]

    Therefore, using the strategy in Remark~\ref{rem:class-extend-strat}, we conclude that $\sigma(\alpha_n)$ does not extend to $\pprimary{\Br'}(\eey_0(2))$. 
\end{proof}

\begin{remark}\label{rem:extend-argument-simplification}
    If $\alpha_1$, which generates $H^1(\Z[\frac{1}{2p}],\Q_p/\Z_p)[p]$, does \textit{not} extend over $\eey_0(2)$---\textit{e.g.}, we find a suitable~$t$ making
    \[\pfrac{\zeta_p,t/(t-1)^2}\mfp\neq1,\]
    where $\mfp$ is the prime of $K_1=\Q_p(\zeta_p)$---then in fact no nonzero $\alpha\in H^1(\Z[\frac{1}{2p}],\Q_p/\Z_p)$ will extend over $\eey_0(2)$. Indeed, suppose some such $\alpha$ does. Then all of its multiples extend over $\eey_0(2)$ as well. However, $\alpha$ will be of order $p^n$ for some $n\ge1$, so $p^{n-1}\alpha$ will be a generator for the $p$-torsion in $H^1(\Z[\frac{1}{2p}],\Q_p/\Z_p)\cong\Q_p/\Z_p$. Thus, $\alpha_1=up^{n-1}\alpha$ for some $u$, from which we would deduce that $\alpha_1$ extends over $\eey_0(2)$, which contradicts the assumption on $\alpha_1$.
\end{remark}

Thus, we arrive at the task of computing some Hilbert symbols over $K_1=\Q_p(\zeta_p)$. Let $\mfp=(\zeta_p-1)\subset\Z_p[\zeta_p]=\msO_{K_1}$ be its maximal ideal.
\begin{remark}\label{rem:hilb-comp}
    Consider any $a\in\msO_{\Q_p(\zeta_p)}=\Z_p[\zeta_p]$ such that $a\equiv1\pmod\mfp$. Then, \cite[Proposition V.3.8]{neukirch} tells us that
    \[\pfrac{\zeta_p,a}\mfp=\zeta_p^{\Tr(\log a)/p}\in\mu_p\p{\Q_p\p{\zeta_p}},\]
    where $\Tr=\Tr_{\Q_p(\zeta_p)/\Q_p}$ and $\log$ is the usual $p$-adic logarithm:
    \[\log(1-x)=-\sum_{n\ge1}\frac{x^n}n.\]
\end{remark}

\begin{proposition}\label{prop:exists-t-hilbert-nonzero}
    There exists some $t\in\units{\Z_p[\zeta_p]}=\units\msO_{K_1}$ such that $\Delta(t):=16t^2(t-1)^2\in\units{\Z_p[\zeta_p]}$ $($i.e., $E_t:y^2=x(x-1)(x-t)$ is a Legendre curve over $\Z_p[\zeta_p])$ and
    \[\pfrac{\zeta_p,t/(t-1)^2}\mfp\neq1\in\mu_p.\]
\end{proposition}
\begin{proof}
    Consider some $a,m\in\Z$ with $p\nmid am$, and observe that
    \[\pfrac{\zeta_p,m^p+ap}\mfp=\pfrac{\zeta_p,1+a\inv[p]mp}\mfp\pfrac{\zeta_p,m^p}\mfp=\pfrac{\zeta_p,1+a\inv[p]mp}\mfp=\inv[{a\inv[p]m}]\zeta_p=\zeta_p^{-a\inv m}\neq1,\]
    with the second-to-last equality following from Remark~\ref{rem:hilb-comp}, because $\Tr_{\Q_p(\zeta_p)/\Q_p}$ is multiplication by $p-1$ on $\Q_p$. Thus, we may construct our desired $t\in\Z_p[\zeta_p]$ by solving $t(t-1)^{-2}=m^p+ap$. 

    In general, the solution of $t(t-1)^{-2}=c$ is
    \begin{equation}\label{eqn:t quadratic formula}
        t = \frac{(2c+1)\pm \sqrt{4c+1}}{2c}.
    \end{equation}
    In our case, $c = m^p+ap$, which is a unit$\pmod p$, so $2c$ is invertible in $\Z_p[\zeta_p]$.
    
    Thus this is solvable exactly when $4(m^p+ap)+1$ is a square in $\Z_p[\zeta_p]$. By Hensel's lemma, this will hold when $4(m^p+ap)+1\equiv 4m+1\pmod p$ is a nonzero square mod $p$, and, in fact, working mod $p$ will imply there is a solution in $\Z_p$.
    \begin{itemize}
        \item If $p>3$, we may choose an $m\in\Z$ so that $m\equiv3\cdot\inv4\pmod p$ (and choose $a$ arbitrarily as long as $p\nmid a$), and then a $t$ satisfying $t/(t-1)^2=m^p+ap$ is guaranteed to exist.

        We also need to check that $\Delta(t)=16t^2(t-1)^2\in\units{\Z_p[\zeta_p]}$. Because $t/(t-1)^2=m^p+ap=:c$ is a unit, this is equivalent to $t$ being a unit. By Equation (\ref{eqn:t quadratic formula}), $t$ is a unit if and only if $(2c+1)\pm\sqrt{4c+1}$ is nonzero modulo $\mfp$. Since $4c+1$ is a nonzero square mod $\mfp$, the expression $(2c+1)\pm\sqrt{4c+1}$ takes on two distinct values mod $\mfp$, and so at least one of them will be nonzero mod $\mfp$.
        \item If $p=3$, we may take $(a,m)=(1,-1)$,  so $4(m^p+ap)+1=3^2$ is a square in $\Z_3[\zeta_3]$. Hence, there will again exist some $t$ satisfying $t/(t-1)^2=m^p+ap$.

        In this case, $t=(5\pm3)/4$; this is a unit, which makes $\Delta(t)$ also a unit, for either possible value.
        \qedhere
    \end{itemize}
\end{proof}
Proposition~\ref{prop:exists-t-hilbert-nonzero} will suffice for showing that classes in the image of
$$\sigma\colon H^1\p{\Z\left[\frac{1}{2p}\right],\Q_p/\Z_p}\too\pprimary{\Br'}(\eey_0(2)[1/p])$$
do not extend over $\eey_0(2)$. In general, however, an element of $\pprimary{\Br'}(\eey_0(2)[1/p])$ is a sum of such a class along with one pulled back from $\pprimary{\Br'}(\Z[\frac{1}{2p}])$. Because of this, it will be useful to also have access to Legendre curves over $\Z_p[\zeta_p]$ for which the relevant Hilbert symbol \textit{does} vanish.

\begin{proposition}\label{prop:exists-t-hilbert-zero}
    There exists some $t\in \units{\Z_p[\zeta_p]}=\units\msO_{K_1}$ such that $\Delta(t):=16t^2(t-1)^2\in\units{\Z_p[\zeta_p]}$ $($i.e., $E_t:y^2=x(x-1)(x-t)$ is a Legendre curve over $\Z_p[\zeta_p])$ and
    \[\pfrac{\zeta_p,t/(t-1)^2}\mfp=1\in\mu_p.\]
\end{proposition}
\begin{proof}
    In this case, the main observation is that $\pfrac{\zeta_p,m^p}\mfp=1$ for all integers $m$, so we are interested in solving $t/(t-1)^2=m^p$ in $\Z_p[\zeta_p]$. Suppose $p\nmid m$. Then, $t/(t-1)^2=m^p$ has a solution if and only if $4m^p+1$ is a square in $\Z_p[\zeta_p]$. Hensel's lemma tells us that this will happen if $4m^p+1\equiv4m+1\pmod p$ is a nonzero square mod $p$. For $p>3$, we may take $m=3\cdot\inv4\pmod p$ and then argue that we can find a resulting $t$ such that $\Delta(t)$ is a unit exactly as in the proof of Proposition~\ref{prop:exists-t-hilbert-nonzero}. 

    If $p=3$, we instead use \cite[Proof of Lemma 8.3]{AM}, which computes that
    \[\pfrac{\zeta_3,t}\mfp=\zeta_3^{-2-b-\frac{b^2}2}\tand\pfrac{\zeta_3,t-1}\mfp=\zeta_3^{b-\frac{b^2}2}\]
    if $t=2+b\pi$, where $\pi:=1-\zeta_3\in\Z_3[\zeta_3]$. We may take $t=2+2\pi$, so $b=2$, and observe that
    \[\pfrac{\zeta_3,t/(t-1)^2}\mfp=\zeta_3^{-6}\left/\p{\zeta_3^0}^2\right.=\inv[6]\zeta_3=1.\]
    At the same time, both $t$ and $t-1$ are units in $\Z_3[\zeta_3]$, so $\Delta(t)$ is as well.
\end{proof}

\subsection{Showing no class in $\boldsymbol{\pprimary{\Br'(\eey_0(2)[1/p])}}$ extends over $\boldsymbol{\eey_0(2)}$}
Let $\sigma$ denote the so-named section constructed in Theorem~\ref{thm:sect-construct}. We first give two proofs that no class in the image of $\sigma$ extends over $\eey_0(2)$; the second of these proofs was suggested to us by the anonymous referee.

\begin{proposition}\label{prop:section-no-extend}
    Let $p\geq 3$ and consider any nonzero $\alpha\in H^1(\Z[\frac{1}{2p}],\Q_p/\Z_p)$. Then, $\sigma(\alpha)\in\Br'(\eey_0(2)[1/p]) =\Br'(\eey_0(2)[\frac{1}{2p}])$ does not extend over $\eey_0(2)$; i.e., it is not in the image of $\pprimary{\Br'}(\eey_0(2))\to\pprimary{\Br'}(\eey_0(2)[1/p])$.
\end{proposition}
\begin{proof}[First proof of Proposition~\ref{prop:section-no-extend}]
    By Remark~\ref{rem:extend-argument-simplification}, we may assume without loss of generality that $\alpha=\alpha_1=[\msO_{F_1}]$. By Proposition~\ref{prop:exists-t-hilbert-nonzero}, we may choose some $t\in\Z_p[\zeta_p]$ so that $E:y^2=x(x-1)(x-t)$ is a Legendre curve over $\Z_p[\zeta_p]$ and so that
    \begin{equation}\label{eqn:hilb-nonvanish-1}
        \pfrac{\zeta_p,t/(t-1)^2}{\mathfrak p}\neq1\in\mu_p(\Q_p(\zeta_p)).
    \end{equation}
    The result follows from Proposition~\ref{prop:not extend}.
\end{proof}

\begin{proof}[Second proof of Proposition~\ref{prop:section-no-extend}]
    The below proof was suggested to us by the anonymous referee, inspired by the work of Voorde's master's thesis \cite{voorde}. It avoids the use of Propositions~\ref{prop:not extend} and~\ref{prop:exists-t-hilbert-nonzero} as well as Remark~\ref{rem:extend-argument-simplification}.

    For $t\in\units\Z_p$ (with $t-1\in\units\Z_p$ as well),  consider the Legendre curve $E_t\colon y^2=x(x-1)(x-t)$ with associated map $f_t\colon\spec(\Z_p)\to\eey_0(2)$. 
    By Remark~\ref{rem:class-extend-strat}, it suffices to construct a $t$ such that $\pull f_t(\sigma(\alpha))\neq0\in\Br(\Q_p)$. 
    It follows from the definition of $\sigma$ that $\pull f_t\sigma(\alpha)\in\Br(\Q_p)$ is the class of the cyclic algebra $(\alpha,s)$, where $s\coloneqq t/(t-1)^2$.
    Let $L/\Q_p$ be the cyclic extension, say of degree $p^n$, represented by the class $\alpha\vert_{\Q_p}\in H^1(\Q_p,\Q_p/\Z_p)$. Then, $\pull f_t\sigma(\alpha)=[(\alpha,s)]\in\Br(\Q_p)$ is trivial if and only if $s\in\Nm(\units L)$. 
    Thus, it suffices to construct some $t$ such that
    \[t\in\units\Z_p, \quad t-1\in\units\Z_p,\tand s=t/(t-1)^2\not\in\Nm(\units L).\]
    Remark~\ref{rem:alphan} shows that $L/\Q_p$ is a nontrivial subextension of $\Q_p(\zeta_{p^{n+1}})/\Q_p$ and so is (totally) ramified.  Local class field theory gives an isomorphism $\units\Q_p/\Nm(\units L)\iso\Gal(L/\Q_p)$ identifying the image of $\units\Z_p$ with the inertia subgroup of $\Gal(L/\Q_p)$, so we conclude that $\Nm(\units L)\subset\units\Q_p$ is a subgroup of index $p^n$ whose image under $\units\Q_p\simeq p^\Z\by\units\Z_p\onto\units\Z_p$ is a proper subgroup, necessarily of index divisible by $p$. Since $\units\Z_p\cong\mu_{p-1}(\Z_p)\by\Z_p$ has a unique subgroup of index $p$, we conclude that it suffices to construct some $t\in\units\Z_p$ such that $t-1\in\units\Z_p$ and $t/(t-1)^2\not\in\mu_{p-1}(\Z_p)\by p\Z_p$. 

    Note that, since $-1\in\mu_{p-1}(\Z/p^2)$, some element among $1,2,\dots,p-1\in\punits{\Z/p^2}$ must \textit{not} be a $\supst{(p-1)}$ root of unity. Let $r$ be the smallest such element (so $r-1\in\mu_{p-1}(\Z/p^2)$), let $\zeta\in\mu_{p-1}(\Z_p)$ be the unique lift of $r-1$, and set $t\coloneqq\zeta+1$. Then, $t\equiv r\bmod p^2$ and $t^{p-1}\neq1$. Furthermore, $(t-1)^{p-1}=\zeta^{p-1}=1$, so $s^{p-1}=t^{p-1}\equiv r^{p-1}\not\equiv1\bmod p^2$. Finally, we note that the image of the subgroup $\mu_{p-1}(\Z_p)\by p\Z_p$ under the map $\units\Z_p\onto\punits{\Z/p^2}$ is $\mu_{p-1}(\Z/p^2)$, and so conclude that $s\not\in\mu_{p-1}(\Z_p)\by p\Z_p$, completing the proof.
\end{proof}

\begin{proposition}\label{prop:constant-class-no-extend}
    Consider any nonzero $\beta\in\pprimary{\Br'}(\Z[\frac{1}{2p}])\subset\pprimary{\Br'}(\eey_0(2)[1/p])$. Then, $\beta$ does not extend over $\eey_0(2)$.
\end{proposition}
\begin{proof}
    Remark~\ref{rem:Z[1/2]-point} shows that $\eey_0(2)$ has a $\Z[\frac{1}{2}]$-point, say $x\colon\spec(\Z[\frac{1}{2}])\to\eey_0(2)$. Let $x_p\colon\spec(\Z[\frac{1}{2p}])\to\eey_0(2)[1/p]$ be the corresponding restricted section, and consider the diagram
    \[\begin{tikzcd}
        \spec\p{\Z[\frac{1}{2p}]}\ar[d]\ar[r, "x_p"]\ar[rr, bend left, equals]&\eey_0(2)[1/p]\ar[r]\ar[d]&\spec\p{\Z[\frac{1}{2p}]}\ar[d]\\
        \spec\p{\Z[\frac{1}{2}]}\ar[r, "x"]&\eey_0(2)\ar[r]&\spec\p{\Z[\frac{1}{2}]}.
    \end{tikzcd}\]
    The commutativity of this diagram shows that
    \[
    \pprimary{\Br'}\p{\Z\left[\frac{1}{2p}\right]}\cap\im\p{\pprimary{\Br'}(\eey_0(2))\to\pprimary{\Br'}\p{\eey_0(2)\left[1/p\right]}}\subset\im\p{\pprimary{\Br'}\p{\Z\left[\frac{1}{2}\right]}\to\pprimary{\Br'}\p{\Z\left[\frac{1}{2p}\right]}}.
      \]
      Since $\Br'(\Z[\frac{1}{2}])=\zmod2$, see \cite[Example 2.19(2)]{AM}, has no $p$-power torsion, any such $\beta$ that extends must be~0.
\end{proof}

\begin{proposition}\label{prop:nothing-extends}
    Choose $\alpha\in H^1(\Z[\frac{1}{2p}],\Q_p/\Z_p)$ and $\beta\in\pprimary{\Br'}(\Z[\frac{1}{2p}])\subset\Br'(\eey_0(2)[1/p])$ such that $\sigma(\alpha)+\beta\in\hphantom{\mkern-2mu}\pprimary{\Br'}(\eey_0(2)[1/p])$ is nonzero. Then, $\sigma(\alpha)+\beta$ does not extend over $\eey_0(2)$.
\end{proposition}

\begin{proof}
    By Proposition~\ref{prop:exists-t-hilbert-zero}, we may fix a $t\in\Z_p[\zeta_p]$ such that $E:y^2=x(x-1)(x-t)$ is a Legendre curve over $\Z_p[\zeta_p]$ and so that
    \begin{equation}\label{eqn:hilb-vanish-1}
        \pfrac{\zeta_p,t/(t-1)^2}{\mathfrak p}=1\in\mu_p(\Q_p(\zeta_p)).
    \end{equation}
    Let $f\colon\spec(\Q_p(\zeta_p))\to\eey_0(2)$ be the point corresponding to $E$ with the usual level 2 structure on Legendre curves. Suppose that $\sigma(\alpha)+\beta$ \textit{does} extend over $\eey_0(2)$, so then $\pull f\p{\sigma(\alpha)+\beta}=0\in\Br(\Q_p(\zeta_p))$. Because of Equation~(\ref{eqn:hilb-vanish-1}), Proposition~\ref{prop:cyclic-alg-hilbert} shows that $\pull f\sigma(\alpha)=0\in\Br(\Q_p(\zeta_p))$. We conclude that $\pull f\beta=0$ as well. Since $\pprimary{\Br}(\Z[\frac{1}{2p}])\iso\pprimary{\Br}(\Q_p)\into\Br(\Q_p(\zeta_p))$ is injective (with the first isomorphism following from \cite[Example~2.19(3) and the discussion above it]{AM}), we conclude that $\beta=0$. Thus, $\sigma(\alpha)$ extends over $\eey_0(2)$, but now Proposition~\ref{prop:section-no-extend} shows this forces $\sigma(\alpha)=0$.
\end{proof}

We have the following corollary, which is Proposition~\ref{prop:p-tors-over-Z[1/2]-i}, restated for the convenience of the reader. 

\begin{corollary}\label{cor:p-tors-over-Z[1/2]-ii}
    For all primes $p\geq 3$, we have $\pprimary{\Br'}(\eey_0(2))=0$. 
\end{corollary}

\section{The 2-primary torsion}\label{sect:2-prim-torsion}
\subsection{The 2-local spectral sequence}

\begin{notation}
    Throughout this section, $S$ will denote a connected, regular noetherian $\Z[\frac{1}{2}]$-scheme. Furthermore, given an abelian group $M$, we write $M_{(2)}:=M\otimes_\Z\Z_{(2)}$, where $\Z_{(2)}$ denotes the integers localized at the prime $(2)$.
\end{notation}

The goal of this section is to compute $\leftindex_2{\Br}(\eey_0(2)_S)$, the $2$-primary torsion in the Brauer group of $\eey_0(2)$ over $S$. As in Section~\ref{sect:p>=3 torsion, p invertible}, our main tool for doing so will be the Hochschild--Serre spectral sequence associated to the $C_2$-cover $\eey(2)\to\eey_0(2)$. Localized at the prime $(2)$, this is the spectral sequence
\[E_2^{i,j}=H^i(C_2,H^j(\eey(2)_S,\G_m))_{(2)}\implies H^{i+j}(\eey_0(2)_S,\G_m)_{(2)}.\]
The relevant part of its $E_2$-page is pictured in Figure~\ref{fig:2-SS take i}.
\begin{figure}[ht]
    \centering
    \[\begin{tikzcd}
    	{\leftindex_2{\Br}'(\eey(2)_S)^{C_2}} \\
    	{\Pic(S)_{(2)}\oplus \mu_2(S)} & {\Pic(S)[2]\oplus \mu_2(S)} & {\Pic(S)/2\oplus \mu_2(S)} \\
    	{\G_m(S)_{(2)} \oplus \Z_{(2)}} & {\mu_2(S)} & {\G_m(S)/2} & {\mu_2(S)}
    	\arrow[from=2-2, to=3-4, "d_2^{1,1}"]
    	\arrow[from=1-1, to=2-3, "d_2^{0,2}"]
    	\arrow[from=2-1, to=3-3, "d_2^{0,1}"]
    \end{tikzcd}\] 
    \caption{The $E_2$-page of the 2-local Hochschild--Serre spectral sequence associated to the $C_2$-cover $\eey(2)\to\eey_0(2)$.}
    \label{fig:2-SS take i}
\end{figure}
\begin{remark}
    The values indicated in Figure~\ref{fig:2-SS take i} can be justified by a combination of Lemmas~\ref{lem:01-lines-of-ss} and~\ref{lem:C2-coh}. 
    \begin{itemize}
        \item For the bottom row ($j=0$), Lemma~\ref{lem:01-lines-of-ss} shows that $H^0(C_2,\G_m(\eey(2)_S))\cong\G_m(S)\oplus\Z$ and that $H^i(C_2,\G_m(\eey(2)_S))\cong H^i(C_2,\G_m(S))$ for $i\ge1$. Since $C_2$ acts trivially on $S$ (and so also on $\G_m(S)$), the claimed values for $E_2^{i0}$ ($i\ge1$) now follow from Lemma~\ref{lem:C2-coh}.
        \item For the middle row ($j=1$), Lemma~\ref{lem:01-lines-of-ss} shows that \[H^i(C_2,H^1(\eey(2)_S,\G_m))\cong H^i(C_2,\Pic(S))\oplus H^i(C_2,\mu_2(S)).\]
          As before, $C_2$ acts trivially on $\Pic(S)$ and $\mu_2(S)$, so the claimed values for $E_2^{i, 1}$ now follow from Lemma~\ref{lem:C2-coh}.
        \item For the top row ($j=2$), we have $E_2^{0, 2}=H^0(C_2,H^2(\eey(2)_S,\G_m))_{(2)}=H^2(\eey(2)_S,\G_m)^{C_2}\otimes\Z_{(2)}$ by definition. Since $S$ is regular and noetherian and $\eey_0(2)_S\to S$ is smooth, $\eey_0(2)_S$ is a regular, noetherian DM stack, so \cite[Proposition 2.5(iii)]{AM} shows that $H^2(\eey(2)_S,\G_m)=\Br'(\eey(2)_S)$. Hence, $E_2^{0, 2}=\Br'(\eey(2)_S)^{C_2}\otimes\Z_{(2)}\cong\leftindex_2{\Br}'(\eey(2)_S)^{C_2}$.
    \end{itemize}
\end{remark}

The only object in Figure~\ref{fig:2-SS take i} which we would like to better understand is $E_2^{02}=\leftindex_2{\Br}'(\eey(2)_S)^{C_2}$. In order to compute this, we will take $C_2$-invariants of the exact sequence \cite[Equation~(5.6)]{AM} (which ultimately comes from \cite[Proposition 3.2]{AM}).

\begin{proposition}
\label{prop:2-torsion C2 sequence}
    There is an exact sequence
    \begin{equation}\label{es:2-torsion C2 sequence}       
        0\too\leftindex_2{\Br}'(S)\oplus\leftindex_2{H^1}(S,\Q/\Z)\too\leftindex_2{\Br}'(\eey(2)_S)^{C_2}\too H^1(S,\mu_2)\oplus\Z/2\xtoo\partial\Br'(S)[2].
    \end{equation}
    Additionally, $H^1(X_S,\mu_2)^{C_2}\cong H^1(S,\mu_2)\oplus\zmod2$, where the $\zmod2$ is generated by the image of\, $t/(t-1)^2$ under the Kummer map $\G_m(X_S)\to H^1(X_S,\mu_2)$. 
\end{proposition}

\begin{proof}
    From \cite[Lemmas 5.10 and 5.11]{AM}, we have a $C_2$-equivariant short exact sequence
    \[0\too\leftindex_2{\Br}'(S)\too\leftindex_2{\Br}'(X_S)\too\wt\rho\vert_{C_2}\otimes H^1(S,\Q_2/\Z_2)\too0,\]
    where $X=\A^1\sm\{0,1\}$ and $\wt\rho$ is the so-named representation appearing in Section~\ref{sect:C2-coh}. This sequence is $C_2$-equivariantly split by pullback along the $C_2$-invariant section $-1\in X(S)$ (recall that $C_2\actson X(S)$ via $t\mapsto1/t$). Thus,
    \begin{equation}\label{eqn:Hi(C2, Br(XS))}
        H^i(C_2,\leftindex_2{\Br}'(X_S))\cong H^i(C_2,\leftindex_2{\Br}'(S))\oplus H^i\p{C_2,\wt\rho\vert_{C_2}\otimes H^1(S,\Q_2/\Z_2)}
    \end{equation}
    for all $i\ge0$. Now, \cite[Equation~(5.6)]{AM} produces an exact sequence $0\to\Br'(X_S)\to\Br'(\eey(2)_S)\to H^1(X_S,\mu_2)\to0$ which is $C_2$-equivariant. Localizing at $2$ and taking $C_2$-invariants (applying the isomorphism (\ref{eqn:Hi(C2, Br(XS))}) and/or Lemma~\ref{lem:C2-coh} where relevant), we obtain the exact sequence
    \[0\too\leftindex_2{\Br}'(S)\oplus H^1(S,\Q_2/\Z_2)\too\leftindex_2{\Br}'(\eey(2)_S)^{C_2}\too H^1(X_S,\mu_2)^{C_2}\xtoo\del\Br'(S)[2].\]

    The last sentence of the proposition follows from Lemma~\ref{lem:C2-inv}.
\end{proof}

\begin{proposition}\label{prop:del-map-description}
    In the exact sequence \eqref{es:2-torsion C2 sequence}, the boundary map $\del\colon H^1(S,\mu_2)\oplus\zmod2\to\Br'(S)[2]$ is computed as follows. On the first factor, it sends $u\in H^1(S,\mu_2)$ to the class of the quaternion algebra $[(-1,u)]\in\Br'(S)$. On the second factor, it sends the generator of\, $\zmod2$ to the class of the quaternion algebra $[(-1,-1)]\in\Br'(S)$.
\end{proposition}
\begin{proof}
    By \cite[Lemma 5.13]{AM}, for $u\in H^1(S, \mu_2)$, we have that $\partial(u)$ equals the Brauer class of the cyclic quaternion algebra $(-1, u)$.  We can use the same method as in the proof of that lemma to compute $\partial$ applied to the $\Z/2$ term.  Indeed, the proof there shows that the boundary map 
    \[
\partial\colon H^1(X_S, \mu_2)^{C_2} \too \Br'(S)[2] 
    \]
    can be computed as follows.  Let $X_S\xrightarrow{\pi} \eey(2)_S\xrightarrow{c} X_S$ be the projection composed with the coarse space map. Let $\chi\in H^1(\eey(2)_S, C_2)$ be the class of the cocycle classifying $\pi$.  Define $s\colon H^1(X_S, \mu_2)\rightarrow \Br'(\eey(2)_S)$ by \[s(u) \coloneqq [(\chi, u)] = [(\chi, c^*u)].\]  
    
    Given $u\in H^1(X_S,\mu_2)$ and $g\in C_2$, $\delta(u)\in H^1(C_2, \leftindex_2{\Br}'(X_S))$ is given by the crossed homomorphism
 \begin{equation}\label{eqn: cocycle evaluate in Br'(X_S)}
     g\longmapsto \pi^*(g(s(u))-s(u))\in \Br'(X_S). 
 \end{equation}

    Let $z\colon S\to X_S$ denote the section $t=-1$. 
     The element in the homomorphism  \eqref{eqn: cocycle evaluate in Br'(X_S)} can be considered an element of $\Br'(S)[2]$ via the isomorphisms
    \[
H^1(C_2, \leftindex_2{\Br}'(X_S))\overset{z^*}\too H^1(C_2, \leftindex_2{\Br}'(S))\xrightarrow{(23)} \Br'(S)[2],
    \] where the last map is evaluation of the crossed homomorphism at the nontrivial element $(23)\in C_2$. We remark $C_2$ acts trivially on $u$ since $u\in  H^1(X_S, \mu_2)^{C_2} $

    Thus, we have that $\partial\colon H^1(X_S, \mu_2)^{C_2}\rightarrow \Br(S)[2]$ is given by $z^*\pi^*((23)(s(u))-s(u))$.  
    
    For $u\in H^1(S, \mu_2)\subset H^1(X_S, \mu_2)^{C_2}$, Lemma 5.13 of \cite{AM} shows this gives $(-1, u)$.
    
    Now we consider the image of $\delta$ on the $\Z/2\Z$ factor, namely on the element $u$ given by the image of $\langle\frac{t}{(t-1)^2}\rangle \in \G_m(X_S)$ in $H^1(X_S, \mu_2)^{C_2}$.  Here it is given by the pullback $z^*\pi^*(23)[(\chi, c^*u)]$.   This is equal to $(-1, z^*\pi^*c^* u)$.
    \[
S\overset{z}\too X_S\overset{\pi}\too\eey(2)_S\overset{c}\too X_S.
    \] 
    By going in the other direction in the commutative diagram 
\[\begin{tikzcd}
	{\G_m(X_S)} & {\G_m(S)} \\
	{H^1(X_S, \mu_2)} & {H^1(S, \mu_2)}
	\arrow[from=1-1, to=1-2]
	\arrow[from=1-2, to=2-2]
	\arrow[from=1-1, to=2-1]
	\arrow[from=2-1, to=2-2]
\end{tikzcd}\]
and using the fact that $u$ is obtained through the image of $\langle\frac{t}{(t-1)^2}\rangle \in \G_m(X_S)$, we see that $z^*\pi^*c^* u$ is the image of $-\frac{1}{4}\in \G_m(S)$ (by plugging in $t=-1$ in $\frac{t}{(t-1)^2})$ in $H^1(S, \mu_2)$.   Thus the generator of $\Z/2\Z$ is sent to the class of $(-1, -\frac{1}{4}) = (-1, -1)$, as desired.      
\end{proof}

\begin{proposition}\label{prop:Br(Y2)-2-tors-splitting}
    The exact sequence 
    \begin{equation}\label{ses:2Br-split}
        0\too\leftindex_2{\Br}'(S)\oplus\leftindex_2{H^1}(S,\Q/\Z)\too\leftindex_2{\Br}'(\eey(2)_S)^{C_2}\too\ker\del\too0
    \end{equation}
    obtained from the exact sequence \eqref{es:2-torsion C2 sequence} is split.
\end{proposition}
\begin{proof}
    We will construct an explicit splitting map $\phi=\phi_1\oplus \phi_2\colon\leftindex_2{\Br}'(\eey(2)_S)^{C_2}\to\leftindex_2{\Br}'(S)\oplus\leftindex_2{H^1}(S,\Q/\Z)$. We first consider the point $f\colon S\to\eey(2)$ given by the Legendre curve $E\colon y^2=x^3-x$ (with discriminant $32\in\Gamma(S,\msO_S)^\times$) and then set $\phi_1=\pull f\colon\leftindex_2{\Br}'(\eey(2)_S)^{C_2}\to\leftindex_2{\Br}'(S)$. Since $f$ is a section of the structure map $\eey(2)_S\to S$, its pullback map $\phi_1$ is a section to $\leftindex_2{\Br}'(S)\into\leftindex_2{\Br}'(\eey(2)_S)^{C_2}$. Next, we define $\phi_2\colon\leftindex_2{\Br}'(\eey(2)_S)^{C_2}\to\leftindex_2{H^1}(S,\Q/\Z)=H^1(S,\Q_2/\Z_2)$ as the composition
    \[\leftindex_2{\Br}'(\eey(2)_S)^{C_2}\xtoo{\pull\pi}\leftindex_2{\Br}'(X_S)^{C_2}\xto{\ram_0}H^1(S,\Q_2/\Z_2),\]
    where the first arrow is pullback along the Legendre family $\pi\colon X_S\to\eey(2)_S$ and where $\ram_0$ above computes the ramification along the $0$-divisor $S\into\A^1_S\supset X_S$, as in \cite[Proposition 2.14 or Lemmas 5.10 and 5.11]{AM}. We claim that $\phi_2$ is a section to the injection $\iota\colon\leftindex_2{H^1}(S,\Q/\Z)\into\leftindex_2{\Br}'(\eey(2)_S)^{C_2}$. From the proof of Proposition~\ref{prop:2-torsion C2 sequence}, we see that this injection is computed as follows:  
    \begin{enumerate}[label=($\star$),ref=$\star$]
        \item\label{star} Given $\alpha\in\leftindex_2{H^1}(S,\Q/\Z)$, one lets $\beta\in\leftindex_2{\Br}'(X_S)$ be the unique $C_2$-invariant Brauer class whose ramification along each of $0$, $1$ is equal to $\alpha$ and whose pullback along $-1\in X(S)^{C_2}$ is trivial; then, one lets $\iota(\alpha)\in\leftindex_2{\Br}'(\eey(2)_S)^{C_2}$ be the pullback of $\beta$ along the coarse space map $c\colon \eey(2)_S\to X_S$.
    \end{enumerate}
     From this description, one now sees that $\phi_2$ is a section to $\iota$, as claimed. Indeed, given $\alpha\in\leftindex_2{H^1}(S,\Q/\Z)$, if $\beta\in\leftindex_2{\Br}'(X_S)^{C_2}$ is as in~\eqref{star}, then $\phi_2(\iota(\alpha))=\ram_0(\pull\pi\pull c\beta)=\ram_0(\beta)=\alpha$ because $c\circ\pi=\id_{X_S}$ and because~$\beta$ was chosen to have ramification $\alpha$ along $0$.
\end{proof}

\begin{notation}\label{notn:G-G'}
    \hfill\begin{itemize}
        \item We let $G'=G'(S):=\ker\del$, with $\del$ the so-named map appearing in Proposition~\ref{prop:2-torsion C2 sequence}. By Proposition~\ref{prop:del-map-description}, we can describe $G'$ as the subgroup $G'\subset H^1(S,\mu_2)\oplus\{\pm1\}$ of pairs $(u,\pm1)$ such that $[(-1,u)]=[(-1,\pm1)]\in\Br'(S)$.
        \item We let $G=G(S)$ denote the subgroup $G\subset H^1(S,\mu_2)$ consisting of the elements $u$ such that $[(-1,u)]=0\in\Br'(S)$ as in \cite[Section 9]{AM}.
    \end{itemize}
    We remark that $G'$ (resp.\ $G$) above is functorial; a morphism $S\to T$ of (connected, regular, noetherian) $\Z[\frac{1}{2}]$-schemes induces a morphism $G'(T)\to G'(S)$ (resp.\ $G(T)\to G(S)$).
\end{notation}

\begin{remark}\label{rem:G'-splits}
    The natural inclusion $G\into G'$, $u\mapsto(u,+1)$ induces a split exact sequence
    \[0\too G\too G'\too\{\pm1\}\too0,\]
    with splitting map $\{\pm1\}\to G'$ sending $\pm1\mapsto(\kappa_2(\pm1),\pm1)$, where $\kappa_2:\colon \G_m(S)\to H^1(S,\mu_2)$ is the Kummer map.
\end{remark}

\subsection{Determining the differentials of the spectral sequence}\label{sec6.2}
We wish to compute the differentials on the $E_2$-page of the 2-local Hochschild--Serre spectral sequence associated to the $C_2$-cover $\eey(2)\to\eey_0(2)$. Recall that this is the spectral sequence
\begin{equation}\label{ss:2-local HS C2}
    E_2^{i,j}=E_2^{i,j}(S)=H^i(C_2,H^j(\eey(2)_S,\G_m))_{(2)}\implies H^{i+j}(\eey_0(2)_S,\G_m)_{(2)}
\end{equation}
pictured in Figure~\ref{fig:2-SS take i}. We write
\[d_r^{i,j}=d_r^{i,j}(S)\colon E_r^{i,j}(S)\too E_r^{i+r,j+(r-1)}(S)\]
for the differential on the $\supth{r}$ page, emphasizing the base scheme $S$ if need be. By comparing this spectral sequence with the analogous spectral sequence
\[{}'E_2^{i,j}=H^i(S_3,H^j(\eey(2)_S,\G_m))_{(2)}\implies H^{i+j}(\eem_{1,1,S},\G_m)_{(2)}\]
(with differentials denoted by ${}'d_r^{i,j}$) coming from the $S_3$-cover $\eey(2)_S\to\eem_{1,1,S}$ (see \cite[Section 9]{AM}), we may understand most of the differentials $d_r^{i,j}$ in our spectral sequence by leveraging the computations of the differentials ${}'d_r^{i,j}$ in \cite[Section 9]{AM}. Indeed, the natural transformation $(-)^{S_3}\to(-)^{C_2}$ of functors induces a morphism $({}'E,{}'d)\to(E,d)$ of spectral sequences.

\begin{lemma}\label{lem:d201 and d211}
    Recall that $S$ denotes a connected, regular noetherian $\Z[\frac{1}{2}]$-scheme. The differential $d_2^{0,1}$ in the spectral sequence \eqref{ss:2-local HS C2} vanishes. The differential $d_2^{1,1}$ is always surjective; it is furthermore an isomorphism if $\Pic(S)[2]=0$.
\end{lemma}
\begin{proof}
    By comparing Figure~\ref{fig:2-SS take i} with \cite[Figure 3]{AM}, we see that the morphism $({}'E,{}'d)\to(E,d)$ of spectral sequences induces commutative diagrams
    \[\commsquare{\Pic(S)_{(2)}\oplus\mu_2(S)}{{}'d_2^{0,1}}{\G_m(S)/2}=={\Pic(S)_{(2)}\oplus\mu_2(S)}{d_2^{0,1}}{\G_m(S)/2}\tand\commsquare{\Pic(S)[2]\oplus\mu_2(S)}{{}'d_2^{1,1}}{\mu_2(S)}=={\Pic(S)[2]\oplus\mu_2(S)}{d_2^{1,1}}{\mu_2(S).}\]
    The result follows from \cite[Proposition 9.2 and Corollary 9.5]{AM}.
\end{proof}
\begin{corollary}
    The differential $d_3^{0,2}$ vanishes, so $E_3^{0,2}=E_\infty^{0,2}$.
\end{corollary}
\begin{proof}
    By Lemma~\ref{lem:d201 and d211}, the differential $d_2^{1,1}$ is surjective, so $E_3^{3,0}=0$. Thus, $d_3^{0,2}\colon E_3^{0,2}\to E_3^{3,0}$ must be the zero map. This is the last possibly nonzero differential the $(0,2)$ object participates in, so we conclude that $E_3^{0,2}=E_\infty^{0,2}$.
\end{proof}

To compute the differential $d_2^{0,2}$, we use that the spectral sequence (\ref{ss:2-local HS C2}) is functorial in $S$.

We will also use Theorem~\ref{thm:alg-closed-comp}, which implies that $\Br'(\eey_0(2)_{\bar{k}})\cong \Z/2$, where $\bar{k}$ is an algebraically closed field of characteristic not 2.  

\begin{lemma}\label{lem:d202-diff-comp}
   If\, $\Pic(S)/2=0$, the differential 
   \[
d_2^{0,2}\colon\leftindex_2{\Br}(\eey(2)_{S})^{C_2}\too \Pic(S)/2\oplus \mu_2(S) = \mu_2(S) 
   \]
   vanishes. 
\end{lemma} 
\begin{proof}
    We can pull back the whole sequence along a geometric point $\spec(\bar k)\to S$ in order to get a commutative square
\[\commsquare{\leftindex_2{\Br}(Y(2)_S)^{C_2}}{d_2^{0,2}}{\mu_2(S)}{}={\leftindex_2{\Br}(Y(2)_{\bar k})^{C_2}}{}{\mu_2(\bar k)\rlap{.}}\]

On the $E_2$-page of the 2-local spectral sequence for $\eey_0(2)_{\bar k}$, we are taking $S=\bar{k}$, and the $(1, 1)$-term is killed by the differential and the $(2, 0)$-term vanishes.  Thus the morphism $\Br(\eey_0(2)_{\bar k})\rightarrow E_{\infty}^{0, 2}$ is an isomorphism.     

Using $S=\bar{k}$ in the exact sequence of Proposition~\ref{prop:2-torsion C2 sequence} shows that $\leftindex_2{\Br}(Y(2)_{\bar k})^{C_2}\iso\zmod2$, so the above square looks like
\[\commsquare{\leftindex_2{\Br}(Y(2))^{C_2}}{d_2^{0,2}}{\mu_2(S)}{}={\zmod2}{}{\mu_2(\bar k)\rlap{,}}\]
and we are interested in understanding the map $\zmod2\to\mu_2(\bar k)$.

    By Theorem~\ref{thm:alg-closed-comp}, since $\ch \overline{k}\neq 2$, we have $\Br(Y_0(2)_{\bar{k}})\cong\zmod2$.  Then the 2-local spectral sequence for $\eey_0(2)_{\bar k}$ implies that the bottom map $\Z/2\Z\rightarrow \mu_2(\bar k)$ must be zero, or otherwise the only nonzero term in the spectral sequence would be killed after the second page, in which case $\Br(Y_0(2)_{\bar k}) = 0$, which is impossible.  Then by the commutativity of the square above, this implies $d_2^{0, 2}=0$, as desired. 
\end{proof}

In conclusion, if $\Pic(S)=0$, then the $E_2$-page of the 2-local Hochschild--Serre spectral sequence we are considering is as indicated in Figure~\ref{fig:2-SS} (see Proposition~\ref{prop:Br(Y2)-2-tors-splitting} for the $E_2^{0,2}$-term).

\begin{figure}[ht]
    \centering
    \[\begin{tikzcd}
    	{\leftindex_2{\Br}'(S)\oplus \leftindex_2{H^1}(S, \Q/Z)\oplus G'} \\
    	{\mu_2(S)} & { \mu_2(S)} & {\mu_2(S)} \\
    	{\G_m(S)_{(2)} \oplus \Z_{(2)}} & {\mu_2(S)} & {\G_m(S)/2} & {\mu_2(S)}
    	\arrow[from=2-2, to=3-4, "\cong "]
    	\arrow[from=1-1, to=2-3, "0"]
    	\arrow[from=2-1, to=3-3, "0"]
    \end{tikzcd}\] 
    \caption{The $E_2$-page of the 2-local Hochschild--Serre spectral sequence associated to the $C_2$-cover $\eey(2)_S\to\eey_0(2)_S$, if $\Pic(S)=0$.}
    \label{fig:2-SS}
\end{figure}

\subsection{Computing the 2-primary torsion} 

\begin{assumption}
    Assume, for the rest of the section, that $\Pic(S)=0$.
\end{assumption}

By our computation of the spectral sequence (\ref{ss:2-local HS C2}) in Section~\ref{sec6.2}, the terms on the second diagonal on the $E_\infty$-page are 
\begin{align*}
    E_\infty^{0,2}=\leftindex_2{\Br}'(\eey(2)_S)^{C_2}, \quad E_\infty^{1,1}=0, \quad E_\infty^{2,0}=\G_m(S)/2.
\end{align*}
Thus we have a short exact sequence
\begin{equation}\label{ses:2-SS-output}
    0\too \G_m(S)/2 \too \leftindex_2{\Br}(\eey_0(2)_S)\too \leftindex_2{\Br}(\eey(2)_S)^{C_2}\too 0. 
\end{equation}

Theorem 9.1 of \cite{AM} produces an analogous short exact sequence for $\leftindex_2{\Br}((\eem_{1,1})_S)$. The covering map $\eey_{0}(2)_S\rightarrow (\eem_{1, 1})_S$ induces the following map of short exact sequences: 
\[\begin{tikzcd}
	0 & {\G_m(S)/2} & {\leftindex_2{\Br}((\eem_{1,1})_S)} & {\leftindex_2{\Br}(\eey(2)_S)^{S_3}} & 0 \\
	0 & {\G_m(S)/2} & {\leftindex_2{\Br}(\eey_{0}(2)_S)} & {\leftindex_2{\Br}(\eey(2)_S)^{C_2}} & 0
	\arrow[from=1-1, to=1-2]
	\arrow[from=1-2, to=1-3]
	\arrow[from=2-1, to=2-2]
	\arrow[from=2-2, to=2-3]
	\arrow[from=1-3, to=1-4]
	\arrow[from=1-4, to=1-5]
	\arrow[from=2-3, to=2-4]
	\arrow[from=2-4, to=2-5]
	\arrow[from=1-2, to=2-2]
	\arrow[from=1-3, to=2-3]
	\arrow[from=1-4, to=2-4].
\end{tikzcd}\]

The snake lemma identifies the cokernels of the two rightmost vertical arrows, so we obtain
\begin{equation}
    \label{eqn: m11 to y02}
0\too\leftindex_2{\Br}((\eem_{1,1})_S)\too\leftindex_2{\Br}(\eey_0(2)_S)\too \coker[\leftindex_2{\Br}(\eey(2)_S)^{S_3}\too \leftindex_2{\Br}(\eey(2)_S)^{C_2}]\too0\rlap{.} 
\end{equation}

\begin{lemma}
\label{lem: m11 y02 split}
The short exact sequence \eqref{eqn: m11 to y02} is split.
\end{lemma} 
\begin{proof}
Recall that the morphism $f\colon \eey_0(2)_S\rightarrow (\eem_{1,1})_S$ is a degree 3 finite cover.  Thus we can define a norm map $f_*\G_m\rightarrow \G_m$ by defining it to be the determinant function over any affine open on $(\eem_{1,1})_S$ for which $f_*\mathcal{O}_{\eey_0(2)_S}$ is finite and locally free.  Precomposing with the natural map $\G_m\rightarrow f_*\G_m$, we see that a function $g\in \G_m((\eem_{1,1})_S)$ is sent to $g^3$.  Thus the norm map induces a morphism 
\[
H^2((\eey_0(2))_S, \G_m)\cong H^2((\eem_{1,1})_S, f_*\G_m)\too H^2((\eem_{1,1})_S, \G_m)
\]
such that precomposition with the pullback $H^2((\eem_{1,1})_S, \G_m)\rightarrow H^2((\eey_0(2))_S, \G_m)$ is multiplication by 3 on $H^2((\eem_{1,1})_S, \G_m)$.  Therefore, the 2-torsion exact sequence of Brauer groups is split, as desired.
\end{proof}

It remains to compute $\coker[\leftindex_2{\Br}(\eey(2)_S)^{S_3}\to \leftindex_2{\Br}(\eey(2)_S)^{C_2}]$.  We recall that the split exact sequence of Proposition~\ref{prop:Br(Y2)-2-tors-splitting} along with Remark~\ref{rem:G'-splits} and the analogous split exact sequence for $(\eem_{1, 1})_S$, given in \cite[Proposition 5.14]{AM}, give compatible isomorphisms $\leftindex_2{\Br}(\eey(2)_S)^{S_3} \cong \leftindex_2{\Br}(S)\oplus G$ and $\leftindex_2{\Br}(\eey(2)_S)^{C_2} \cong$ $\hphantom{\mkern-2mu}\leftindex_{2}{\Br}(S) \oplus \leftindex_2{H^1}(S, \Q/\Z) \oplus G \oplus \Z/2$.    Thus we have 
\begin{equation}
\label{eqn: coker}
\coker\left[\leftindex_2{\Br}(\eey(2)_S)^{S_3}\too \leftindex_2{\Br}(\eey(2)_S)^{C_2}\right] \cong \leftindex_2{H^1}(S, \Q/\Z) \oplus \Z/2. 
\end{equation}

Thus we can compute the 2-torsion of $\Br_{\eey_0(2)_S}$ in terms of the 2-torsion of $\Br_{(\eem_{1,1})_S}$.  We recall the latter computation. 

\begin{proposition}[\textit{cf.} {\cite[Proposition 9.7]{AM}}] 
    Let $P$ be a set of prime numbers including $2$, and denote by $\Z_P \subset \Q$ the subset of all fractions where the denominator is only divisible by primes in $P$. Then 
\[
\Br\p{(\eem_{1,1})_{\Z_P}} \cong \Br(\Z_P)~~\oplus \bigoplus_{\substack{p\in P\cup \{-1\}, \\ p\equiv 3\pmod 4}}\Z/2 ~~\oplus \bigoplus_{\substack{p\in P,\\ p\not\equiv 3\pmod 4}}\Z/4.
\]
\end{proposition} 

Combining this result with Lemma~\ref{lem: m11 y02 split} and the isomorphism (\ref{eqn: coker}), we arrive at the following conclusion. 

\begin{proposition}
\label{prop: 2-torsion of Y02}
    Let $P$ be a set of prime numbers including $2$, and denote by $\Z_P \subset \Q$ the
subset of all fractions where the denominator is only divisible by primes in $P$. Then 
\[
\leftindex_2{\Br}\p{\eey_0(2)_{\Z_P}} \cong \leftindex_2{H^1}(\Z_P, \Q/\Z) \oplus \Z/2 \oplus \leftindex_2{\Br}(\Z_P)~~\oplus \bigoplus_{\substack{p\in P\cup \{-1\}, \\ p\equiv 3\pmod 4}}\Z/2 ~~\oplus \bigoplus_{\substack{p\in P,\\ p\not\equiv 3\pmod 4}}\Z/4.
\]
\end{proposition}
Finally, we make the above proposition more explicit in the case that $\Z_P=\Z[\frac{1}{2}]$.

\begin{lemma}\label{lem:H1Q2/Z2}
  We have 
    \[H^1\p{\Z\sq{\frac12},\frac\Q\Z}\simeq\frac\Z{2\Z}\oplus\frac{\Q_2}{\Z_2}.\]
\end{lemma}
\begin{proof}
    We follow a similar strategy to the one employed in Lemma~\ref{lem:H1Qp/Zp}.
    Because cohomology commutes with colimits, we have
    \[H^1\p{\Z\left[\frac{1}{2}\right],\Q/\Z}=\dirlim_nH^1\p{\Z\left[\frac{1}{2}\right],\zmod n}.\]
    We will show that $H^1(\Z[\frac{1}{2}],\zmod n)=0$ if $n$ is odd, while $H^1(\Z[\frac{1}{2}],\zmod{2^k})\cong\zmod2\oplus\zmod{2^k}$ if $k\ge1$. After showing this, taking the above direct limit will give the claim.

    For $n$ any positive integer, \cite[Example 11.3]{milneLEC} gives the first isomorphism below:
    \[H^1\p{\Z\left[\frac{1}{2}\right],\zmod n}\simeq H^1_{\t{Group}}\p{\etpi_1\p{\Z\left[\frac{1}{2}\right]},\zmod n}=\ctsHom\p{\etpi_1\p{\Z\left[\frac{1}{2}\right]},\zmod n}.\]
    At the same time, $\etpi_1(\Z[\frac{1}{2}])=:G$ is the Galois group of the maximal extension $K/\Q$ unramified away from~$2$. By class field theory, $G^{\t{ab}}\simeq\units\Z_2\cong\zmod2\oplus\Z_2$. Thus,
    \[\ctsHom\p{\etpi_1\p{\Z\left[\frac{1}{2p}\right]},\zmod n}=\ctsHom(\zmod2\oplus\Z_2,\zmod n).\]
    Hence, if $n$ is odd, then $H^1(\Z[\frac{1}{2}],\zmod n)\simeq\ctsHom(\zmod2\oplus\Z_2,\zmod n)=0$. On the other hand, if $n=2^k$ for some $k\ge1$, then $H^1(\Z[\frac{1}{2}],\zmod{2^k})\simeq\ctsHom(\zmod2\oplus\Z_2,\zmod{2^k})\simeq\zmod2\oplus\zmod{2^k}$. The claim follows.
\end{proof}

\begin{corollary}\label{cor:2-tors-over-Z[1/2]}
  We have
  \[\leftindex_2{\Br}(\eey_0(2))\cong\Q_2/\Z_2\oplus(\zmod2)^{\oplus4}\oplus\zmod4.\]
\end{corollary}
\begin{proof}
    Apply Proposition~\ref{prop: 2-torsion of Y02} to $P=\{2\}$, using Lemma~\ref{lem:H1Q2/Z2} to compute $\leftindex_2{H^1}(\Z[\frac{1}{2}],\Q/\Z)$ and \cite[Example 2.19(2)]{AM} to compute $\Br(\Z[\frac{1}{2}])\cong\zmod2$.
\end{proof}

\medskip

\section{The Brauer group of \texorpdfstring{$\boldsymbol{\eey_0(2)}$}{Y\_0(2)} over an algebraically closed field}\label{sect:coarse-space}

Following \cite{Shin}, we will compute the Brauer group of $\eey_0(2)$ over an algebraically closed field $k$ whose characteristic is not 2.  

\subsection{Preliminaries}
Recall (Proposition~\ref{prop:y0(2)-cms}) that the coarse space of $\eey_0(2)$ is $Y_0(2)=\A^1\sm\{0\}=\spec(\Z[\frac{1}{2}][s,\inv s])$. Our first task in the current section will be to show that, away from the point $s=-1/4$, $\eey_0(2)$ is a trivial $C_2$-gerbe over its coarse space.

\begin{lemma}\label{lem:taut-family-y0(2)}
    Let $\pi\colon\eey_0(2)\to Y_0(2)$ denote the coarse space map of Proposition~\ref{prop:y0(2)-cms}. Consider the family of elliptic curves
    \begin{equation}\label{eqn:tautological s family}
        E:y^2=x\p{x^2-(8s+2)x+4s(1+4s)}
    \end{equation}
    over $Y_0(2)\sm\{-1/4\}$, equipped with $2$-torsion point $P=(0,0)$. The pair $(E,P)$ defines a section $Y_0(2)\sm\{-1/4\}\to\eey_0(2)$ of\, $\pi$.
\end{lemma}
\begin{proof}
    One can compute that the family $E\to Y_0(2)\sm\{-1/4\}$ has discriminant $1024s^2(1+4s)^3$, which is an invertible function on $Y_0(2)\sm\{-1/4\}$. Thus, $E$ really defines a family of elliptic curves, and it is visibly clear that $P=(0,0)$ is a nonzero 2-torsion section of this family, so $(E,P)$ defines a map $f\colon Y_0(2)\sm\{-1/4\}\to\eey_0(2)$. It remains to show that this is a section of $\pi$, \textit{i.e.}, that its $s$-invariant is the inclusion $Y_0(2)\sm\{-1/4\}\into Y_0(2)$. In order to compute its $s$-invariant, we first remark that over $Y_0(2)[\sqrt{1+4s}]:=\spec(\Z[\frac{1}{2}][s,\inv s,\sqrt{1+4s}])$, the family $E$ becomes
    \[y^2=x(x-\alpha)(x-\bar\alpha),\quad\twhere\alpha=1+4s-\sqrt{1+4s},\]
    which, by Proposition~\ref{prop:Y(2)-cms}, gives a point of $Y(2)$ with $t$-invariant
    \[t\coloneqq\frac{\alpha-0}{\bar\alpha-0}=\frac{1+2s-\sqrt{1+4s}}{2s}.\]
    Thus, $(E,P)$ has $s$-invariant
    \[\frac t{(t-1)^2}=s,\]
    as desired.
\end{proof}

\begin{notation}
    Let $Y_0(2)^\circ:=Y_0(2)\sm\{-1/4\}$ and $\eey_0(2)^\circ:=\eey_0(2)\by_{Y_0(2)}Y_0(2)^\circ$. Let $\pi^\circ\colon \eey_0(2)^\circ\to Y_0(2)^\circ$ be the restriction of the coarse space map.
\end{notation}

\begin{lemma}\label{lem:same s-inv => locally iso}
    Let $S$ be an arbitrary $\Z[\frac{1}{2}]$-scheme. Suppose that $(E,P),(E',P')\in\eey_0(2)^\circ(S)$ are two families of elliptic curves, each equipped with a point of order $2$, whose $s$-invariants agree, as morphisms $S\to Y_0(2)^\circ$. Then, there exists an \'etale cover $S'\to S$ supporting an isomorphism $(E_{S'},P)\cong(E'_{S'},P')$.
\end{lemma}
\begin{proof}
    Let $s\in H^0(S,\msO_S)$ be the coarse space map $S\to Y_0(2)^\circ\subset\A^1$ induced by $(E,P)$. It suffices to consider a single $(E,P)\in\eey_0(2)^\circ(S)$ and show that, after passage to an \'etale cover, it becomes isomorphic to the family
    \[E':y^2=x\p{x^2-(8s+2)x+4s(1+4s)}\quad\twith P'=(0,0)\]
    of Lemma~\ref{lem:taut-family-y0(2)}. For this, we may assume that $S=\spec(R)$ is affine. By \cite[Proof of Proposition 4.5]{AM}, there exists an (affine) \'etale cover $\spec(R')\to\spec(R)$ such that $E_{R'}$ can be put in Legendre form. Once it is in this form, we can perform a linear change of variables to move $P$ to the origin. Thus, we may assume without loss of generality that, over $S=\spec(R)$, we have 
    \[E:y^2=x(x-1)(x-t)\tand P=(0,0).\]
    We seek an isomorphism $\phi\colon E\to E'$ with $\phi(0,0)=(0,0)$. We first remark that, by assumption, the quantities
    \[
        t-1, \quad s=\frac t{(t-1)^2}, \quad 4s+1
    \]
    are all units in $R=H^0(S,\msO_S)$. Thus, we see that $(t+1)^2=(4s+1)(t-1)^2$ is also a unit. Since $R$ is a $\Z[\frac{1}{2}]$-algebra, this means that $2(t+1)\in\units R$ as well. Hence, $R':=R[\sqrt{2t+2}]$ is an \'etale $R$-algebra. Now, over $R'$, we can define our desired isomorphism $\phi\colon E_{R'}\iso E'_{R'}$ via 
    \[\phi(x,y):=\p{u^2x,u^3y},\quad\twhere u:=\frac{\sqrt{2t+2}}{t-1}\in\punits{R'}.\qedhere\]
\end{proof}

\begin{lemma}\label{lem:Aut(E,P)=C_2}
    Let $S$ be an arbitrary $\Z[\frac{1}{2}]$-scheme, and fix some $(E,P)\in\eey_0(2)^\circ(S)$. Then, $\ul\Aut(E,P)\simeq\ul{C_2}$ as sheaves on $S_{\mathrm{\acute{e}t}}$.
\end{lemma}
\begin{proof}
  Because $E$ is of finite presentation over $S$, we may reduce to the case that $S$ is noetherian. We may furthermore assume that $S$ is connected and then prove that $\Aut(E/S,P)=\{\pm1\}$. In fact, \cite[Section~13.1.11]{olsson} shows that the automorphism scheme $\ul\Aut(E)$ (so also the subscheme $\ul\Aut(E,P)\into\ul\Aut(E)$) is (formally) unramified over $S$. Because $S_{\mathrm{red}}\into S$ is a finite-order thickening (since $S$ is noetherian), we deduce from this that the natural map $\Aut(E/S,P)\to\Aut(E/S_{\t{red}},P)$ is injective, so we may assume that~$S$ is reduced. Because $S$ is assumed connected, by passage to its irreducible components, we may assume that~$S$ is integral. That is, it suffices to prove that $\Aut(E/S,P)=\{\pm1\}$ whenever $S$ is integral.
    
    Because $P$ is 2-torsion, it is clear that $\{\pm1\}\subset\Aut(E/S,P)$. If $\Aut(E/S)=\{\pm1\}$, we are done, so suppose this is not the case. Let $\eta\in\Aut(E/S)$ be an automorphism fixing $P$. Then, at worse, $\eta$ swaps the other two 2-torsion sections of $E$, so $\eta^2$ acts trivially on $E[2]$. By \cite[Corollary 2.7.2]{katz-mazur}, this forces $\eta^2=\pm1$. Thus, either $\eta=\pm1$, in which case we win, or $E/S$ supports an automorphism of order $4$. Suppose the latter. Then, $E$ must have constant $j$-invariant $1728\in\Gamma(S,\msO_S)$. By Lemma~\ref{lem:j-inv from s}, the $s$-invariant of $(E,P)$ then satisfies $256(s+1)^3/s^2=1728$. Since $\Gamma(S,\msO_S)$ is a domain by assumption, one can check that this forces $s=-1/4$ or $s=2$; since $(E,P)$ has $s$-invariant landing in $Y_0(2)^\circ=Y_0(2)\sm\{-1/4\}$ by assumption, we conclude that $(E,P)$ must have constant $s$-invariant $2$ and that $3\in\units{\Gamma(S,\msO_S)}$. Indeed, since $2=s\neq -1/4$, we have $8\neq -1$, so $9\neq 0$ at every point on $S$, and thus 3 is a unit on $S$. 
    Now, Lemma~\ref{lem:same s-inv => locally iso} shows that families with the same $s$-invariant are always locally isomorphic. Combined with Lemma~\ref{lem:taut-family-y0(2)}, this shows that (after \'etale base change), we may assume that $E/S$ is of the form
    \[y^2=x(x^2-18x+72)=x(x-6)(x-12)\]
    with $P=(0,0)$. Making the linear change of variables $x\squigto x+6$ shows that $E$ is isomorphic to the curve
    \[y^2=(x+6)x(x-6)=x^3-36x\]
    with $P=(-6,0)$. Because $3$ is invertible on $S$, one can check that the automorphism group of this curve is $\mu_4(S)$ with any $\zeta\in\mu_4(S)$ acting via $[\zeta](x,y)=(\zeta^2x,\zeta y)$. From this description, it is finally clear that $\Aut(E/S,P)=\{\pm1\}$.
\end{proof}

\begin{lemma}[\textit{cf.} {\cite[Lemma 3.2]{Shin}}]\label{lem:trivial gerbe} 
    The coarse space map $\pi^\circ\colon \eey_0(2)^\circ\to Y_0(2)^\circ$ is a trivial $C_2$-gerbe. Hence, $\eey_0(2)^\circ\simeq BC_{2,Y_0(2)^\circ}$.
\end{lemma} 
\begin{proof}
    By Lemma~\ref{lem:taut-family-y0(2)}, the restriction $\pi^\circ\colon \eey_0(2)^\circ\to Y_0(2)^\circ$ of the coarse space map has a section. Thus, it suffices to show that $\eey_0(2)^\circ$ is a $C_2$-gerbe over $Y_0(2)^\circ$, \textit{i.e.}, that families with the same $s$-invariant are (\'etale-)locally isomorphic and each have automorphism sheaf $\ul{C_2}$. The first of these is Lemma~\ref{lem:same s-inv => locally iso}, while the second is Lemma~\ref{lem:Aut(E,P)=C_2}.
\end{proof}

Our goal is now to compute the Brauer group of $\eey_0(2)$ over an algebraically closed fields by exploiting this gerbe structure along with knowledge of the coarse space. 

To get started with our computation, we begin with the following immediate consequence of \cite[Lemma 3.3]{Shin}.

\begin{lemma}\label{lem:brauer is subgroup}
    Let $k$ be an algebraically closed field with $\ch k\neq2$. Then, $\Br'(\eey_0(2)_k)$ is a subgroup of $\zmod2\oplus\zmod2$.
\end{lemma}
\begin{proof}
    By \cite[Proposition 2.5(iv)]{AM}, the restriction map $\Br'(\eey_0(2)_k)\to\Br'(\eey_0(2)_k^\circ)$ is an injection. Since $\eey_0(2)_k^\circ\simeq BC_{2,Y_0(2)^\circ}$ by Lemma~\ref{lem:trivial gerbe} and $\Pic(Y_0(2)^\circ_k)=0$, Lemma~3.3 of \cite{Shin} gives the second equality in
    \[\Br'\p{\eey_0(2)_k^\circ}=\Br'\p{BC_{2,Y_0(2)^\circ_k}}=\G_m\p{Y_0(2)^\circ_k}/2=\zmod2\oplus\zmod2.\]

    The last equality holds because $\G_m(Y_0(2)^\circ_k) = (k[s, s^{-1}, (s+1/4)^{-1}])^{\times}=\units k\by s^\Z\by(s+1/4)^\Z.$
\end{proof}

\subsection{The computation} 

\begin{setup}
    Fix an algebraically closed field $k$ with $\Char k\neq2$.
\end{setup}

We compute $\Br(\eey_0(2)_k)$ following the strategy of \cite[Section 4]{Shin}. We first observe that $\Br(\eey_0(2)_k)$ is a $2$-torsion group by Lemma~\ref{lem:brauer is subgroup}. For $n\ge1$, we have $\Br(\eey_0(2)_k)=H^2(\eey_0(2)_k,\G_m)[2^n]$, so the Kummer sequence $0\to\mu_{2^n}\to\G_m\to\G_m\to0$ gives rise to the short exact sequence
\begin{equation}\label{ses:Kummer-2n}
    0\too\Pic\p{\eey_0(2)_k}/2^n\too H^2\p{\eey_0(2)_k,\mu_{2^n}}\too\Br\p{\eey_0(2)_k}\too0.
\end{equation}
\begin{remark}
    For computing $\Br(\eey_0(2)_k)$, although the exact sequence (\ref{ses:Kummer-2n}) shows that it would suffice to compute $H^2(\eey_0(2)_k,\mu_{2^n})$ for $n=1$ alone, we will instead find it easier to take a colimit of the exact sequence (\ref{ses:Kummer-2n}) as $n$ varies, computing the short exact sequence
    \[\begin{tikzcd}
        0\ar[r]&\dirlim_n\Pic\p{\eey_0(2)_k}/2^n\ar[d, equals]\ar[r]&\dirlim_nH^2(\eey_0(2)_k,\mu_{2^n})\ar[d, equals]\ar[r]&\Br\p{\eey_0(2)_k}\ar[d, equals]\ar[r]&0\\
        &\zmod4&\zmod4\oplus\zmod2&\zmod2
    \end{tikzcd}\]
    in the proof of Theorem~\ref{thm:alg-closed-comp}. In fact, Remark~\ref{rem:Pic Y0(2)} and Theorem~\ref{thm:alg-closed-comp} will show that the two colimits above stabilize already at the $n=2$ terms.
\end{remark}

\begin{remark}\label{rem:Pic Y0(2)}
    By \cite[Theorem 1.1]{Niles} (or Theorem~\ref{thm:app-main} if $\Char k=3$), $\Pic(\eey_0(2)_k)=\zmod4$ if $\Char k\neq2$.
\end{remark}

In light of the above, we are interested in understanding $H^2(\eey_0(2)_k,\mu_{2^n})$. To ease notation, set $U:=Y_0(2)_k^\circ$. Following \cite{Shin}, let $Z\into Y_0(2)_k=\A^1_k\sm\{0\}$ denote the complement of $U$, with its induced reduced structure (so $Z\simeq\spec(k)$). With this notation setup, we have the following commutative diagram of Cartesian squares whose vertical morphisms are coarse space maps: 

\[\begin{tikzcd}
	{\yzt^\circ_k} & \yzt_k & {\yzt_Z} \\
	U & Y_0(2)_k & Z\rlap{.}
	\arrow[from=1-1, to=1-2]
	\arrow[from=1-3, to=1-2]
	\arrow["{\pi_Z}", from=1-3, to=2-3]
	\arrow["i", from=2-3, to=2-2]
	\arrow["\pi", from=1-2, to=2-2]
	\arrow["{\pi^\circ}"', from=1-1, to=2-1]
	\arrow["j"', from=2-1, to=2-2]
\end{tikzcd}\]  

As in \cite[Equation~(4.1.4)]{Shin}, the distinguished triangle 
\[
j_!j^*R\pi_*\mu_{2^n}\too R\pi_*\mu_{2^n} \too i_*i^* R\pi_*\mu_{2^n} \overset{+1}\too 
\]
combined with the fact that $j^*R\pi_*\mu_{2^n}\cong R\pi^\circ_*\mu_{2^n}$ gives the following long exact sequence:  
    \begin{align}
        \label{coarse les} 
        \begin{split} 
    0 &\too  H^0\p{\A^1_k\backslash\{0\}, j_!R\pi_*^\circ \mu_{2^n}} \too H^0\p{\yzt_k, \mu_{2^n}}\too H^0\p{Z, i^*R\pi_*\mu_{2^n}} \\ 
    & \too  H^1\p{\A^1_k\backslash\{0\}, j_!R\pi_*^\circ \mu_{2^n}} \too H^1\p{\yzt_k, \mu_{2^n}}\too H^1\p{Z, i^*R\pi_*\mu_{2^n}} \\ 
    & \too H^2\p{\A^1_k\backslash\{0\}, j_!R\pi_*^\circ \mu_{2^n}} \too H^2\p{\yzt_k, \mu_{2^n}}\too H^2\p{Z, i^*R\pi_*\mu_{2^n}} \\ 
    & \too \cdots.
\end{split}
\end{align}

    To compute the terms $H^s(\A^1_k\backslash\{0\}, j_!R\pi_*^\circ \mu_{2^n})$, we will use the following result, in which $j$ is allowed to be slightly more general.
    
\begin{lemma}
\label{lem:mu coh of az}
    Fix an integer $r\ge1$. Let $j\colon U\hookrightarrow \A^1_k\backslash\{0\}$ be the inclusion of the complement of\, $r$ distinct $k$-points, and let $\l$ be a positive integer relatively prime to the characteristic of\, $k$. Then, we have a Galois-equivariant exact sequence
    \[0\too\mu_\l(k)^{r-1}\too H^1\p{\A^1_k\sm\{0\},j_!\mu_\l}\too \zmod\l\too0,\]
    while $H^s(\A^1_k\sm\{0\},j_!\mu_\l)=0$ for $s\neq1$. In particular, $H^1(\A^1_k\sm\{0\},j_!\mu_\l)\cong\mu_\l(k)^{r-1}\oplus\zmod\l$ as abelian groups.
\end{lemma}
\begin{proof} 
    Let $i\colon Z\rightarrow \A^1_k\backslash\{0\}$ be the complement of $j\colon U \rightarrow \A^1_k\backslash\{0\}$, with its reduced induced scheme structure. The long exact sequence associated to $0\rightarrow j_!j^*\mu_\l\rightarrow \mu_\l\rightarrow i_*i^*\mu_\l\rightarrow 0$ gives 
    \[
        0\too H^0\p{\az, j_!\mu_\l} \too \mu_\l(k)\too \mu_\l(k)^{\oplus r} \too H^1\p{\az, j_!\mu_\l} \too H^1\p{\az, \mu_\l} \too 0. 
    \]
    The last term is 0 because the dimension of $Z$ is 0.
    The morphism $\mu_\l(k)\rightarrow \mu_\l(k)^{\oplus r}$ is the diagonal, so we have $H^0(\az, j_!\mu_\l)=0$.  Using the Kummer sequence, we have $H^1(\az, \mu_\l)=\G_m(\az)/\ell = \zmod\l$.  This gives the claimed exact sequence for $H^1$ (which is automatically split as a sequence of $\F_\l$-vector spaces). In higher degrees, $H^s(\az, j_!\mu_\l)= 0$ 
    for $s\ge 2$, because the affine \'etale cohomological dimension is the dimension of the scheme.
\end{proof}  

We can now compute the terms in the leftmost column of our long exact sequence (\ref{coarse les}).

\begin{lemma}[\textit{cf.} {\cite[Lemma 4.3]{Shin}}]\label{lem:derived coh of az}
    Let $m$ be any positive integer. In the setup of Lemma~\ref{lem:mu coh of az}, let ${\pi^\circ\colon BC_{m,U}\to U}$ be the trivial $C_m$-gerbe over $U$.  We have 
    \[
H^n\p{\az, j_!R\pi_*^\circ \mu_\l} \cong \begin{cases}
    \hfill 0,\hfill  & n = 0, \\ 
    \hfill \mu_\l^{\oplus r-1} \oplus \Z/\l,\hfill & n = 1, \\ 
    \hfill\mu_{\gcd(m, \l)}^{\oplus r-1} \oplus \Z/\gcd(m, \l),\hfill & n \ge 2,
\end{cases}
    \]
    where the above isomorphisms are not necessarily Galois equivariant.
\end{lemma}
\begin{proof}
  Given an object $\mc C$ in the derived category of sheaves on some scheme, let $h^t(\mc C)$ denote its $\supth{t}$ cohomology sheaf. Since the functor $j_!$ is exact, we observe that $h^t(j_!R\push\pi^\circ\mu_\l)\simeq j_!h^t(R\push\pi^\circ\mu_\l)=j_!R^t\push\pi^\circ\mu_\l$. Thus, the hypercohomology spectral sequence \cite[Application 5.7.10]{weibel}
  $$E_2^{st}=H^s\p{\A^1_k\sm\{0\},h^t(\mc C)}\implies H^{s+t}\p{\A^1_k,\mc C}$$
  applied to $\mc C=j_!R\push\pi^\circ\mu_\l$ becomes
    \[E^{st}_2=H^s\p{\A^1_k\sm\{0\},j_!R^t\push\pi^\circ\mu_\l}\implies H^{s+t}\p{\A^1_k\sm\{0\},j_!R\push\pi^\circ\mu_\l}.\]
   We claim that the $E_2$-page of this sequence has only one nonzero column, the $s=1$ column. Indeed, \cite[Lemma B.1]{Shin} shows that
    \[R^t\push\pi^\circ\mu_\l\simeq\Threecases{\mu_\l}{t=0,}{\mu_\l[m]}{t>0\t{ odd,}}{\mu_\l/m}{t>0\t{ even.}}\]
    Combining this with Lemma~\ref{lem:mu coh of az} and the observations that $\mu_\l[m]=\mu_{\gcd(m,\l)}\cong\mu_\l/m$, we conclude that $E^{st}_2=0$ if $s\neq1$, while
    \begin{align}\label{eqn:hypcoh-ss-conclusion}
        E^{1,t-1}_2
        =H^1\p{\A^1_k\sm\{0\},j_!R^{t-1}\push\pi^\circ\mu_\l}
        &\simeq\Threecases{H^1\p{\A^1_k\sm\{0\},j_!\mu_\l}}{t=1,}{H^1\p{\A^1_k\sm\{0\},j_!(\mu_\l[m])}}{t>1\text{ even,}}{H^1\p{\A^1_k\sm\{0\},j_!(\mu_\l/m)}}{t>1\t{ odd,}} \nonumber\\[.4ex]
        &\cong\Twocases{\mu_\l(k)^{\oplus(r-1)}\oplus\zmod\l}{t=1,}{\mu_{\gcd(m,\l)}^{\oplus(r-1)}(k)\oplus\zmod{\gcd(m,\l)}}{t>1.}
    \end{align}
    Since this $E_2$-page is a concentrated in a single column, we must have $H^n(\A^1_k\sm\{0\},j_!R\push\pi^\circ\mu_\l)\simeq E_2^{1,n-1}$, from which the claim follows.
\end{proof}

Taking $m=2$, $r=1$, and $\l=2^n$ in Lemma~\ref{lem:derived coh of az}, we obtain the terms in the left column of the long exact sequence (\ref{coarse les}).  Next we wish to compute the terms in the right column, namely the groups $H^s(Z, i^*R\pi_*\mu_{2^n})$.  Recall that in our situation, $i\colon Z\rightarrow \az$ is the inclusion of the point $-\frac{1}{4}$ (so $Z\simeq\spec(k)$).  Consider the stacky point of $\eey_0(2)$ that gets mapped to $-\frac{1}{4}$.  It comes from the elliptic curve corresponding to $t=-1$ in the Legendre family, \textit{i.e.}, to $E:y^2=x(x-1)(x+1)=x^3-x$ equipped with the order 2 point $P=(0,0)$. If $\ch(k)>3$, then this curve has automorphism group $\zmod4$ generated by $(x,y)\mapsto(-x,iy)$; if $\ch(k)=3$, then it has an automorphism group of order $12$ generated by the previous automorphism along with the order 3 automorphism $(x,y)\mapsto(x+1,y)$. Thus, in any case, the group of automorphisms preserving $P$ is simply $\zmod4$. \\

As in \cite[Section 4.4]{Shin}, for any $n\ge0$, we have the following sequence of isomorphisms, explained below: 
\begin{equation}\label{eqn:Z-coh}
    H^s\p{Z, i^*R\pi_*\mu_{2^n}}\overset{(a)}\cong i^*R^s\pi_*\mu_{2^n}\overset{(b)}\cong H^s\p{\eey_0(2)_Z, \mu_{2^n}}  \overset{(c)}\cong H^s((B \Z/4)_k, \mu_{2^n}) \overset{(d)}\cong H^s(\Z/4, \mu_{2^n}(k)). 
\end{equation}
Above, $(a)$ holds because $Z\cong\spec(k)$; $(b)$ holds by proper base change, see \cite[Theorem 1.3]{OLSSON200593}; $(c)$ holds because $(\eey_0(2)_Z)_{\mathrm{red}}\cong(B\zmod4)_k$, and so one can apply topological invariance of \'etale cohomology, see \cite[\href{https://stacks.math.columbia.edu/tag/03SI}{Tag 03SI}]{stacks-project} (note that $\mu_{2^n,\eey_0(2)_Z}\vert_{B\zmod4}\simeq\mu_{2^n,B\zmod4}$ as \'etale sheaves because $\mu_{2^n}$ is \'etale over $\eey_0(2)_Z$); and $(d)$ follows from the Hochschild--Serre spectral sequence applied to the $\zmod4$-cover $\spec(k)\to(B\zmod4)_k$. In Equation~(\ref{eqn:Z-coh}), $\Z/4$ acts trivially on $\mu_{2^n}(k)$, so 
\begin{equation}\label{eqn:Z-coh-ii}
     H^s(Z,\pull iR\push\pi\mu_{2^n}) \cong H^s(\zmod4,\mu_{2^n}(k))\cong\Threecases{\mu_{2^n}(k)}{s=0,}{\mu_{2^n}(k)[4]}{s>0\t{ odd,}}{\mu_{2^n}(k)/4}{s>0\t{ even}}
\end{equation} 
by Fact~\ref{cyclic coh}.

The final ingredient we will use is the computation of the terms in the middle column of the sequence (\ref{coarse les}) in degrees~0 and 1.  First, we have $H^0(\yzt_k, \mu_{2^n}) \cong H^0(\az, \mu_{2^n})\cong \mu_{2^n}$.  By \cite[Proposition 2.9]{AM}, we have a split short exact sequence 
\[
0\too \G_m(\yzt_k)/\G_m(\yzt_k)^{2^n}\too H^1(\yzt_k, \mu_{2^n}) \too \Pic(\yzt_k)[2^n]\too 0.
\]
Because $k$ is algebraically closed, using Proposition~\ref{prop:y0(2)-cms}, we have $\G_m(\yzt_k)=\G_m(\A^1_k\sm\{0\})=s^\Z$, where $\A^1_k\sm\{0\}=\spec(k[s,\inv s])$. By Remark~\ref{rem:Pic Y0(2)}, we have $\Pic(\yzt_k)\cong\Z/(4)$. Thus, for $n\ge 2$, this gives an isomorphism $H^1(\yzt_k, \mu_{2^n})\cong \Z/(2^n) \oplus \Z/(4)$.  

Thus for $n\ge 2$, the long exact sequence (\ref{coarse les}) becomes the following:  
\begin{align*}
\label{coarse les filled in}
\begin{split}
    0 &\too  0 \too \mu_{2^n}(k)\too \mu_{2^n}(k) \\ 
    & \too \Z/2^n \too \Z/2^n \oplus \Z/4 \too \mu_4(k) \\ 
    & \too \Z/2 \too H^2(\yzt_k, \mu_{2^n})\too \mu_{2^n}(k)/4 \\ 
    & \too \cdots. 
\end{split}
\end{align*}
Looking at these terms, we conclude that we must have an exact sequence of the form 
\[
0\too \Z/2\too H^2(\yzt_k, \mu_{2^n})\too \mu_{2^n}(k)/4.
\]
We will need to understand how this sequence behaves functorially with $n$.  We recall that the $\Z/2$ here is computed in Equation~(\ref{eqn:hypcoh-ss-conclusion}) in the proof of Lemma~\ref{lem:derived coh of az} as 
\[
H^2(\az, j_!R\pi_*^\circ \mu_{2^n})  \cong H^1(\az,j_!R^1\push\pi^\circ\mu_{2^n}) \cong H^1(\az, j_!\mu_2) \cong \Z/2, 
\]
where the penultimate congruence uses the fact that
\[
R^1\pi_*^\circ \mu_{2^n}\cong \mu_{2^n}[2]\cong \mu_2. 
\]
This is significant because this implies that as $n$ varies, the corresponding morphism between these cohomology groups is an isomorphism. 
 On the other hand, we see that as we go from $n$ to $n+2$, the corresponding morphism $\mu_{2^n}/(4)\rightarrow \mu_{2^{n+2}}/(4)$ is 0. These observations allow us to prove the main theorem of this section.

\begin{theorem}[\textit{cf.} {\cite[Theorem 1.1]{Shin}}]\label{thm:alg-closed-comp}
    Let $k$ be an algebraically closed field with $\ch k\neq 2$. Then, we have $\Br'(\eey_0(2)_k)\cong \Z/2$. 
\end{theorem}
\begin{proof}
    By the discussion preceding this theorem statement, we have a commutative diagram
    \[\begin{tikzcd}
    	0 & {\Z/2} & {H^2(\eey_0(2)_k, \mu_{2^n})} & {\Z/4} \\
    	0 & {\Z/2} & {H^2(\eey_0(2)_k, \mu_{2^{n+2}})} & {\Z/4}
    	\arrow["f_n", from=1-3, to=2-3]
    	\arrow[from=1-3, to=1-4]
    	\arrow["0", from=1-4, to=2-4]
    	\arrow[from=2-3, to=2-4]
    	\arrow[from=1-1, to=1-2]
    	\arrow[from=1-2, to=1-3]
    	\arrow[from=2-1, to=2-2]
    	\arrow[from=2-2, to=2-3]
    	\arrow["\id", from=1-2, to=2-2]
    \end{tikzcd}\]
    with exact rows. Observe that the morphism $f_n$ indicated above both lands in $\zmod2\subset H^2(\eey_0(2)_k,\mu_{2^{n+2}})$ and restricts to the identity map on $\zmod2$. Thus, it gives a splitting of the top left exact sequence, so
    \[H^2(\eey_0(2)_k,\mu_{2^n})\cong\zmod2\oplus G,\quad\t{where }G:=\im\p{H^2(\eey_0(2)_k,\mu_{2^n})\to\zmod4}.\]
    Since $\Pic(\eey_0(2)_k)=\zmod4$, and so $\Pic(\eey_0(2)_k)/2^n=\Pic(\eey_0(2)_k)$, the embedding $\Pic(\eey_0(2)_k)\into H^2(\eey_0(2)_k,\mu_{2^n})$ (see the exact sequence (\ref{ses:Kummer-2n})) shows that $H^2(\eey_0(2)_k,\mu_{2^n})$ must contain an element of order $4$. Since $H^2(\eey_0(2)_k,\mu_{2^n})\cong\zmod2\oplus G$, we must have $G=\zmod4$ and so $H^2(\eey_0(2)_k,\mu_{2^n})\cong\zmod2\oplus\zmod4$. From the exact sequence (\ref{ses:Kummer-2n}), we conclude that $\Br'(\eey_0(2)_k)=\zmod2$. 
\end{proof}

To finish this section, we would like to say a bit more about the generator of $\Br'(\eey_0(2)_k)\cong\Z/2$. In particular, we will prove that it is not represented by a quaternion algebra over $\eey_0(2)_k$.

\begin{proposition}\label{prop:Z/2Z-not-quat}
    Let $\alpha\in\Br(\eey_0(2)_k)$ denote the nontrivial element. Then, $\alpha$ is a sum of classes of quaternion algebras. Equivalently, the image of
    \begin{equation}\label{composition:quat-alg}
        H^1(\eey_0(2)_k,C_2)\otimes H^1(\eey_0(2)_k,\mu_2)\xtoo\smile H^2(\eey_0(2)_k,\mu_2)\too H^2(\eey_0(2)_k,\G_m)=\Br\p{\eey_0(2)_k}
    \end{equation}
    is $0$.
\end{proposition}
\begin{proof}
    Let $f\colon\eey(2)_k\to\eey_0(2)_k$ denote the natural map, and recall from Proposition~\ref{prop:Y(2)-cms} that $\eey(2)_k$ has coarse space $c\colon\eey(2)_k\to X_k=\A^1_k\sm\{0,1\}$. The exact sequence (\ref{ses:2-SS-output}) from Section~\ref{sect:2-prim-torsion} shows that $\pull f\colon\Br(\eey_0(2)_k)\to\Br(\eey(2)_k)$ is injective. Thus, it suffices to show that any element in the image of the composition (\ref{composition:quat-alg}) pulls back to $0\in\Br(\eey(2)_k)$. Since $\ul{C_2}\simeq\mu_2$ over $k$, it suffices to show that
    \begin{equation}\label{containment:blah}
      \im\p{\pull f\colon H^1(\eey_0(2)_k,\mu_2)\to H^1(\eey(2)_k,\mu_2)}\subset\im\p{\pull c\colon H^1(X_k,\mu_2)\to H^1(\eey(2)_k,\mu_2)}, 
    \end{equation}
    since then any quaternion algebra over $\eey_0(2)_k$ pulls back to an element of $\im\p{\pull c\colon\Br(X_k)\to\Br(\eey(2)_k)}=0$ (note that $\Br(X_k)=0$ by Tsen's theorem).

    With this in mind, consider the commutative diagram
    \[\begin{tikzcd}
        0\ar[r]&\G_m\p{\eey_0(2)_{\bar k}}/2\ar[d]\ar[r]&H^1\p{\eey_0(2)_{\bar k},\mu_2}\ar[d, "\pull f"]\ar[r]&\Pic\p{\eey_0(2)_{\bar k}}[2]\ar[d, "\pull f"]\ar[r]&0\\
        0\ar[r]&\G_m\p{\eey(2)_{\bar k}}/2\ar[d, equals]\ar[r]&H^1\p{\eey(2)_{\bar k},\mu_2}\ar[r]&\Pic\p{\eey(2)_{\bar k}}[2]\ar[r]&0\\
        &\G_m\p{X_{\bar k}}/2\ar[r, "\sim"]&H^1\p{X_{\bar k},\mu_2}\rlap{,}\ar[u, "\pull c"]       
    \end{tikzcd}\]
    whose rows are exact. We claim that the map on the Picard groups is trivial; this suffices to show the inclusion (\ref{containment:blah}). By Theorem~\ref{thm:app-main}, we have $\Pic(\eey_0(2)_{\bar k})\cong\zmod4$. At the same time, Propositions~\ref{prop:Y(2)-cms} and~\ref{Propn: BC_n Brauer} show that $\Pic(\eey(2)_{\bar k})\cong\zmod2$. Thus, $\pull f\colon\Pic(\eey_0(2)_{\bar k})\to\Pic(\eey(2)_{\bar k})$ is a map $\zmod4\to\zmod2$ and so necessarily becomes trivial when restricted to $2$-torsion. This completes the proof.
\end{proof}

\section{\texorpdfstring{$\boldsymbol{\Br(\eey_0(2)_k)}$}{Br(Y\_0(2)\_k)} over other fields \texorpdfstring{$\boldsymbol{k}$}{k}}\label{sect:Br-field}

In this section, we wish to compute $\Br(\eey_0(2)_k)$ for fairly general fields $k$. We first remark that using the Hochschild--Serre spectral sequences associated to the covers $X\to\eey(2)\to\eey_0(2)$, as we did when computing \textit{e.g.}~$\Br(\eey_0(2)_\Q)$, would present at least two challenges:
\begin{itemize}
    \item If $\Char k=p\ge3$, our general result Proposition~\ref{thm:p-tors-away-from-p} for computing $p$-primary torsion would not apply.
    \item In any case, computing $2$-primary torsion would require solving the extension problem presented by the short exact sequence (\ref{ses:2-SS-output}).
\end{itemize}
For these reasons, we adopt a new approach in this section. Since we will be working over a field $k$, we can consider the Galois cover $\eey_0(2)_{k^s}\to\eey_0(2)_k$, where $k^s/k$ is a separable closure of $k$, and apply the Hochschild--Serre spectral sequence to it. Because we have computed $\Br(\eey_0(2)_{\bar k})$ in Section~\ref{sect:coarse-space}, this spectral sequence will allow us to compute $\Br(\eey_0(2)_k)$ for many perfect fields $k$.

\begin{remark}
    Despite the warning just given about attempting to do this computation using the covers $X\to\eey(2)\to\eey_0(2)$, the argument in the current section still relies nontrivially on the work in Section~\ref{sect:2-prim-torsion}.
\end{remark}

\begin{lemma}\label{lem:Br-over-field-low-degree}
    Let $k$ be a field of characteristic not $2$. Then, there is an exact sequence
    \[0\too\Br(k)\oplus H^1(k,\Q/\Z)\too\ker\p{\Br\p{\eey_0(2)_k}\too\Br\p{\eey_0(2)_{k^s}}}\too H^1(k,\zmod4)\too H^3(k,\G_m)\oplus H^3(k,\Z).\]
\end{lemma}
\begin{proof}
    This follows from writing down the exact sequence of low-degree terms (see, \textit{e.g.}, \cite[Proposition~6.7.1]{poonen-rat-pts}) associated to the spectral sequence 
    \[E_2^{p,q}=H^p(k,H^q(\eey_0(2)_{k^s},\G_m))\implies H^{p+q}(\eey_0(2)_k,\G_m).\]
    In general, this exact sequence takes the form
    \[0\too E_2^{1,0}\too H^1\too E_2^{0,1}\too E_2^{2,0}\too\ker\p{H^2\too E_2^{0,2}}\too E_2^{1,1}\too E_2^{3,0}.\]
    In the present context, we compute the individual terms as follows.
    Below, note that $H^0(\eey_0(2)_{k^s}, \G_m)=H^0(Y_0(2)_{k^s}, \G_m)=\G_m(k^s)\oplus \Z$. 
    \begin{itemize}\setlength\itemsep{.4em}
        \item $E_2^{1,0}=H^1(k, \G_m(k^s)\oplus \Z)=\Pic(k)\oplus H^1(k,\Z)=0$.
        \item $H^1(\eey_0(2)_k,\G_m)=\Pic(\eey_0(2)_k)\cong\zmod4$ by Theorem~\ref{thm:app-main}.
        \item $E_2^{0,1}=H^0(k,H^1(\eey_0(2)_{k^s},\G_m))\cong\zmod4$ by Theorem~\ref{thm:app-main}. 
        \item $E_2^{2,0}=H^2(k, H^0(\eey_0(2)_{k^s},\G_m))=H^2(k, \G_m(k^s)\oplus\Z)=\Br(k)\oplus H^2(k,\Z)\cong\Br(k)\oplus H^1(k,\Q/\Z)$ with the last isomorphism following from taking cohomology of the exact sequence $0\to\Z\to\Q\to\Q/\Z\to0$.
        \item $H^2(\eey_0(2)_k,\G_m)=\Br(\eey_0(2)_k)$ by Lemma~\ref{lem:coh-Br=Az-Br}.
        \item $E_2^{0,2}=H^0(k, H^2(\eey_0(2)_{k^s}, \G_m))=H^0(k,\Br(\eey_0(2)_{k^s}))$.
        \item $E^{1,1}_2=H^1(k,H^1(\eey_0(2)_{k^s},\G_m))=H^1(k,\Pic(\eey_0(2)_{k^s}))=H^1(k,\zmod4)$. Here, the Galois action on $\Pic(\eey_0(2)_{k^s})\cong\zmod4$ is trivial since Theorem~\ref{thm:app-main} shows that $\Pic(\eey_0(2)_k)\to\Pic(\eey_0(2)_{k^s})$ is an isomorphism (with both groups generated by the Hodge bundle).
        \item $E^{3,0}_2=H^3(k,H^0(\eey_0(2)_{k^s},\G_m))=H^3(k,\G_m)\oplus H^3(k,\Z)$.
    \end{itemize}
    Also note that
    \[
    \ker\p{\Br\p{\eey_0(2)_k}\to H^0\p{k, \Br\p{\eey_0(2)_{k^s}}}}= \ker\p{\Br\p{\eey_0(2)_k}\to\Br\p{\eey_0(2)_{k^s}}}
    \]
    because $H^0(k, \Br(\eey_0(2)_{k^s}))\subseteq \Br(\eey_0(2)_{k^s})$.
\end{proof}

\begin{lemma}\label{lem:Brk-ks-split-surjection}
    Let $k$ be a perfect field of characteristic not $2$. Then, the natural map
    \[\Br\p{\eey_0(2)_k}\too\Br\p{\eey_0(2)_{k^s}}\cong\zmod2\]
    is a split surjection.
\end{lemma}
\begin{proof}
    First note that $\Br(\eey_0(2)_{k^s})\cong\twoprimary{\Br}(\eey_0(2)_{k^s})\cong\zmod2$ by Theorem~\ref{thm:alg-closed-comp}, since $k$ is perfect. By the work in Section~\ref{sect:2-prim-torsion} (see in particular the exact sequence (\ref{ses:2-SS-output})), the map $\Br(\eey_0(2)_k)\to\Br(\eey_0(2)_{k^s})$ extends to the following homomorphism of short exact sequences:
    \begin{equation}\label{seshom-k-ksep}
        \begin{tikzcd}
            0\ar[r]&\G_m(k)/2\ar[r]\ar[d]&\twoprimary{\Br}\p{\eey_0(2)_k}\ar[r]\ar[d]&\twoprimary{\p{\Br\p{\eey(2)_k}}}^{C_2}\ar[d, "f"]\ar[r]&0\\
            0\ar[r]&\G_m(k^s)/2\ar[r]&\twoprimary{\Br}\p{\eey_0(2)_{k^s}}\ar[r, "\sim"]\ar[d, equals]&\twoprimary{\p{\p{\Br\eey(2)_{k^s}}}}^{C_2}\ar[r]&0\\
            &&\Br\p{\eey_0(2)_{k^s}}\rlap{,}
        \end{tikzcd}
    \end{equation}
    where the above-indicated horizontal map is an isomorphism because $\Char k^s\neq2$, so $\G_m(k^s)/2=0$.
    
    We claim that the map $f$ indicated above is surjective. Using the functoriality of exact sequence (\ref{ses:2Br-split}) along with Notation~\ref{notn:G-G'} and Remark~\ref{rem:G'-splits} to compute $\ker\del$, we obtain the following commutative diagram with exact rows: 
    \begin{equation}\label{seshom-k-ksep2}
        \begin{tikzcd}
            0\ar[r]&\twoprimary{\Br}(k)\oplus\twoprimary{H}(k,\Q/\Z)\ar[r]\ar[d]&\twoprimary{\Br}(\eey(2)_k)^{C_2}\ar[r]\ar[d, "f"]&G(k)\oplus\{\pm1\}\ar[r]\ar[d]&0\\
            0\ar[r]&0\ar[r]&\twoprimary{\Br}(\eey(2)_{k^s})^{C_2}\ar[r, "\sim"]&\{\pm1\}\ar[r]&0\rlap{.}
        \end{tikzcd}
    \end{equation}
    Note that $f$ is surjective if and only if the rightmost vertical map above is. Remark~\ref{rem:G'-splits}, applied once to $S=\spec(k^s)$ and then to $S=\spec(k)$, shows the rightmost vertical map is a surjection, with kernel $G(k)$.
    
    This is enough to conclude that $\Br(\eey_0(2)_k)\to\Br(\eey_0(2)_{k^s})=\twoprimary{\Br}(\eey_0(2)_{k^s})\twoprimary{(\Br(\eey(2)_{k^s}))}^{C_2}$ is surjective. In order to show that it is a split surjection, we appeal to Lemma~\ref{lem: m11 y02 split}, or rather, to its proof. The proof of this lemma (in particular, the isomorphism (\ref{eqn: coker})) shows that $\coker\sq{\twoprimary{\Br}(\eey(2)_k)^{S_3}\to\twoprimary{\Br}(\eey(2)_k)^{C_2}}\cong\hphantom{\mkern-3mu}\twoprimary{H}(k,\Q/\Z)\oplus\zmod2$, where this $\zmod2$ is, more canonically, $G'(k)/G(k)$. The above discussion of the diagram (\ref{seshom-k-ksep2}) identifies $G'(k)/G(k)\simeq\twoprimary{\Br}(\eey(2)_{k^s})^{C_2}$, and the first paragraph showed that $\twoprimary{\Br}(\eey(2)_{k^s})^{C_2}\simeq\Br(\eey_0(2)_{k^s})$, so we conclude by observing that Lemma~\ref{lem: m11 y02 split} says that 
    \[\twoprimary{\Br}(\eey_0(2)_k)\to\coker\sq{\twoprimary{\Br}\p{\eey(2)_k}^{S_3}\to\twoprimary{\Br}(\eey(2)_k)^{C_2}}\cong\twoprimary{H}(k,\Q/\Z)\oplus\twoprimary{\Br}\p{\eey_0(2)_{k^s}}\] 
    is split surjective and that $\twoprimary{\Br}(\eey_0(2)_k)$ is a direct summand of $\Br(\eey_0(2)_k)$.
\end{proof}

\begin{notation}\label{notn:Br1}
    For a field $k$ and a scheme or stack $\eex/k$, we set
    \[\Br_1(\eex)\coloneqq\ker\p{\Br(\eex)\to\Br\p{\eex_{k^s}}}.\]
\end{notation}

As a consequence of Lemma~\ref{lem:Brk-ks-split-surjection}, if $k$ is a perfect field of characteristic not $2$, we have $\Br(\eey_0(2)_k)\cong\Br_1(\eey_0(2)_k)\oplus\zmod2$. By Lemma~\ref{lem:Br-over-field-low-degree}, for any field $k$, we also have an exact sequence
\begin{equation}\label{es:Br-over-field}
    0\too\Br(k)\oplus H^1(k,\Q/\Z)\too\Br_1(\eey_0(2)_k)\too H^1(k,\zmod4)\overset{d}\too H^3(k,\G_m)\oplus H^3(k,\Z).
\end{equation}
In what follows, we further analyze this exact sequence in order to completely determine $\Br_1(\eey_0(2)_k)$.

\begin{theorem}\label{thm:Br-over-field}
    Let $k$ be any field of characteristic not $2$. The homomorphism $d$ in the exact sequence \eqref{es:Br-over-field} is zero, and there exists a retraction $\Br_1(\eey_0(2)_k)\to\Br(k)\oplus H^1(k,\Q/\Z)$. Thus,
    \[\Br_1(\eey_0(2)_k)\simeq\Br(k)\oplus H^1(k,\Q/\Z)\oplus H^1(k,\zmod4).\]
    If\, $k$ is perfect, then we further have
    \[\Br(\eey_0(2)_k)\simeq\Br_1(\eey_0(2)_k)\oplus\zmod2\simeq\Br(k)\oplus H^1(k,\Q/\Z)\oplus H^1(k,\zmod4)\oplus\zmod2.\]
\end{theorem}
\begin{proof}
    Recall $X_k=\A^1_k\sm\{0,1\}$ and $Y_0(2)_k=\A^1_k\sm\{0\}$. Let $X_k\to\eey_0(2)_k$ be the morphism corresponding to the Legendre family (with $(0,0)$ as the chosen 2-torsion point), and let $\eey_0(2)_k\to Y_0(2)_k$ be the coarse space map of Corollary~\ref{cor:Y0(2)-cms-general-base}. To ease notation in this argument, write $\bar k\coloneqq k^s$ for the separable closure, and let $G=\Gal(\bar k/k)$. The maps $X_k\to\eey_0(2)_k\to Y_0(2)_k$ induce morphisms between their corresponding Hochschild--Serre spectral sequences, and so induce the following homomorphisms between the induced exact sequences of low-degree terms:
    \begin{equation}\label{cd:big-one-at-end}
        \begin{tikzcd}[column sep=small]
            0\ar[r]&H^2\p{G,\units{\bar k\left[t,\inv t,1/(t-1)\right]}}\ar[r,"\sim"]&\Br_1(X_k)\ar[r]&0\ar[r]&H^3\p{G,\bar k\left[t,\inv t,1/(t-1)\right]^\by}\\
            0\ar[r]&\Br(k)\oplus H^1(k,\Q/\Z)\ar[r]\ar[u]&\Br_1(\eey_0(2)_k)\ar[u]\ar[r]&H^1(k,\zmod4)\ar[r, "d"]\ar[u]&H^3(k,\G_m)\oplus H^3(k,\Z)\ar[u, "\phi"]\\
            0\ar[r]&H^2\p{G,\units{\bar k\left[s,\inv s\right]}}\ar[r, "\sim"]\ar[u, "\sim" sloped]&\Br_1(Y_0(2)_k)\ar[u]\ar[r]&0\ar[u]\ar[r]&H^3\p{G,\bar k\left[s,\inv s\right]^\times}\rlap{.}\ar[u, "\sim" sloped]    
        \end{tikzcd}
    \end{equation}
    By construction, the maps $H^i(G,\bar k[s,\inv s]^\by)\to H^i(G,\bar k[t,\inv t,1/(t-1)]^\times)$, for $i=2,3$, going from the bottom row to the top row above are simply the maps on $H^i(G,-)$ induced by $s\mapsto t/(t-1)^2$. Identifying $\units{\bar k[s,\inv s]}\simeq\units{\bar k}\oplus s^\Z\simeq\units{\bar k}\oplus\Z$ and $\units{\bar k[t,\inv t,1/(t-1)]}\simeq\units{\bar k}\oplus t^\Z\oplus(t-1)^\Z\simeq\units{\bar k}\oplus\Z^2$, the map $s\mapsto t/(t-1)^2$ is identified with the map
    \[\mapdesc{f}{\units{\bar k}\oplus\Z}{\units{\bar k}\oplus\Z^2}{(\alpha,n)}{(\alpha,n,-2n).}\]
    This map is visibly a section of the projection $g\colon(\beta,a,b)\mapsto(\beta,a)$. Write $g\colon\bar k[t,\inv t,1/(t-1)]^\times\to\bar k[s,\inv s]$ for the corresponding map on unit groups. The commutativity of the diagram now gives us the following two conclusions:
    \begin{itemize}
        \item $d\colon H^1(k,\zmod4)\to H^3(k,\G_m)\oplus H^3(k,\Z)$ is the zero morphism.

        Indeed, $\phi\colon H^3(k,\G_m)\oplus H^3(k,\Z)\simeq H^3(G,\bar k[s,\inv s]^\times)\xto{\push f}H^3(G,\bar k[t,\inv t,1/(t-1)]^\times)$ is split by $\push g\colon H^3(G,\bar k[t,\inv t,1/(t-1)]^\times)\to H^3(G,\bar k[s,\inv s])$ and so injective. Since $\phi$ is injective, the commutativity of the top right square of the diagram (\ref{cd:big-one-at-end}) shows that $d=0$.
        \item The map $\Br(k)\oplus H^1(k,\Q/\Z)\to\Br_1(\eey_0(2)_k)$ appearing in the middle row of the diagram (\ref{cd:big-one-at-end}) is a split injection.

        Indeed, the commutativity of the diagram (\ref{cd:big-one-at-end}), along with the previous discussion, shows that it is split by the composition
        \[\Br_1(\eey_0(2)_k)\too\Br(X_k)\simeq H^2\p{G,\units{\bar k\left[t,\inv t,(t-1)\right]}}\overset{\push g}\too H^2\p{G,\units{\bar k\left[s,\inv s\right]}}\simeq\Br(k)\oplus H^1(k,\Q/\Z).\]
    \end{itemize}
    Taken together, these show that
    \[\Br_1(\eey_0(2)_k)\simeq\Br(k)\oplus H^1(k,\Q/\Z)\oplus H^1(k,\zmod4).\]
    Finally, if $k$ is perfect, then Lemma~\ref{lem:Brk-ks-split-surjection} shows that $\Br(\eey_0(2)_k)\simeq\Br_1(\eey_0(2)_k)\oplus\zmod2$.
\end{proof}

\begin{remark}\label{rem:Br-coarse-field}
    In the course of proving Theorem~\ref{thm:Br-over-field}, we showed that $\Br_1(Y_0(2)_k)\simeq H^2(G,\bar k[s,\inv s]^\by)\simeq\Br(k)\oplus H^1(k,\Q/\Z)$. Thus, the conclusion of Theorem~\ref{thm:Br-over-field} can equivalently be written as
    \[\Br_1(\eey_0(2)_k)\simeq\Br_1(Y_0(2)_k)\oplus H^1(k,\zmod4)\]
    for any field $k$ of characteristic not $2$. If $k$ is furthermore perfect, then Tsen's theorem shows that $\Br_1(Y_0(2)_k)\simeq\Br(Y_0(2)_k)$ and we have
    \[\Br(\eey_0(2)_k)\simeq\Br_1(\eey_0(2)_k)\oplus\zmod2\simeq\Br(Y_0(2)_k)\oplus H^1(k,\zmod4)\oplus H^1(k,\zmod2).\]
\end{remark}

\renewcommand{\appendixname}{Appendix. Computing \texorpdfstring{$\boldsymbol{\Pic(\eey_0(2)_S)}$}{Pic(Y\_0(2)\_S)}}

\appendix
\addappendix

In Section~\ref{sect:coarse-space}, we need to know $\Pic(\eey_0(2)_S)$ when $S=\spec(k)$ is an algebraically closed field of characteristic not $2$. In \cite{Niles}, this Picard group was computed for any $S/\Z[1/6]$, but it was not computed in characteristic~$3$. In contrast, the main results of \cite{lopez2023picard} apply over any field, and so could be used to compute $\Pic(\eey_0(2)_S)$ in the cases we need. Still, in the interest of plugging a gap in the literature, we show in this appendix that the ideas going into \cite{Niles,lopez2023picard} apply just as well to compute $\Pic(\eey_0(2)_S)$ for any $S/\Z[\frac{1}{2}]$.

\begin{recall}
    Corollary~\ref{cor:Y0(2)-cms-general-base} constructed a coarse space map $c=c_S\colon \eey_0(2)_S\to Y_0(2)_S:=\A^1_S\sm\{0\}$.
\end{recall}

\begin{notation}\label{notn:hodge-bund}
    Let $S$ be a $\Z[\frac{1}{2}]$-scheme, and let $\pi=\pi_S\colon \eee\to\eey_0(2)_S$ denote the universal family over $\eey_0(2)$. We write $\lambda=\lambda_S:=\push\pi\Omega^1_{\eee/\eey_0(2)}$ to denote the Hodge bundle on $\eey_0(2)$. Note that, if $f_S\colon\eey_0(2)_S\to\eey_0(2)$ denotes the natural projection, then $\lambda_S\simeq\pull f_S\lambda_{\Z[\frac{1}{2}]}$.
\end{notation}
The main theorem of this appendix is the below extension of \cite[Theorem 2.4]{Niles} from $\Z[1/6]$-schemes to $\Z[\frac{1}{2}]$-schemes.

\begin{theorem}\label{thm:app-main}
    Let $S$ be a $\Z[\frac{1}{2}]$-scheme. Then, the map
    \[\mapdesc{}{\zmod4\by\Pic(Y_0(2)_S)}{\Pic(\eey_0(2)_S)}{(n,\msL)}{\lambda^{\otimes n}\otimes\pull c\msL}\]
    is an isomorphism.
\end{theorem}

\begin{remark}
    In the setting of this paper, when $S$ is noetherian and regular, $\Pic(Y_0(2)_S)=\Pic(S)$, so we have an isomorphism $\Z/4\Z \times \Pic(S)\cong \Pic(\eey_0(2)_S)$ in this case.
\end{remark}

We first show that the map in Theorem~\ref{thm:app-main} is well defined. For this, we will make use of the following fact.

\begin{proposition}[\textit{cf.} {\cite[Proposition 6.2]{olsson-G-torsors}}]\label{prop:tame-line-bundles}
    Let $S$ be a scheme and $\mathscr{X}$ a tame stack over $S$, with coarse moduli space $\pi\colon\mathscr{X}\to X$. The pullback functor
    \[\pi^*\colon(\text{line bundles on } X)\too (\text{line bundles on } \mathscr{X})\]
    induces an equivalence of categories between line bundles on $X$ and line bundles $\mathscr{L}$ on $\mathscr{X}$ such that for all geometric points $x\colon\Spec (k) \to \mathscr{X}$, the representation of\, $G_{x}$ corresponding to $x^*\mathscr{L}$ is trivial.
\end{proposition}

\begin{lemma}\label{lem:hodge-bund-4-tors}
    Let $S$ be a $\Z[\frac{1}{2}]$-scheme. Then, the Hodge bundle $\lambda=\lambda_S$ is $4$-torsion; i.e., $\lambda^{\otimes4}\simeq\msO$ on $\eey_0(2)_S$.
\end{lemma}
\begin{proof}
    As remarked in Notation~\ref{notn:hodge-bund}, the Hodge bundle $\lambda_S$ is the pullback of the Hodge bundle on $\eey_0(2)$, so we may assume that $S=\spec(\Z[\frac{1}{2}])$. Lemma~\ref{lem:y0(2)-tame} shows that $\eey_0(2)$ is a tame $\Z[\frac{1}{2}]$-stack, and furthermore it shows that all its geometric points have automorphism groups of order dividing $4$. Thus, $\lambda^{\otimes4}$ induces the trivial representation at any geometric point of $\eey_0(2)$, so Proposition~\ref{prop:tame-line-bundles} shows that $\lambda^{\otimes4}\simeq\pull c\msM$ for some line bundle $\msM$ on $Y_0(2)$. Finally, $Y_0(2)=\spec(\Z[\frac{1}{2}][s,\inv s])$ has trivial Picard group, so $\msM$ (and hence $\lambda^{\otimes4}$) is trivial.
\end{proof}

We can now prove the main result of this appendix.

\begin{proof}[Proof of Theorem~\ref{thm:app-main}]
    In brief, once one has Lemma~\ref{lem:hodge-bund-4-tors} showing that the map in Theorem~\ref{thm:app-main} is well defined, the rest of the proof goes through exactly as in \cite{Niles}. For the sake of completeness, we include details below.

    Let $x\colon B\mu_{4,S}\to\eey_0(2)_S$ be the morphism corresponding to the elliptic curve $E:y^2=x^3-x$ equipped with the point $P=(0,0)$ of exact order $2$ (this has automorphism group scheme $\mu_4$, where $\zeta\in\mu_4$ acts via $(x,y)\mapsto(\zeta^2x,\zeta y)$). Let $\wt x\colon S\to\eey_0(2)$ be the restriction of $x$ along $S\to B\mu_{4,S}$. If $\msL$ is a line bundle on $\eey_0(2)_S$, then $\pull{\wt x}\msL=:\msM$ is a line bundle on $S$ with an action by $\mu_4$. Because $\ul\Aut(\msM)=\G_m$, this action corresponds to some representation $\chi\colon \mu_4\to\G_m$. Identifying $\ul{\zmod4}\iso\ul\hom(\mu_4,\G_m)$ in the natural way, we see that pullback along $x$ produces a map $f\colon \Pic(\eey_0(2)_S)\to\zmod4$. 
    One can check that $f(\lambda)=1\in\zmod4$ by arguing as in \cite[Lemma 2.5]{fulton-olsson}, so $f$ is surjective and split by the Hodge bundle; \textit{i.e.}, $\Pic(\eey_0(2)_S)\simeq\zmod4\oplus\ker f$. It remains to identify $\ker f\simeq\pull c\Pic(Y_0(2)_S)$. Say $\msL\in\ker f$. Then, $\msL$ induces the trivial representation when restricted to $x$. However, by Lemma~\ref{lem:Aut(E,P)=C_2}, away from $x$, every other geometric point of $\eey_0(2)$ has $\{\pm1\}$ as its automorphism group. Since $\msL\in \ker f$, the automorphism $-1$ acts trivially on the fiber of $\msL$ above any point with automorphism group $\mu_4$, but one can argue as in \cite[Corollary 2.4]{fulton-olsson} to show that this means that $-1$ acts trivially on the fiber of $\msL$ over any geometric point. Thus, $\msL\in\pull c\Pic(Y_0(2)_S)$ by Proposition~\ref{prop:tame-line-bundles}, finishing the proof.
\end{proof}


\phantomsection 


\begin{thebibliography}{LMFDB23+++}

\bibitem[AOV08]{AOV}
D.~Abramovich, M.~Olsson, and A.~Vistoli, \emph{Tame stacks in positive
  characteristic}, Ann.\ Inst.\ Fourier (Grenoble) \textbf{58} (2008), no.~4,
  1057--1091, \doi{10.5802/aif.2378}.

\bibitem[Ach24]{achenjang-brauer}
N.~Achenjang, \emph{On {B}rauer groups of tame stacks},
 preprint \arXiv{2410.06217} (2024).

\bibitem[AM20]{AM}
B.~Antieau and L.~Meier, \emph{The {B}rauer group of the moduli stack of
  elliptic curves}, Algebra Number Theory \textbf{14} (2020), no.~9,
  2295--2333, \doi{10.2140/ant.2020.14.2295}.

\bibitem[AMS24]{AMS}
B.~Antieau, L.~Meier, and V.~Stojanoska, \emph{Picard sheaves, local Brauer
  groups, and topological modular forms}, J.~Topol.\ \textbf{17}
  (2024), no.~2, Paper No.~e12333, \doi{10.1112/topo.12333}.

\bibitem[AW67]{coh-of-groups-CF}
M.\,F.~Atiyah and C.\,T.\,C.~Wall, \emph{Cohomology of groups}, in: \emph{Algebraic {N}umber  {T}heory} ({P}roc.\ {I}nstructional {C}onf., {B}righton, 1965),  pp.~94--115, Academic Press, London, 1967.

\bibitem[DLP25]{dilorenzo2022cohomological}
  A.\,Di Lorenzo and R.~Pirisi, \emph{Cohomological invariants and Brauer groups of  algebraic stacks in positive characteristic}, Algebr.\ Geom.\ \textbf{12} (2025), no.~2, 189--236, \doi{10.14231/ag-2025-006}.

\bibitem[Fu15]{fu-etale}
L.~Fu, \emph{Etale cohomology theory}, revised ed., Nankai Tracts
  Math., vol.~14, World Sci.\ Publ., Hackensack, NJ, 2015, \doi{10.1142/9569}. 

\bibitem[FO10]{fulton-olsson}
W.~Fulton and M.~Olsson, \emph{The {P}icard group of {$\mathscr M_{1,1}$}},
  Algebra Number Theory \textbf{4} (2010), no.~1, 87--104,
  \doi{10.2140/ant.2010.4.87}.

\bibitem[Gro66a]{Gr1}
A.~Grothendieck, \emph{Le groupe de {Brauer} : {I.} {Alg\`ebres} {d'Azumaya} et
interpr\'etations diverses}, in: \emph{S\'eminaire Bourbaki : ann\'ees 1964/65 1965/66, expos\'es 277--312},  Exp.\ No.~290, pp.~199--219, S\'eminaire Bourbaki, no.~9, Soc.\ Math.\ France, Paris, 1966.


\bibitem[Gro66b]{Gr2}
\bysame, \emph{Le groupe de {Brauer} : {II.} {Th\'eorie} cohomologique},
in: \emph{S\'eminaire Bourbaki : ann\'ees 1964/65 1965/66, expos\'es 277--312},  Exp.\ No.~297, pp.~287--307,  S\'eminaire Bourbaki, no.~9, Soc.\ Math.\ France, Paris, 1966. 
  
\bibitem[Gro68]{Gr3}
\bysame, \emph{Le groupe de {B}rauer. {III}. {E}xemples et compl\'{e}ments}, in:
  \emph{Dix expos\'{e}s sur la cohomologie des sch\'{e}mas}, pp.~88--188, Adv.\ Stud.\ Pure Math., vol.~3, North-Holland, Amsterdam, 1968.

\bibitem[KM85]{katz-mazur}
N.\,M.~Katz and B.~Mazur, \emph{Arithmetic moduli of elliptic curves}, Ann.\ of Math.\ Stud., vol.~108, Princeton Univ.\ Press, Princeton, NJ,
  1985, \doi{10.1515/9781400881710}.

\bibitem[Lie11]{lieblich2011period}
M.~Lieblich, \emph{Period and index in the {B}rauer group of an arithmetic
  surface} (with an appendix by D.~Krashen), J.~reine angew. Math.\ \textbf{659} (2011), 1--41,  \doi{10.1515/crelle.2011.059}.

\bibitem[LMFDB23]{lmfdb}
The LMFDB Collaboration, \emph{The {L}-functions and modular forms database},
  \url{https://www.lmfdb.org}, 2023, [Online; accessed 19 September 2023].
  
\bibitem[Lop23]{lopez2023picard}
R.~Lopez, \emph{Picard groups of stacky curves},
  preprint \arXiv{2306.08227} (2023).

\bibitem[Mei18]{meier-cs}
L.~Meier, \emph{Computing Brauer groups via coarse moduli - draft version}, 2018, available at \url{https://webspace.science.uu.nl/~meier007/CoarseBrauer.pdf}.

\bibitem[Mil13]{milneLEC}
J.\,S.~Milne, \emph{Lectures on Etale Cohomology (v2.21)}, 2013, available at
  \url{www.jmilne.org/math/CourseNotes/lec.html}. 

\bibitem[Mum65]{mumford-picard}
D.~Mumford, \emph{Picard groups of moduli problems}, in: \emph{Arithmetical {A}lgebraic  {G}eometry} ({P}roc.\ {C}onf.\ {P}urdue {U}niv., 1963),  pp.~33--81, Harper \& Row, New York,  1965.

\bibitem[Neu99]{neukirch}
J.~Neukirch, \emph{Algebraic number theory} (translated from the 1992 original and with a note by N.~Schappacher, with a foreword by G.~Harder), Grundlehren math.\ Wiss., vol.~322,  Springer-Verlag, Berlin, 1999, \doi{10.1007/978-3-662-03983-0}.

\bibitem[Nil16]{Niles}
A.~Niles, \emph{The {P}icard groups of the stacks {$\mathcal{Y}_0(2)$} and
  {$\mathcal{Y}_0(3)$}}, Funct.\ Approx.\ Comment.\ Math.\ \textbf{55} (2016), no.~1, 105--112, \doi{10.7169/facm/2016.55.1.7}.

\bibitem[Ols12]{olsson-G-torsors}
M.~Olsson, \emph{Integral models for moduli spaces of {$G$}-torsors}, Ann.\
  Inst.\ Fourier (Grenoble) \textbf{62} (2012), no.~4, 1483--1549, \doi{10.5802/aif.2728}

\bibitem[Ols16]{olsson}
\bysame, \emph{Algebraic spaces and stacks}, Amer.\ Math.\ Soc.\ Colloq.\ Publ., vol.~62, Amer.\ Math.\ Soc., Providence, RI, 2016, \doi{10.1090/coll/062}. 

\bibitem[Ols05]{OLSSON200593}
\bysame, \emph{On proper coverings of {A}rtin stacks}, Adv.\ Math.\
  \textbf{198} (2005), no.~1, 93--106, \doi{10.1016/j.aim.2004.08.017}.

\bibitem[Poo17]{poonen-rat-pts}
B.~Poonen, \emph{Rational points on varieties}, Grad.\ Stud.\ Math., vol.~186, Amer.\ Math.\ Soc., Providence, RI, 2017, \doi{10.1090/gsm/186}.

\bibitem[San23]{santens2023brauermanin}
T.~Santens, \emph{The Brauer-Manin obstruction for stacky curves},
 preprint \arXiv{2210.17184} (2023).

\bibitem[Shi19]{Shin}
M.~Shin, \emph{The {B}rauer group of the moduli stack of elliptic curves over
  algebraically closed fields of characteristic 2}, J.~Pure Appl.\ Algebra
  \textbf{223} (2019), no.~5, 1966--1999, \doi{10.1016/j.jpaa.2018.08.010}.

\bibitem[Sil09]{silverman}
J.\,H.~Silverman, \emph{The arithmetic of elliptic curves}, 2nd ed., Grad.\
  Texts in Math., vol.~106, Springer, Dordrecht, 2009, \doi{10.1007/978-0-387-09494-6}.

\bibitem[{Sta}23]{stacks-project}
The Stacks Project Authors, \emph{The Stacks Project},
  \url{https://stacks.math.columbia.edu}, 2023.

  
\bibitem[tVo20]{voorde}
A.~ten Voorde, \emph{Cyclic algebras over local fields arising from elliptic
curves}, Master's thesis, Utrecht University, 2020, available at \url{https://studenttheses.uu.nl/handle/20.500.12932/37070}.

\bibitem[Wei94]{weibel}
C.\,A.~Weibel, \emph{An introduction to homological algebra}, Cambridge Stud.\
   Adv.\ Math., vol.~38, Cambridge Univ.\ Press, Cambridge,
  1994, \doi{10.1017/CBO9781139644136}.

\end{thebibliography}
\end{document}